\documentclass[
a4paper,reqno,10pt]{amsart}%{article}

\usepackage{amssymb,amsmath,amsthm,array}
\usepackage{graphicx}
\usepackage{float}
\usepackage{xcolor}
\usepackage{enumitem}
\usepackage{booktabs}
\usepackage{adjustbox}
\usepackage{mathtools}
\usepackage[top=2.5cm,bottom=2.5cm,left=2.5cm,right=2.5cm]{geometry}
\usepackage{placeins}
\usepackage[authoryear,round]{natbib}

\usepackage[
colorlinks=true,allcolors=black,bookmarks=true,
bookmarksnumbered=true,bookmarkstype=toc,linktocpage=true
]{hyperref}
\hypersetup{
pdftitle={Debiased parametric inference for regime-switching jump diffusions with infinite variation jumps},
pdfauthor={Yuzhong Cheng}
}
\providecommand{\doi}[1]{\href{https://doi.org/#1}{https://doi.org/#1}}

\newtheorem{thm}{Theorem}[section]

\newtheorem{lem}[thm]{Lemma}
\newtheorem{prop}[thm]{Proposition}
\newtheorem{cor}[thm]{Corollary}

\newtheorem{rem}[thm]{Remark}
\newtheorem{ass}[thm]{Assumption}

\numberwithin{equation}{section}
\allowdisplaybreaks

\newcommand{\mcl}{\mathcal{L}}
\newcommand{\cil}{\xrightarrow{\mcl}} % convergence in law
\newcommand{\cip}{\xrightarrow{p}} % convergence in probability
\newcommand{\argmax}{\mathop{\rm argmax}}
\newcommand{\E}{\mathbb{E}}

\newcommand{\Pb}{\mathbb{P}}
\newcommand{\proofstep}[1]{%
	\par\medskip\noindent\emph{#1.}\quad
}
\newcommand{\rev}[1]{#1}
\newenvironment{revision}{}{}

\newenvironment{newrevision}{}{}
\newcommand{\greenrev}[1]{#1}
\newenvironment{greenrevision}{}{}

\newenvironment{currentrevision}{}{}
\newcommand{\darkgreenrev}[1]{#1}
\newenvironment{darkgreenrevision}{}{}

\title[Debiased Parametric Inference for Switching Jump Diffusions]
{{\rev{Debiased parametric inference for regime-switching jump
diffusions with infinite variation jumps}}}

\author{Yuzhong Cheng}
\address{Kyushu University, Institute of Mathematics for Industry, 744 Motooka, Fukuoka, Japan}
\email{cheng.yuzhong.451@m.kyushu-u.ac.jp}
\thanks{Corresponding author: Yuzhong Cheng,
cheng.yuzhong.451@m.kyushu-u.ac.jp.}

\keywords{Regime-switching diffusion, infinite-variation L\'evy jumps,
high-frequency inference, debiased truncated variation, Gaussian
quasi-likelihood, long-span asymptotics.}
\subjclass[2020]{Primary 62M05; Secondary 62F12, 60G51, 60J75.}

\begin{document}
	\begin{abstract}
		\begin{darkgreenrevision}
		We study high-frequency inference for an ergodic regime-switching jump
		diffusion with locally stable infinite-variation jumps. The state and the
		finite-state regime are observed on the same grid. The mesh satisfies
		\(h_n\to0\), while the horizon satisfies \(T_n=nh_n\to\infty\). We estimate
		structural drift and diffusion parameters in the presence of unknown
		regime-specific L\'evy measures. For a common activity index
		\(1<\beta<8/5\), we remove two leading jump biases from each regime's
		truncated variation. We then combine the diffusion estimator with estimated
		signed tail means in an adaptive drift contrast. We prove joint asymptotic
		normality at rates \(\sqrt{T_n}\) and \(\sqrt n\), together with feasible
		studentization. The growing horizon requires uniform control of the
		debiasing remainders and random denominators over increasing regime
		exposures. It also produces an explicit L\'evy-measure term in the drift
		covariance. A fixed-grid median stabilizes the two-step correction without
		changing its first-order limit. We study finite-sample behavior by simulation
		and apply the method to 127.5 hours of cattle sensor data. In the application,
		jaw-activity labels are recorded by a device separate from the sensor that
		measures the state process.
		\end{darkgreenrevision}
	\end{abstract}

	\maketitle
	\setlength{\emergencystretch}{1em}

	%%%%%%%%%%%%%%%%%%%%%%%%%

	%%%%%%%%%%%%%%%%%%%%%%%%%
	\section{Introduction}

	\begin{revision}
	Diffusions with Markovian switching describe systems whose local dynamics
	vary with a finite-state Markov chain; see
	\citet{mao2006stochastic,yin2009hybrid}. Their jump extensions allow both
	the continuous dynamics and the distribution of abrupt shocks to depend on
	the prevailing regime. Examples include the regime-switching L\'evy models
	of \citet{ChevallierGoutte2017}, the switching spike models of
	\citet{janczura2010empirical}, and the multistate Langevin diffusion of
	\citet{McClintockLander2024}.

	\begin{newrevision}
	We study a joint process \((X,Z)\) satisfying
	\begin{equation}
	\label{eq:SDE_model}
	dX_t
	=b(X_t,Z_t,\alpha)\,dt
	+\greenrev{a_{Z_t}(\gamma)}\,dW_t
	+\sum_{i=1}^m\int_{\mathbb R}
	z\,\mathbf 1_{\{Z_{t-}=i\}}\,\widetilde N_i(dt,dz),
\end{equation}
	where
	\begin{itemize}
		\item \(Z=(Z_t)_{t\ge0}\) is an exogenous continuous-time Markov chain on
		\(S=\{1,\ldots,m\}\) with generator \(Q=(q_{ik})_{i,k\in S}\), where
		\(q_{ik}\ge0\) for \(i\ne k\) and
		\(q_{ii}=-\sum_{k\ne i}q_{ik}\);
		\item \(W\) is a standard Wiener process and, for each \(i\in S\),
		\(N_i\) is a Poisson random measure with compensator
		\(\nu_i(dz)\,dt\), where the L\'evy measure \(\nu_i\) has density
		\(s_i\), and
		\(\widetilde N_i(dt,dz):=N_i(dt,dz)-\nu_i(dz)\,dt\);
		\item \(W,Z,X_0,N_1,\ldots,N_m\) are mutually independent.
	\end{itemize}

	\begin{darkgreenrevision}
	Let
	\(\Theta_\alpha=\overline{\mathcal O_\alpha}
	\subset\mathbb R^{p_\alpha}\),
	\(\Theta_\gamma=\overline{\mathcal O_\gamma}
	\subset\mathbb R^{p_\gamma}\), and
	\(\Theta=\Theta_\alpha\times\Theta_\gamma\). Here
	\(\mathcal O_\alpha\) and \(\mathcal O_\gamma\) are bounded convex domains
	with Lipschitz boundaries. The true parameter is
	\(\theta^\star=(\alpha^\star,\gamma^\star)\in\operatorname{int}\Theta\).
	We write \(\nu_i^\star(dz)=s_i^\star(z)\,dz\) for the true L\'evy measure and
	density in regime \(i\). The vector \(\theta\) indexes the full coefficient
	family and may contain separate blocks for different regimes. Thus it does not
	require the drift or diffusion coefficients to have the same numerical value
	in every regime. We assume that \(Q\) is irreducible. Since \(S\) is finite,
	\(Z\) then has a unique stationary distribution
	(cf.\ \cite[Definition~A.7]{yin2009hybrid}). We treat \(Q\) as a nuisance
	parameter and omit it from the notation for the law below.
	\end{darkgreenrevision}
	\end{newrevision}

	\begin{newrevision}
	The pair \((X,Z)\) is observed on the grid
	\(\{(X_{t_j},Z_{t_j})\}_{j=0}^n\), with \(t_j=jh_n\), where
	\[
	h_n\to0,\qquad T_n=nh_n\to\infty,\qquad nh_n^2\to0.
	\]
	\end{newrevision}
	\begin{darkgreenrevision}
	Only the grid values of \(Z\) are observed. The data therefore give the two
	endpoint labels, but not the full regime path on \([t_{j-1},t_j]\). This
	setting is useful when an external sensor records the regime together with
	\(X\). We do not consider latent regimes.

	Ecological monitoring provides a direct example. Motion and behavior are
	often recorded by different sensors. In the Precision Beef data set of
	\citet{Pavlovic2021PrecisionBeef}, a collar records three-axis acceleration at
	about 10 Hz. A RumiWatch halter separately records jaw-pressure labels at
	10 Hz. A collar coordinate can serve as \(X\), while the halter label serves
	as \(Z\). This separate recording does not make \(X\) and \(Z\) statistically
	independent: the law of \(X\) still depends on \(Z\) through
	\eqref{eq:SDE_model}. It means only that \(Z\) is observed directly rather
	than inferred from \(X\).
	\end{darkgreenrevision}
	
	\begin{newrevision}
	We estimate the structural parameter
	\(\theta=(\alpha,\gamma)\), treating the regime-wise L\'evy measures as
	infinite-dimensional nuisance quantities.
	\end{newrevision}

		\begin{darkgreenrevision}
		We make two model choices. First, the diffusion coefficient may change with
		the regime, but not with \(x\). After debiasing, truncated variation can then
		estimate \(a_i^2(\gamma)\) directly. With a general coefficient
		\(a(x,i,\gamma)\), the same statistic would average the spot variance over
		the observed values of \(X\). Estimating \(\gamma\) would require a local or
		weighted volatility estimator and new uniform error bounds. The simpler
		coefficient lets us focus on growing-horizon debiasing and the adaptive drift
		estimator.
		
		Second, we assume that the L\'evy density in each regime is locally stable.
		All regimes have the same activity index \(\beta\in(1,8/5)\):
		for every \(i\in S\),
		\[
		s_i(z)=c_iq_i(z)|z|^{-1-\beta},\qquad 0<|z|\le1,
		\]
		where \(c_i>0\), each
		\(q_i:[-1,1]\setminus\{0\}\to[0,\infty)\) is bounded, and
		\(q_i(z)\to1\) as \(z\to0\). The condition \(\beta>1\) means that the jump
		component has infinite variation. A common \(\beta\) lets us use the same
		cutoff and the same two-step correction in every regime. The scale \(c_i\)
		and the function \(q_i\) may still vary across regimes.

		The upper bound \(\beta<8/5\) makes the cutoff interval in
		Assumption~\ref{ass:iv_sampling} nonempty. This bound comes from the
		two-step correction, not from regime switching.
		Remark~\ref{rem:iv_activity_boundary} gives the calculation and explains
		each cutoff condition.
		
		The limit \(q_i(z)\to1\) makes the leading stable term symmetric. We allow
		the next term to be asymmetric. Namely, for some
		\(\ell_i^+,\ell_i^-\in\mathbb R\), \(\rho_q>0\), and \(C>0\),
		\[
		|q_i(z)-1-\ell_i^+z|
		\le C|z|^{\beta+\rho_q}\quad(0<z\le1),
		\qquad
		|q_i(z)-1-\ell_i^-z|
		\le C|z|^{\beta+\rho_q}\quad(-1\le z<0).
		\]
		The two slopes allow different behavior on the positive and negative sides.
		They also give \(q_i(z)-q_i(-z)=O(z)\), so the small-jump part of the signed
		principal-value mean is finite. After removing the stable and linear terms,
		the remaining density is \(O(|z|^{-1+\rho_q})\) near zero and therefore has
		finite mass. This form supports the fixed-regime debiasing expansion used
		below. Boundedness of \(q_i\) also gives
		\[
		\sup_{i\in S}\nu_i(\{|z|>r\})\le Cr^{-\beta},
		\qquad 0<r\le1.
		\]
		We call these conditions the standing infinite-variation specification.
	\end{darkgreenrevision}

	\begin{darkgreenrevision}
	Two asymptotic scales organize the problem. The mesh \(h_n\to0\) separates
	continuous and jump increments and gives \(\sqrt n\)-rate diffusion
	inference. The horizon \(T_n\to\infty\) supplies regime exposure and gives
	\(\sqrt{T_n}\)-rate inference for the drift and tail means. A growing horizon
	is essential. Over a fixed interval, the drift information and the expected
	number of detected jumps away from zero remain bounded. The drift parameter
	and the tail means are then not consistently estimable by our method.

	The regimes are observed only at the sampling times. The method therefore
	does not solve a hidden-state filtering problem. However, equal endpoint
	labels do not rule out a switch and return within an interval. Our analysis
	controls the resulting difference between endpoint counts and occupation
	times.

	For infinite-variation semimartingales, \citet{JacodTodorov2014} constructed
	efficient fixed-horizon volatility estimators from local empirical
	characteristic functions. The closest methodological reference is
	\citet{BonieceFigueroaLopezHan2024}. They developed a two-step algebraic
	correction of truncated realized variation for integrated volatility on a
	fixed interval. Their expansion removes the leading small-jump terms when
	\(1<\beta<8/5\).

	We use their correction, but our target and asymptotic setting differ. Their
	target is fixed-interval integrated volatility. Our regime-specific variances
	identify a structural diffusion parameter and then enter an adaptive drift
	estimator. Drift inference requires \(T_n\to\infty\), uses the mixed
	normalization \(\operatorname{diag}(\sqrt{T_n}I_{p_\alpha},
	\sqrt n I_{p_\gamma})\), and depends on estimated signed tail means. We must
	therefore make the fixed-horizon expansion uniform over growing random regime
	exposures. Both rational denominators must remain stable in every regime, and
	the accumulated remainder must satisfy \(\sqrt n\rho_n\to0\), or equivalently
	\(\sqrt{T_n}h_n^{\delta_{\rm iv}^{\sharp}}\to0\). These requirements do not
	arise in the fixed-interval problem. The cutoff conflict that leads to
	\(\beta<8/5\) is explained after Assumption~\ref{ass:iv_sampling}.
	\end{darkgreenrevision}

	\begin{darkgreenrevision}
	Other infinite-variation methods address separate parts of our problem.
	\citet{AmorinoGloter2020Unbiased} derived an explicit
	jump-bias correction for truncated integrated-volatility estimators.
	\citet{Mies2021} used combined high-frequency and increasing-horizon
	asymptotics to estimate state-dependent jump activity and drift, allowing
	symmetric jumps of infinite variation. More recently,
	\citet{LiuLiu2024Spot} combined pre-averaging with local empirical
	characteristic functions to estimate spot volatility under infinite-variation
	jumps and dependent microstructure noise. These models do not contain an
	observed finite-state switching mechanism. Table~\ref{tab:literature_comparison}
	compares their targets and sampling regimes with ours.

	\begin{table}[H]
		\centering
		\caption{Selected high-frequency inference under infinite-variation jumps.}
		\label{tab:literature_comparison}
		\begingroup
		\scriptsize
		\setlength{\tabcolsep}{3pt}
		\renewcommand{\arraystretch}{1.08}
		\begin{tabular}{@{}>{\raggedright\arraybackslash}p{0.18\textwidth}
		>{\raggedright\arraybackslash}p{0.09\textwidth}
		>{\raggedright\arraybackslash}p{0.12\textwidth}
		>{\raggedright\arraybackslash}p{0.25\textwidth}
		>{\raggedright\arraybackslash}p{0.25\textwidth}@{}}
		\toprule
		Reference & Switching & Horizon & Target and method & Difference from this paper \\
		\midrule
		\citet{JacodTodorov2014}
		& No & Fixed
		& Integrated volatility; local empirical characteristic functions
		& No structural drift/diffusion parameters or switching. \\
		\citet{AmorinoGloter2020Unbiased}
		& No & Fixed
		& Integrated volatility; analytic correction of truncated variation
		& No joint drift/diffusion inference or switching. \\
		\citet{Mies2021}
		& No & Increasing
		& State-dependent jump activity and drift; smooth truncation
		& Continuous-state nonparametrics; no switching or diffusion-parameter inference. \\
		\citet{BonieceFigueroaLopezHan2024}
		& No & Fixed
		& Integrated volatility; two-step debiased truncated variation
		& Source of the debiasing step; no regime exposure or adaptive drift. \\
		\citet{LiuLiu2024Spot}
		& No & Local/fixed
		& Spot volatility under dependent noise; pre-averaging and empirical characteristic functions
		& Noise-robust local volatility, not joint structural inference. \\
		Present paper
		& Yes (observed) & Increasing
		& Structural drift/diffusion; debiasing and regime-wise tail correction
		& Joint mixed-rate inference with random regime exposures. \\
		\bottomrule
		\end{tabular}
		\endgroup
		\parbox{0.96\textwidth}{\scriptsize\raggedright\emph{Note.}
		Switching means that the coefficients and jump laws are driven by an
		observed finite-state regime process; an increasing horizon means
		\(T_n=nh_n\to\infty\).}
	\end{table}

	\end{darkgreenrevision}

	\begin{darkgreenrevision}
	To describe the method, choose
	\(\sqrt{h_n}\ll u_n\to0\) and define
	\end{darkgreenrevision}
	\begin{equation}
		\label{eq:truncated_jump_features}
		\kappa_{i,n}:=\int_{|z|>u_n}z\,\nu_i^\star(dz),
		\qquad
		q_{i,n}:=\int_{|z|\le u_n}z^2\,\nu_i^\star(dz).
	\end{equation}
	\begin{darkgreenrevision}
	Given a left endpoint in regime \(i\), the leading truncated moments have
	the schematic form
	\end{darkgreenrevision}
	\[
		\frac1{h_n}\E_{j-1}\!\left[
		\Delta_jX\mathbf1_{\{|\Delta_jX|\le u_n\}}\right]
		\simeq b_{j-1}^\star-\kappa_{i,n},
		\qquad
		\frac1{h_n}\E_{j-1}\!\left[
		(\Delta_jX)^2\mathbf1_{\{|\Delta_jX|\le u_n\}}\right]
		\simeq (a_{j-1}^\star)^2+q_{i,n}.
	\]
	\begin{darkgreenrevision}
	The cutoff removes observed large jumps, but two jump effects remain. The
	drift contains the signed large-jump compensator, and the diffusion statistic
	contains the variation of retained small jumps. Under infinite variation,
	ordinary truncated variation has two non-negligible bias terms, of orders
	\(v^{2-\beta}\) and \(h_nv^{-\beta}\). We remove both terms regime by regime.
	We estimate \(\kappa_{i,n}\) directly from detected increments and then
	estimate the drift using the debiased diffusion parameter. This order avoids
	estimating the full L\'evy density first.
	\end{darkgreenrevision}

	\begin{darkgreenrevision}
	Our main result is a feasible joint central limit theorem for the adaptive
	drift and debiased diffusion estimators. Their rates are \(\sqrt{T_n}\) and
	\(\sqrt n\). The required remainder condition is
	\[
		\sqrt n\,\rho_n\to0
		\qquad\Longleftrightarrow\qquad
		\sqrt{T_n}\,h_n^{\delta_{\rm iv}^{\sharp}}\to0.
	\]
	This condition shows the extra restriction caused by the growing horizon.
	Tail-mean estimation adds an explicit L\'evy-measure term to the drift
	covariance. The first-order diffusion covariance remains the Gaussian
	quasi-score covariance. We also prove regime-specific denominator stability
	for both correction steps and obtain a studentized joint limit.

	The numerical study examines the finite-sample effect of the mesh--horizon
	restrictions and compares a fixed-grid median with a single correction pair.
	The empirical illustration applies the method and its diagnostics to a long,
	finely sampled animal-sensor record. A separate jaw-pressure device supplies
	the regime labels.

	The remainder of the paper is organized as follows.
	Section~\ref{sec:setup} gives the notation and assumptions, and
	Section~\ref{sec:estimation} presents the estimators and their asymptotic
	properties.
	Section~\ref{sec:proof_strategy} gives the principal proofs,
	Section~\ref{sec:sim} reports the infinite-variation numerical study, and
	Section~\ref{sec:empirical} gives the empirical illustration. The
	technical supplement gives the detailed infinite-variation arguments.
	\end{darkgreenrevision}
	\end{revision}

	%%%%%%%%%%%%%%%%%%%%%%%%%
	\section{Notation and assumptions}
	\label{sec:setup}

	Under the assumptions below, the compensated integral in $\eqref{eq:SDE_model}$ is the
	\(L^2\)-limit of the integrals restricted to
	\(\{|z|>\varepsilon\}\) as \(\varepsilon\downarrow0\), using symmetric
	absolute cutoffs. Every principal-value integral of
	\(z\,\nu_i(dz)\) uses the same cutoff convention.

	\subsection{Notation}
	Write \(\Delta_jX=X_{t_j}-X_{t_{j-1}}\). Throughout, we abbreviate
	\(Z_{j-1}:=Z_{t_{j-1}}\) and \(Z_j:=Z_{t_j}\).

	\begin{darkgreenrevision}
	We write \(\mathbb P_{\theta,s}\) and \(\mathbb E_{\theta,s}\) for the law of
	\((X,Z)\) and its expectation when the parameter is \(\theta\) and the L\'evy
	densities are \(s=\{s_i\}_{i\in S}\). At the true values, we abbreviate
	$\mathbb P:=\mathbb P_{\theta^\star,s^\star}$,
	$\mathbb E:=\mathbb E_{\theta^\star,s^\star}$,
	$\E_{j-1}[\cdot]:=\mathbb E[\cdot\mid\mathcal F_{t_{j-1}}]$.
	For an integrable \(Y\mathbf1_B\), we use the notation
		$\E_{j-1}[Y;B]
		:=\E_{j-1}[Y\mathbf1_B]$,
		$\Pb_{j-1}(B):=\E_{j-1}[\mathbf1_B]$.
	The symbols \(\cip\) and \(\cil\) denote convergence in probability and in
	distribution under \(\mathbb P\).
	\end{darkgreenrevision}

	{
	We use \(|\cdot|\) for the Euclidean norm, \(\mathbf1_A\) for an
	indicator, \(v^{\otimes2}=vv^\top\). For nonnegative quantities
	\(x_n,y_n\), the notation \(x_n\lesssim y_n\) means
	\(x_n\le Cy_n\) for a deterministic constant \(C\) independent of \(n\);
	\(x_n\asymp y_n\) means that both \(x_n\lesssim y_n\) and
	\(y_n\lesssim x_n\) hold. Whenever a bound is stated uniformly over
	additional indices, the implied constant is independent of those indices
	as well.

	Throughout, \(R_{j-1}\) denotes a generic predictable polynomial
	envelope: for deterministic constants \(C,c>0\), independent of
	\(n,j\), the regime, the \darkgreenrev{cutoff}, and the parameter,
		$R_{j-1}=C\{1+|X_{t_{j-1}}|^c\}$.
	The constants and exponent may change from line to line. Thus
	a bound \(r_nR_{j-1}\) always has a deterministic rate \(r_n\), uniformly
	in the indices displayed in the statement.
	\par}

	\subsection{Assumptions}
	\label{sec:assumptions}

	%---------------------------------------------------------------------------
	\subsubsection*{Regularity of the coefficients}

	\begin{ass}
		\label{ass:smooth}
		\begin{greenrevision}
		For each \(i\in S\):
		\begin{enumerate}[label=(\roman*), leftmargin=1.8em]
			\item \(b(\cdot,i,\cdot)\) is continuous on
			\(\mathbb R\times\Theta_\alpha\), and \(a_i(\cdot)\) is continuous on
			\(\Theta_\gamma\). There is a constant \(C>0\) such that, for all
			\(x,y\in\mathbb R\) and \(\alpha\in\Theta_\alpha\),
			\begin{align}
				|b(x,i,\alpha)-b(y,i,\alpha)|^2
				&\le C|x-y|^2, \label{eq:Lip_x}\\
				|b(x,i,\alpha)|^2
				&\le C\,(1+|x|^2). \label{eq:lin_growth}
			\end{align}
			\item (Uniform ellipticity) there exists \(c_0>0\) such that
			\(\inf_{i\in S,\gamma\in\Theta_\gamma}a_i(\gamma)\ge c_0\).
			\item The map \(x\mapsto b(x,i,\alpha)\) is twice continuously
			differentiable, \(\alpha\mapsto b(x,i,\alpha)\) and
			\(\gamma\mapsto a_i(\gamma)\) are three times continuously
			differentiable, and, for some \(C>0\),
			\begin{align}
				\label{eq:deriv_poly_growth}
				\max_{i\in S}\sup_{(x,\alpha)\in\mathbb R\times\Theta_\alpha}
				\frac{\displaystyle\sum_{k=0}^3\sum_{l=0}^2
				\big|\partial_\alpha^k\partial_x^l b(x,i,\alpha)\big|}
				{1+|x|^{C}}
				&<\infty,\\
				\max_{i\in S}\sup_{\gamma\in\Theta_\gamma}
				\left\{a_i(\gamma)^{-1}
				+\sum_{k=0}^3|\partial_\gamma^ka_i(\gamma)|\right\}
				&<\infty.
			\end{align}
		\end{enumerate}
		\end{greenrevision}
	\end{ass}

	%---------------------------------------------------------------------------
	\subsubsection*{L\'evy-measure conditions}

	\begin{ass}
		\label{ass:levy}
		For each \(i\in S\) and every \(q>0\),
			$\int_{|z|>1}|z|^q\,s_i^\star(z)\,dz<\infty$.
	\end{ass}

	\begin{ass}[\darkgreenrev{IV sampling and cutoff}]
		\label{ass:iv_sampling}
		For fixed \(\zeta_1,\zeta_2>1\), there exist \(c_{\rm IV}>0\) and
		\(\omega\) such that
		\[
			u_n=c_{\rm IV}h_n^\omega,\qquad
			\max\left\{\frac1{4-\beta},\frac2{4+\beta}\right\}
			<\omega<\frac4{8+\beta}.
			\tag{IV}
			\label{eq:iv_threshold}
		\]
		Put
		\[
			\epsilon_0:=\frac{3\beta-2}{4},\qquad
			\rho_n:=h_n^2u_n^{-\beta-2}
			+h_n^{1/2}u_n^{1-\beta/2}
			+h_nu_n^{2-2\beta}
			+u_n^{2-\epsilon_0}+h_n.
		\]
		The sampling scheme satisfies the direct remainder condition
		\[
			\sqrt n\,\rho_n\to0.
			\tag{IV-horizon}
			\label{eq:iv_span}
		\]
	\end{ass}

	\begin{rem}
		\label{rem:iv_span_interpretation}
		Under \eqref{eq:iv_threshold}, define
		\[
			\delta_{\rm iv}^{\sharp}:=
			\min\left\{
			\frac32-\omega(\beta+2),
			\ \omega\left(1-\frac\beta2\right),
			\ \frac12-2\omega(\beta-1),
			\ -\frac12+\frac{\omega(10-3\beta)}4,
			\ \frac12
			\right\}>0.
		\]
		Since \(u_n=c_{\rm IV}h_n^\omega\), condition
		\eqref{eq:iv_span} is equivalent to
		\(\sqrt{T_n}h_n^{\delta_{\rm iv}^{\sharp}}\to0\).
		For the polynomial mesh \(h_n=n^{-\kappa}\),
		\(T_n=n^{1-\kappa}\), condition \eqref{eq:iv_span} is equivalent to
			$\kappa>\{1+2\delta_{\rm iv}^{\sharp}\}^{-1}$.
	\end{rem}

	With the symmetric-cutoff convention specified for \eqref{eq:SDE_model},
		$m_i^\star
		=\lim_{\varepsilon\downarrow0}
		\int_{|z|>\varepsilon}z\,\nu_i^\star(dz)$.
	The signed tail mean in \eqref{eq:truncated_jump_features} satisfies
	\(\kappa_{i,n}=m_i^\star+O(u_n^{2-\beta})\).
	\begin{rem}[Activity bound]
		\label{rem:iv_activity_boundary}
		\begin{darkgreenrevision}
		The bound \(\beta<8/5\) comes from the cutoffs used in the two-step
		estimator, not from regime switching. Let
		\(u_n=c_{\rm IV}h_n^\omega\). Three key conditions are
		\[
			\frac{u_n^{4-\beta}}{h_n}\to0,
			\qquad
			\frac{u_n^{(4+\beta)/2}}{h_n}\to0,
			\qquad
			h_n^{4/(8+\beta)}=o(u_n).
		\]
		The first condition makes the small-jump fluctuation negligible at the
		\(\sqrt n\) scale. The second controls increments between two cutoff
		levels. Their fluctuation must be smaller than the second bias, whose order
		is \(h_nu_n^{-\beta}\). This also keeps the second denominator away from
		zero. The third condition controls expansions at nearby cutoffs and the
		nonlinear errors from both correction steps. It matches the
		truncated-difference condition in
		\citet{BonieceFigueroaLopezHan2024}. Written in terms of \(\omega\), the
		three conditions give
		\[
			\max\left\{\frac1{4-\beta},\frac2{4+\beta}\right\}
			<\omega<\frac4{8+\beta}.
		\]
		The first lower bound is the one that meets the upper bound:
		\[
			\frac1{4-\beta}<\frac4{8+\beta}
			\quad\Longleftrightarrow\quad \beta<\frac85.
		\]
		At \(\beta=8/5\), the two strict bounds meet. Above \(8/5\), the lower
		bound exceeds the upper bound. No cutoff of the form
		\(u_n=c_{\rm IV}h_n^\omega\) can then satisfy both. Our proof cannot control
		the denominator and the \(\sqrt n\)-scale remainder at the same time. The
		estimator can still be computed, but our Gaussian limit no longer follows.
		Higher activity requires another correction or another volatility
		estimator.
		\end{darkgreenrevision}
	\end{rem}

	%---------------------------------------------------------------------------

	\subsubsection*{Well-posedness}

	\begin{darkgreenrevision}
	By Assumptions~\ref{ass:smooth} and~\ref{ass:levy}, the switching rates are
	bounded because the state space is finite. The hypotheses of
	\citet[Theorem~2.1]{xi2017feller} then hold. Thus
	\eqref{eq:SDE_model}, at the true parameter, has a unique non-explosive
	strong solution.
	\end{darkgreenrevision}
	%---------------------------------------------------------------------------
	\subsubsection*{Ergodicity}

	Let \((P_t)_{t\ge0}\) denote the transition semigroup of \((X,Z)\). For a
	measurable weight \(g:\mathbb R\times S\to[1,\infty)\) and a signed measure
	\(m\) on \(\mathcal B(\mathbb R\times S)\), define the \(g\)-norm
	\[
	\|m\|_g:=\sup\Big\{\,|m(f)|\;:\;f\ \text{measurable},\ |f|\le g\,\Big\}.
	\]

	\begin{ass}[Exponential ergodicity in every polynomial weight]
		\label{ass:ergodic}
		\begin{enumerate}[label=(\roman*), leftmargin=1.8em]
			\item There exists a probability measure \(\mu\) on
			\(\mathbb R\times S\). For every \(q>0\), there are constants
			\(C_q,\varkappa_q>0\) such that
			\[
			\sup_{t\ge0}e^{\varkappa_q t}\,
			\big\|P_t((x,i),\cdot)-\mu(\cdot)\big\|_{g_q}
			\le C_q\,g_q(x,i),
			\qquad (x,i)\in\mathbb R\times S,
			\]
			where \(g_q(x,i):=1+|x|^{q}\).
			\item The initial value satisfies \(\E_{\theta^\star}\big[|X_0|^{q}\big]<\infty\)
			for every \(q>0\).
		\end{enumerate}
	\end{ass}

\begin{rem}
	The weighted exponential ergodicity in
	Assumption~\ref{ass:ergodic} can be verified through the
	Foster--Lyapunov conditions in
	\citet[Assumption~5.2 and Theorem~6.3]{Xi2009asynptotic}, together with
	the corresponding irreducibility and regularity hypotheses.
\end{rem}

	\noindent
	{
	\begin{darkgreenrevision}
	Assumption~\ref{ass:ergodic} gives the invariant law \(\mu\), exponential
	mixing, and the ergodic laws used below. In particular,
	\citet[Part~III]{meyn2012markov} and \citet{bhattacharya1982functional}
	imply that, for every measurable \(f\) with
	\(|f(x,i)|\lesssim 1+|x|^{C}\),
	\begin{equation}
		\label{eq:ergodic_thm}
		\frac1T\int_0^T f(X_t,Z_t)\,dt\;\cip\;\int_{\mathbb R\times S}f\,d\mu,
		\qquad T\to\infty.
	\end{equation}
	For \(i\in S\), write \(\pi_i:=\mu(\mathbb R\times\{i\})\). Irreducibility
	of \(Z\) gives \(\pi_i>0\) for every \(i\).
	\end{darkgreenrevision}
	\par}

	%---------------------------------------------------------------------------

	\begin{ass}
		\label{ass:ident}
		\begin{greenrevision}
		If, for some \(\theta_1=(\alpha_1,\gamma_1)\) and
		\(\theta_2=(\alpha_2,\gamma_2)\),
		$b(x,i,\alpha_1)=b(x,i,\alpha_2)$
		for \(\mu\)-almost every \((x,i)\) and
		\(a_i(\gamma_1)^2=a_i(\gamma_2)^2\) for every \(i\in S\),
		then \(\theta_1=\theta_2\).
		\end{greenrevision}
	\end{ass}

	Define the (block-diagonal) quasi-information matrix
	\(I(\theta^\star):=\operatorname{diag}\big(I_\alpha(\theta^\star),I_\gamma(\theta^\star)\big)\),
	where
	\[
	I_\alpha(\theta^\star):=\int_{\mathbb R\times S}
	\greenrev{\frac{\{\partial_\alpha b(x,i,\alpha^\star)\}^{\otimes2}}{a_i(\gamma^\star)^2}}\,\mu(dx,di),
	\qquad
	I_\gamma(\theta^\star):=\frac12\int_{\mathbb R\times S}
	\greenrev{\big\{\partial_\gamma\log a_i^2(\gamma^\star)\big\}^{\otimes2}}\,\mu(dx,di).
	\]

	{
	For the effect of tail-mean estimation on the drift limit, define the
	\(p_\alpha\times m\) matrix
	\[
		J(\theta^\star)
		:=\bigl(J_1(\theta^\star),\ldots,J_m(\theta^\star)\bigr),
		\qquad
		J_i(\theta^\star)
		:=\int_{\mathbb R\times\{i\}}
		\frac{\partial_\alpha b(x,i,\alpha^\star)}
		{\greenrev{a_i(\gamma^\star)^2}}\,\mu(dx,di).
	\]
	Define also
	\[
		\Sigma_\kappa
		:=\operatorname{diag}\left(
		\frac1{\pi_i}\int_{\mathbb R}z^2\,\nu_i^\star(dz):
		i\in S\right).
	\]
	Set
	\[
		\Omega(\theta^\star)
		:=\operatorname{diag}\left(
		I_\alpha(\theta^\star)+
		J(\theta^\star)\Sigma_\kappa J(\theta^\star)^\top,
		I_\gamma(\theta^\star)\right),
	\]
	so that the sandwich covariance
	\[
		V_\alpha(\theta^\star)
		:=I_\alpha(\theta^\star)^{-1}
		\left\{I_\alpha(\theta^\star)
		+J(\theta^\star)\Sigma_\kappa J(\theta^\star)^\top\right\}
		I_\alpha(\theta^\star)^{-1},
	\qquad
		V(\theta^\star)
		:=\operatorname{diag}\left(
		V_\alpha(\theta^\star),I_\gamma(\theta^\star)^{-1}\right).
	\]
	}

	\begin{ass}
		\label{ass:info_nondeg}
		The matrices \(I_\alpha(\theta^\star)\) and \(I_\gamma(\theta^\star)\) are
		positive definite.
	\end{ass}

	\begin{darkgreenrevision}
	\subsubsection*{Assumption roadmap}
	The numbered assumptions separate standard conditions from the rate
	restriction caused by the growing horizon. The standing specification in the
	Introduction is comparable to the
	stable-like models used by \citet{JacodTodorov2014},
	\citet{AmorinoGloter2020Unbiased}, \citet{Mies2021}, and
	\citet{BonieceFigueroaLopezHan2024}. Assumption~\ref{ass:iv_sampling}
	uses their two-step cutoff bounds and also controls the error as
	\(T_n\to\infty\). Table~\ref{tab:assumption_roadmap} summarizes the role of
	each condition.

	\begin{table}[H]
		\centering
		\caption{Source and use of the assumptions.}
		\label{tab:assumption_roadmap}
		\small
		\setlength{\tabcolsep}{4pt}
		\renewcommand{\arraystretch}{1.08}
		\begin{tabular}{@{}>{\raggedright\arraybackslash}p{0.28\textwidth}
		>{\raggedright\arraybackslash}p{0.17\textwidth}
		>{\raggedright\arraybackslash}p{0.46\textwidth}@{}}
		\toprule
		Condition & Type & Main use \\
		\midrule
		Standing IV specification
		& Model specification
		& Stable-like bias expansion and constant variance within each regime. \\
		\ref{ass:smooth} and \ref{ass:levy}
		& Classical
		& Well-posedness, moment bounds, and conditional expansions. \\
		\ref{ass:ergodic}
		& Classical
		& Regime-exposure limits and increasing-horizon averages. \\
		\ref{ass:ident} and \ref{ass:info_nondeg}
		& Classical
		& Identification, invertible information, and studentization. \\
		\ref{ass:iv_sampling}
		& New rate condition
		& Control of the two denominators and the \(\sqrt n\)-scale error as regime
		exposures grow. \\
		\bottomrule
		\end{tabular}
	\end{table}
	\end{darkgreenrevision}

	%%%%%%%%%%%%%%%%%%%%%%%%
	%%%%%%%%%%%%%%%%%%%%%%%%

	\section{Estimation and main results}
	\label{sec:estimation}
	\label{sec:main_results}

	%===========================================================================
	\subsection{\greenrev{Exposure and tail-mean estimation}}
	\label{sec:construction}

	\begin{greenrevision}
	Write \(b_{j-1}(\alpha):=b(X_{t_{j-1}},Z_{j-1},\alpha)\).
	For the diffusion coefficient, write
	\(a_{j-1}(\gamma):=a_{Z_{j-1}}(\gamma)\), and set
	\(I_{n,j}:=\mathbf 1_{\{|\Delta_jX|\le u_n\}}\).
	\end{greenrevision}
	For \(i\in S\) and \(j\in\{1,\dots,n\}\), define
	$A_{i,j}:=\mathbf 1_{\{Z_{j-1}=i,\;Z_j=i\}}$.

	\begin{darkgreenrevision}
	The empirical exposure time in regime \(i\) is
	\(T_{i,n}:=h_n\sum_{j=1}^n A_{i,j}\). Its empirical occupation proportion is
	\(\widehat\pi_{i,n}:=T_{i,n}/T_n\). The ergodic theorem
	\eqref{eq:ergodic_thm} and the endpoint--occupation comparison in
	Lemma~\ref{lem:endpoint_exposure_main} give
	\end{darkgreenrevision}
	\begin{equation}
		\label{eq:Tin_convergence}
		\frac{T_{i,n}}{T_n}\;\cip\;\pi_i,\qquad i\in S.
	\end{equation}

	{
	The feasible signed tail-mean estimator is
	\begin{equation}
		\label{eq:kappa_estimator}
		\widehat\kappa_{i,n}
		:=
		\frac1{T_{i,n}}\sum_{j=1}^n
		A_{i,j}\Delta_jX\,\mathbf 1_{\{|\Delta_jX|>u_n\}},
		\qquad i\in S.
	\end{equation}
	\begin{darkgreenrevision}
	On \(\{T_{i,n}=0\}\), set \(\widehat\kappa_{i,n}=0\). This estimator targets
	the cutoff-dependent quantity \(\kappa_{i,n}\), which is the large-jump
	compensator in the retained small increments.
	\end{darkgreenrevision}
	\par}

	\subsection{\rev{Debiased variance and adaptive drift}}

		\begin{darkgreenrevision}
		Under the standing specification in the Introduction and
		Assumption~\ref{ass:iv_sampling}, ordinary truncated quadratic variation has
		two non-negligible small-jump biases. We therefore use the following
		regime-specific version of the debiased truncated realized variation of
		\citet{BonieceFigueroaLopezHan2024}. Put
		\end{darkgreenrevision}
		\[
			\overline\Delta_{j,n}X
			:=\Delta_jX+h_n\widehat\kappa_{Z_{j-1},n}
		\]
		and, for \(v>0\),
		\begin{equation}
			\label{eq:iv_C_raw}
			\widehat C_{i,n}(v)
			:=\frac1{T_{i,n}}\sum_{j=1}^n
			A_{i,j}(\overline\Delta_{j,n}X)^2
			\mathbf 1_{\{|\overline\Delta_{j,n}X|\le v\}}.
		\end{equation}
		For a function \(F:(0,\infty)\to\mathbb R\) and \(\zeta>1\), define
		\begin{equation}
			\label{eq:iv_debias_operator}
			(\mathcal D_{\zeta,n}F)(v)
			:=F(v)-
			\frac{\{F(\zeta v)-F(v)\}^2}
			{F(\zeta^2v)-2F(\zeta v)+F(v)}
			\mathbf 1_{\{
			|F(\zeta^2v)-2F(\zeta v)+F(v)|\ge h_n\}}.
		\end{equation}
		 \darkgreenrev{The two-step estimate of the regime variance is}
		\begin{equation}
			\label{eq:iv_C_debiased}
			\widehat C_{i,n}^{(1)}
			:=\mathcal D_{\zeta_1,n}\widehat C_{i,n},
			\qquad
			\widehat C_{i,n}^{(2)}
			:=\mathcal D_{\zeta_2,n}\widehat C_{i,n}^{(1)},
			\qquad
			\widehat a_{i,n}^{\,2,\rm db}
			:=\widehat C_{i,n}^{(2)}(u_n).
		\end{equation}

		\begin{currentrevision}
		\begin{rem}[Fixed-grid stabilization]
			\label{rem:fixed_grid_median}
			Write
			\(\widehat a_{i,n}^{2,\rm db}(\zeta_1,\zeta_2)\)
			for \eqref{eq:iv_C_debiased} computed with the pair
			\((\zeta_1,\zeta_2)\). Let \(\mathcal G\) be a fixed finite set of
			pairs in \((1,\infty)^2\), with odd cardinality, and define
			\[
				\widehat a_{i,n}^{2,\rm med}
				:=\operatorname{med}_{(\zeta_1,\zeta_2)\in\mathcal G}
				\widehat a_{i,n}^{2,\rm db}(\zeta_1,\zeta_2).
			\]
			Lemma~\ref{lem:iv_debiasing_main} applies to every fixed pair. Its
			leading random term \(\mathcal Z_{i,n}\) does not depend on that pair.
			Because \(\mathcal G\) is finite, a union bound gives
			\[
				\max_{(\zeta_1,\zeta_2)\in\mathcal G}
				\left|\widehat a_{i,n}^{2,\rm db}(\zeta_1,\zeta_2)
				-a_i^2(\gamma^\star)-\frac{\mathcal Z_{i,n}}{\sqrt n}\right|
				=o_p(n^{-1/2}).
			\]
			The median is Lipschitz in the maximum norm. Hence
			\[
				\widehat a_{i,n}^{2,\rm med}
				=a_i^2(\gamma^\star)+\frac{\mathcal Z_{i,n}}{\sqrt n}
				+o_p(n^{-1/2}).
			\]
			Thus Theorem~\ref{thm:parametric_inference} and
			Corollary~\ref{cor:studentization} remain valid if the median replaces
			\(\widehat a_{i,n}^{2,\rm db}\) in the diffusion contrast. We use the
			prespecified grid \(\{1.1,1.2,1.3\}^2\) in the numerical work. This step
			does not estimate a tuning parameter; it only limits the effect of an
			unstable correction pair in a finite sample.
		\end{rem}
		\end{currentrevision}

		Define
		\begin{equation}
			\label{eq:iv_gamma_estimator}
			\mathbb H_{\gamma,n}(\gamma)
			:=-\frac12\sum_{i=1}^m\widehat\pi_{i,n}
			\left\{\log a_i^2(\gamma)
			+\frac{\widehat a_{i,n}^{\,2,\rm db}}{a_i^2(\gamma)}\right\},
			\qquad
			{\bar\gamma_n}
			\in\argmax_{\gamma\in\Theta_\gamma}
			\mathbb H_{\gamma,n}(\gamma).
		\end{equation}
		The drift is then estimated adaptively, conditional on
		\({\bar\gamma_n}\):
		\begin{equation}
			\label{eq:iv_alpha_estimator}
			\mathbb H_{\alpha,n}(\alpha,\gamma)
			:=-\frac12\sum_{j=1}^n I_{n,j}
			\frac{\big[\Delta_jX-h_n\{
			b_{j-1}(\alpha)-\widehat\kappa_{Z_{j-1},n}\}\big]^2}
			{h_na_{Z_{j-1}}^2(\gamma)},
			\qquad
			{\bar\alpha_n}
			\in\argmax_{\alpha\in\Theta_\alpha}
			\mathbb H_{\alpha,n}
			(\alpha,{\bar\gamma_n}).
		\end{equation}
		{
		Thus \(\bar\theta_n
		=(\bar\alpha_n,\bar\gamma_n)\) is sequential rather
		than a joint maximizer.}
		\begin{darkgreenrevision}
		The two estimation steps use different asymptotic scales. The debiased
		regime variances estimate \(\gamma\) from high-frequency quadratic
		information. The drift contrast uses the estimated tail means and the
		ergodic information accumulated as \(T_n\to\infty\). The denominator
		indicators in \eqref{eq:iv_debias_operator} are asymptotically inactive. Their
		stability must still hold uniformly over the random regime exposures.
		\end{darkgreenrevision}

	%===========================================================================
\subsection{\greenrev{Parametric and tail-functional inference}}
\label{sec:asymptotics}

	{
	Write
	\[
		\widehat\kappa_n
		=(\widehat\kappa_{1,n},\ldots,\widehat\kappa_{m,n})^\top,
		\quad
		\kappa_n=(\kappa_{1,n},\ldots,\kappa_{m,n})^\top,
		\quad
		m^\star=(m_1^\star,\ldots,m_m^\star)^\top,
	\]
	and 
	$D_n:=\operatorname{diag}
	(\sqrt{T_n}I_{p_\alpha},\sqrt n I_{p_\gamma})$.

	For later use, define the oracle drift-corrected increment
	\[
		U_j^\star
		:=\Delta_jX-h_n\{b_{j-1}(\alpha^\star)
		-\kappa_{Z_{j-1},n}\}.
	\]
	\begin{prop}[Tail-functional estimation]
		\label{prop:tail_functional}
		\begin{greenrevision}
		Under the standing infinite-variation specification in the Introduction
		and Assumptions~\ref{ass:smooth}, \ref{ass:levy},
		\ref{ass:iv_sampling}, and \ref{ass:ergodic},
		\[
			\widehat\kappa_n\cip m^\star,
			\qquad
			\sqrt{T_n}\bigl(\widehat\kappa_n-m^\star\bigr)
			\cil N(0,\Sigma_\kappa).
		\]
		\end{greenrevision}
	\end{prop}

		\begin{thm}[Joint parametric inference]
		\label{thm:parametric_inference}
		\begin{greenrevision}
		Under the standing infinite-variation specification in the Introduction
		and Assumptions~\ref{ass:smooth}, \ref{ass:levy},
		\ref{ass:iv_sampling}, \ref{ass:ergodic}, \ref{ass:ident}, and
		\ref{ass:info_nondeg},
			\[
				{\bar\theta_n}\cip\theta^\star,\qquad
				D_n({\bar\theta_n}-\theta^\star)
				\cil N\{0,V(\theta^\star)\}.
			\]
		\end{greenrevision}
	\end{thm}

	{
	\subsection{Feasible studentization}
	\label{sec:studentization}

	{
	Use \(\bar\theta_n=(\bar\alpha_n,\bar\gamma_n)\).
	Define
	\[
	\begin{aligned}
		\widehat I_{\alpha,n}
		&:=\frac1n\sum_{j=1}^n I_{n,j}
		\frac{\{\partial_\alpha b(X_{t_{j-1}},Z_{j-1},
		\bar\alpha_n)\}^{\otimes2}}
		{\greenrev{a_{Z_{j-1}}(\bar\gamma_n)^2}},\\
		\widehat I_{\gamma,n}
		&:=\frac1{2n}\sum_{j=1}^n I_{n,j}
		\greenrev{\{\partial_\gamma\log a_{Z_{j-1}}^2(
		\bar\gamma_n)\}^{\otimes2}},\\
		\widehat J_{i,n}
		&:=\frac1n\sum_{j=1}^n I_{n,j}
		\mathbf 1_{\{Z_{j-1}=i\}}
		\frac{\partial_\alpha b(X_{t_{j-1}},i,\bar\alpha_n)}
		{\greenrev{a_i(\bar\gamma_n)^2}},
		\qquad i\in S.
	\end{aligned}
	\]
	Set \(\widehat J_n:=(\widehat J_{1,n},\ldots,
	\widehat J_{m,n})\) and
	\(g_n(y):=y^2\mathbf1_{\{|y|>u_n\}}\). For \(i\in S\), define
	\[
		\widehat Y_{i,j}:=\Delta_jX-h_n\{
		b(X_{t_{j-1}},i,\bar\alpha_n)-\widehat\kappa_{i,n}\}.
	\]
	Let
	\begin{equation}
		\label{eq:studentization_Q_definition}
		\mathcal Q_{i,n}
		:=\int_0^{T_n}\!\!\int_{|z|>u_n}
		z^2\mathbf1_{\{Z_{s-}=i\}}N_i(ds,dz),
	\end{equation}
	and define
	\[
		\widehat M_{2,i,n}
		:=\frac1{T_{i,n}}\sum_{j=1}^n
		A_{i,j}g_n(\widehat Y_{i,j}),
		\qquad
		M_{2,i,n}^\star
		:=\frac1{T_{i,n}}\sum_{j=1}^nA_{i,j}g_n(U_j^\star).
	\]
	On \(\{T_{i,n}=0\}\), set these quantities and
	\(\mathcal Q_{i,n}/T_{i,n}\) equal to zero. Put
	\[
		\widehat\Sigma_{\kappa,n}
		:=\operatorname{diag}\left(
		\frac{\widehat M_{2,i,n}}{\widehat\pi_{i,n}}:i\in S\right).
	\]

	On the event that the empirical quasi-information matrices are
	nonsingular, define
	\[
	\begin{aligned}
		\widehat V_{\alpha,n}
		&:=(\widehat I_{\alpha,n})^{-1}
		\left\{\widehat I_{\alpha,n}
		+\widehat J_n\widehat\Sigma_{\kappa,n}
		\widehat J_n^\top\right\}
		(\widehat I_{\alpha,n})^{-1},\\
		\widehat V_n
		&:=\operatorname{diag}\left(
		\widehat V_{\alpha,n},
		(\widehat I_{\gamma,n})^{-1}\right).
	\end{aligned}
	\]
	On the complementary event, whose probability tends to zero, assign any
	fixed positive-definite value.}

	\begin{cor}[Studentization]
		\label{cor:studentization}
		{
		\begin{greenrevision}
		Suppose the standing infinite-variation specification in the Introduction
		and Assumptions~\ref{ass:smooth}, \ref{ass:levy},
		\ref{ass:iv_sampling}, \ref{ass:ergodic}, \ref{ass:ident}, and
		\ref{ass:info_nondeg} hold. Then
		\[
			\max_{i\in S}\left|
			\widehat M_{2,i,n}
			-\int_{\mathbb R}z^2\nu_i^\star(dz)
			\right|\cip0,\qquad
			\widehat V_n\cip V(\theta^\star),
			\qquad
			\widehat V_n^{-1/2}D_n
			(\bar\theta_n-\theta^\star)
			\cil N(0,I_{p_\alpha+p_\gamma}).
		\]
		\end{greenrevision}}
	\end{cor}

	\begin{rem}[Coordinatewise confidence intervals]
		{
		Corollary~\ref{cor:studentization} gives, for every coordinate \(k\),
		the asymptotic
		\(100(1-\tau)\%\) confidence interval
		\[
			\left[
			\bar\theta_{n,k}
			-z_{1-\tau/2}
			\sqrt{\left[D_n^{-1}\widehat V_nD_n^{-1}\right]_{kk}},
			\quad
			\bar\theta_{n,k}
			+z_{1-\tau/2}
			\sqrt{\left[D_n^{-1}\widehat V_nD_n^{-1}\right]_{kk}}
			\right],
		\]
		where \(z_{1-\tau/2}\) is the corresponding standard-normal
		quantile.}
	\end{rem}

	\begin{rem}[Optional recovery of the L\'evy densities]
		\label{rem:optional_density}
		\begin{darkgreenrevision}
		Structural-parameter estimation does not require the full L\'evy density. If
		the density is also of interest, a standard residual kernel step can be added
		after computing \(\bar\theta_n\) and \(\widehat\kappa_n\). Fix
		\end{darkgreenrevision}
		\(B\Subset\mathbb R\setminus\{0\}\), suppose that \(s_i^\star\) is locally
		\(r\)-H\"older around \(B\), and let \(K_\eta(z)=\eta^{-1}K(z/\eta)\)
		for a compactly supported kernel of order \(r\). With
		\[
			\widehat Y_{i,j}:=\Delta_jX-h_n\{
			b(X_{t_{j-1}},i,\bar\alpha_n)-\widehat\kappa_{i,n}\},
		\]
		the regime-wise estimator
		\[
			\widehat s_{i,n}(z):=\frac1{T_{i,n}}\sum_{j=1}^n
			A_{i,j}\mathbf1_{\{|\widehat Y_{i,j}|>u_n\}}
			K_{\eta_n}(z-\widehat Y_{i,j}),\qquad z\in B,
		\]
		satisfies, under the usual kernel conditions
		\(\eta_n\to0\), \(T_n\eta_n\to\infty\), and
		\(h_n/\eta_n^{5/2}\to0\),
		\[
			\|\widehat s_{i,n}-s_i^\star\|_{L^2(B)}
			=O_p\!\left(\eta_n^r+(T_n\eta_n)^{-1/2}
			+h_n\eta_n^{-5/2}\right).
		\]
		\begin{darkgreenrevision}
		Thus \(\eta_n\asymp T_n^{-1/(2r+1)}\) gives the usual rate
		\(T_n^{-r/(2r+1)}\) whenever
		\(h_nT_n^{(r+5/2)/(2r+1)}\to0\). This is the direct regime-wise
		analogue of residual kernel methods for jump diffusions and L\'evy kernels;
		see \citet{shimizu2006density,Ueltzhofer2013LevyKernel}. Because this step is
		separate from the parametric problem, we do not develop additional
		nonparametric theory or tuning rules here.
		\end{darkgreenrevision}
	\end{rem}
	}
	}

		%%%%%%%%%%%%%%%%%%%%%%%%
	%%%%%%%%%%%%%%%%%%%%%%%%

	{
	\section{Proofs}
	\label{sec:proof_strategy}}

	{
	% BEGIN INLINED FILE: proof_strategy_main.tex
{

\begin{darkgreenrevision}
The proofs are grouped by statistical target. We treat the tail mean,
coefficient estimators, and studentization in turn.
\end{darkgreenrevision}

\subsection{Proof of Proposition~\ref{prop:tail_functional}}

\begin{proof}
\begin{darkgreenrevision}
The proof has three stages. First, we compare each detected increment with
the corresponding ideal regime-specific large-jump sum. This gives
conditional first-, second-, and fourth-moment error bounds. Second, we use
these bounds to identify the centering and predictable covariance of the
detected-jump martingale array and verify its conditional Lindeberg
condition. Finally, we apply the martingale central limit theorem and remove
the cutoff bias between the detected-jump mean and \(m^\star\).
\end{darkgreenrevision}

For \(i\in S\) and \(j\le n\), put
\[
	H_{i,j,n}:=A_{i,j}\Delta_jX
	\mathbf1_{\{|\Delta_jX|>u_n\}},
	\qquad
	\mu_{i,j,n}:=\E_{j-1}H_{i,j,n},
\]
and introduce the ideal large-jump sum and its detection error,
\[
	L_{i,j,n}:=\int_{t_{j-1}}^{t_j}\!\!\int_{|z|>u_n}
	z\mathbf1_{\{Z_{s-}=i\}}N_i(ds,dz),
	\qquad D_{i,j,n}:=H_{i,j,n}-L_{i,j,n}.
\]

\proofstep{Step 1: one-step detection bounds}
For \(0<r\le1\), define
\[
	q_i^{\rm sym}(r):=\min\{q_i(r),q_i(-r)\},
\]
\[
	\nu_i^0(dz):=c_iq_i^{\rm sym}(|z|)|z|^{-1-\beta}
	\mathbf1_{\{0<|z|\le1\}}dz,
	\qquad \bar\nu_i:=\nu_i^\star-\nu_i^0.
\]
The local expansion of \(q_i\) implies
\begin{equation}
	\label{eq:main_iv_nonnegative_split}
	0\le\nu_i^0\le\nu_i^\star,
	\qquad
	0\le\frac{d\bar\nu_i}{dz}\lesssim |z|^{-\beta}
	\quad(0<|z|\le1).
\end{equation}
Thus both terms are genuine L\'evy measures and \(\nu_i^0\) is symmetric.
We can split $N_i$ by independent Poisson random measures
\(N_i^0\) and \(\bar N_i\), with compensators
\(ds\,\nu_i^0(dz)\) and \(ds\,\bar\nu_i(dz)\), such that
	$N_i=N_i^0+\bar N_i$.
\darkgreenrev{We write \(\widetilde N_i^0\) and
\(\widetilde{\bar N}_i\) for the two compensated Poisson random measures.}

For \(0<v\le1\), \eqref{eq:main_iv_nonnegative_split} gives
\begin{equation}
	\label{eq:main_iv_split_bounds}
\begin{aligned}
	\bar\nu_i(|z|>v)
	&\lesssim \int_v^1 r^{-\beta}\,dr+\bar\nu_i(|z|>1)
	\lesssim v^{1-\beta},\\
	\int_{|z|\le v}|z|\bar\nu_i(dz)
	&\lesssim \int_0^v r^{1-\beta}\,dr
	\lesssim v^{2-\beta},\\
	\nu_i^\star(|z|>v)
	&\lesssim \int_v^1 r^{-1-\beta}\,dr+\nu_i^\star(|z|>1)
	\lesssim v^{-\beta}.
\end{aligned}
\end{equation}

Recall from Lemma~\ref{lem:endpoint_exposure_main} that
\(\Gamma_{i,j}\) is the event that the regime remains equal to \(i\) on
\([t_{j-1},t_j]\). On this event, define
\[
\begin{aligned}
	S_{i,j,n}^0
	&:=a_i(\gamma^\star)\Delta_jW
	+\int_{t_{j-1}}^{t_j}\!\!\int_{|z|\le u_n/2}
	z\,\widetilde N_i^0(ds,dz),\\
	J_{i,j,n}^{>u_n/2}
	&:=\int_{t_{j-1}}^{t_j}\!\!\int_{|z|>u_n/2}z\,N_i(ds,dz),\\
	R_{i,j,n}^0
	&:=\int_{t_{j-1}}^{t_j}b(X_s,i,\alpha^\star)\,ds
	+\int_{t_{j-1}}^{t_j}\!\!\int_{|z|\le u_n/2}
	z\,\widetilde{\bar N}_i(ds,dz)\\
	&\quad-h_n\int_{|z|>u_n/2}z\,\nu_i^\star(dz).
\end{aligned}
\]
Thus \(S_{i,j,n}^0\) consists of the Brownian increment and the
compensated symmetric \(\nu_i^0\)-small-jump increment, whereas
\(R_{i,j,n}^0\) consists of the drift, the compensated
\(\bar\nu_i\)-small-jump increment, and the
\(\nu_i^\star\)-large-jump compensator.
Then, on \(\Gamma_{i,j}=1\),
	$\Delta_jX=S_{i,j,n}^0+J_{i,j,n}^{>u_n/2}+R_{i,j,n}^0$.
By the symmetry of the Brownian increment and of \(\nu_i^0\), and by
their independence, \(S_{i,j,n}^0\stackrel d=-S_{i,j,n}^0\).
The Cauchy--Schwarz inequality and Kunita's
inequality give
\begin{equation}
	\label{eq:main_iv_symmetric_moments}
\begin{aligned}
	\E|S_{i,j,n}^0|&\lesssim \sqrt {h_n},\\
	\E(S_{i,j,n}^0)^2
	&=h_na_i^2(\gamma^\star)
	+h_n\int_{|z|\le u_n/2}z^2\nu_i^0(dz)
	\lesssim h_n,\\
	\E|S_{i,j,n}^0|^3
	&\lesssim (h_n^{3/2}+h_nu_n^{3-\beta}).
\end{aligned}
\end{equation}
Moreover, \eqref{eq:main_iv_split_bounds} and the bounded signed tail
compensator imply
\[
	\E_{j-1}[|R_{i,j,n}^0|;\Gamma_{i,j}=1]
	\lesssim h_n(1+u_n^{2-\beta})R_{j-1},
	\qquad
	\E_{j-1}[(R_{i,j,n}^0)^2;\Gamma_{i,j}=1]
	\lesssim (h_nu_n^{3-\beta}+h_n^2)R_{j-1}.
\]
For \(\psi_{u_n}(y)=y\mathbf1_{\{|y|>u_n\}}\), put
	$\Phi_{1,i,n}(z)
	:=\E\psi_{u_n}(z+S_{i,j,n}^0)
	-z\mathbf1_{\{|z|>u_n\}}$.
Since \(S_{i,j,n}^0\stackrel d=-S_{i,j,n}^0\) and \(\psi_{u_n}\) is
odd, \(\Phi_{1,i,n}(-z)=-\Phi_{1,i,n}(z)\), and hence
\begin{equation}
	\label{eq:main_iv_boundary_cancel}
	\int_{|z|>u_n/2}\Phi_{1,i,n}(z)\nu_i^0(dz)=0.
\end{equation}
For deterministic \(z,s\in\mathbb R\),
\[
\begin{aligned}
	\psi_{u_n}(z+s)-z\mathbf1_{\{|z|>u_n\}}
	&=s\mathbf1_{\{|z+s|>u_n\}}+z\left(
	\mathbf1_{\{|z+s|>u_n\}}-
	\mathbf1_{\{|z|>u_n\}}\right).
\end{aligned}
\]
By the reverse triangle inequality,
	$\left|
	\mathbf1_{\{|z+s|>u_n\}}-
	\mathbf1_{\{|z|>u_n\}}\right|
	\le
	\mathbf1_{\{\,\lvert |z|-u_n\rvert\le|s|\,\}}$.
\darkgreenrev{Thus the cutoff indicator changes only on the displayed random strip.}
With \(s=S_{i,j,n}^0\), on its complement one has
\[
	\psi_{u_n}(z+S_{i,j,n}^0)
	-z\mathbf1_{\{|z|>u_n\}}
	=S_{i,j,n}^0\mathbf1_{\{|z|>u_n\}}
	\in\{0,S_{i,j,n}^0\}.
\]
Integrating the strip width against \eqref{eq:main_iv_split_bounds}, and
then using \eqref{eq:main_iv_symmetric_moments}, gives
\begin{equation}
	\label{eq:main_iv_boundary_integrals}
	\begin{aligned}
	\int_{|z|>u_n/2}|\Phi_{1,i,n}(z)|\bar\nu_i(dz)
	&\lesssim (u_n^{2-\beta}+h_nu_n^{-\beta}+h_n),\\
	\int_{|z|>u_n/2}\E|\psi_{u_n}(z+S_{i,j,n}^0)
	-z\mathbf1_{\{|z|>u_n\}}|^2
	\nu_i^\star(dz)
	&\lesssim (u_n^{2-\beta}+h_nu_n^{2-2\beta}
	+h_nu_n^{-\beta}+h_n).
	\end{aligned}
\end{equation}

Define the number of active regime-\(i\) jumps larger than \(u_n/2\) by
\[
	K_{i,j,n}:=\int_{t_{j-1}}^{t_j}\!\!\int_{|z|>u_n/2}
	\mathbf1_{\{Z_{s-}=i\}}N_i(ds,dz).
\]
Conditional on the regime path over the interval, this is Poisson with
parameter
\[
	\Lambda_{i,j,n}
	:=\nu_i^\star(|z|>u_n/2)
	\int_{t_{j-1}}^{t_j}\mathbf1_{\{Z_{s-}=i\}}\,ds
	\lesssim h_nu_n^{-\beta}=o(1).
\]
On \(\{\Gamma_{i,j}=1,K_{i,j,n}=0\}\),
\[
	J_{i,j,n}^{>u_n/2}=L_{i,j,n}=0,
	\qquad
	D_{i,j,n}=H_{i,j,n}
	=\psi_{u_n}(S_{i,j,n}^0+R_{i,j,n}^0).
\]
The symmetry of \(S_{i,j,n}^0\), together with its conditional
independence of the regime path and the large-jump count, gives
\[
	\E_{j-1}[\psi_{u_n}(S_{i,j,n}^0);
	\Gamma_{i,j}=1,K_{i,j,n}=0]=0.
\]
Hence
\[
\begin{aligned}
	&|\E_{j-1}[D_{i,j,n};\Gamma_{i,j}=1,K_{i,j,n}=0]|\\
	&\quad\le
	\E_{j-1}[|\psi_{u_n}(S_{i,j,n}^0+R_{i,j,n}^0)
	-\psi_{u_n}(S_{i,j,n}^0)|;
	\Gamma_{i,j}=1,K_{i,j,n}=0],\\
	&\E_{j-1}[D_{i,j,n}^2;\Gamma_{i,j}=1,K_{i,j,n}=0]\\
	&\quad\le2\E_{j-1}[|\psi_{u_n}(S_{i,j,n}^0)|^2;
	\Gamma_{i,j}=1,K_{i,j,n}=0]
	+2\E_{j-1}[|\psi_{u_n}(S_{i,j,n}^0+R_{i,j,n}^0)
	-\psi_{u_n}(S_{i,j,n}^0)|^2;\Gamma_{i,j}=1].
\end{aligned}
\]
For deterministic \(x,r\in\mathbb R\), the pointwise bounds
\[
\begin{aligned}
	|\psi_{u_n}(x+r)-\psi_{u_n}(x)|
	&\le |r|+|x|
	\mathbf1_{\{\,\lvert|x|-u_n\rvert\le|r|\,\}},\\
	|\psi_{u_n}(x+r)-\psi_{u_n}(x)|^2
	&\le2r^2+2x^2
	\mathbf1_{\{\,\lvert|x|-u_n\rvert\le|r|\,\}}
\end{aligned}
\]
show that the indicator can change only on
	$\{\,\lvert|S_{i,j,n}^0|-u_n\rvert
	\le |R_{i,j,n}^0|\,\}$.
The moment estimates for \(S_{i,j,n}^0\) and \(R_{i,j,n}^0\), together
with the random-boundary estimate of
Proposition~\ref{prop:iv_off_diagonal_density} on
\(\{\lvert|S_{i,j,n}^0|-u_n\rvert\le|R_{i,j,n}^0|\}\), yield
\[
\begin{aligned}
	&|\E_{j-1}[D_{i,j,n};\Gamma_{i,j}=1,K_{i,j,n}=0]|
	+\E_{j-1}[D_{i,j,n}^2;\Gamma_{i,j}=1,K_{i,j,n}=0]\\
	&\qquad\lesssim h_n(u_n^{2-\beta}+h_nu_n^{-\beta}+h_n)R_{j-1}.
\end{aligned}
\]

On \(\{\Gamma_{i,j}=1,K_{i,j,n}=1\}\), let \(\xi_{i,j,n}\) be the
unique jump. When the one-point Campbell--Mecke formula replaces
\(\xi_{i,j,n}\) by the integration variable \(z\), the drift integral
must be evaluated along the path containing this jump. Denote the
resulting remainder by \(R_{i,j,n}^0(z)\); on the one-jump event,
\[
	R_{i,j,n}^0(\xi_{i,j,n})=R_{i,j,n}^0.
\]
For \(|z|>u_n/2\), put
\[
\begin{aligned}
	G_{i,j,n}^0(z)
	&:=\psi_{u_n}(z+S_{i,j,n}^0)
	-z\mathbf1_{\{|z|>u_n\}},\\
	G_{i,j,n}(z)
	&:=\psi_{u_n}(z+S_{i,j,n}^0+R_{i,j,n}^0(z))
	-z\mathbf1_{\{|z|>u_n\}}.
\end{aligned}
\]
Then
\[
	D_{i,j,n}=G_{i,j,n}(\xi_{i,j,n}).
\]
Conditional on the regime path and \(S_{i,j,n}^0\), use the following
one-point Poisson identity: if a Poisson random measure on \(E\) has
finite intensity \(\mu\), \(K\) is its total count, and \(\xi\) is its
unique point on \(\{K=1\}\), then, for every \(\mu\)-integrable \(f\),
\[
	\E[f(\xi);K=1]=e^{-\mu(E)}\int_E f(z)\,\mu(dz).
\]
This is the one-point case of the Campbell--Mecke formula. Applying it
with \(E=\{|z|>u_n/2\}\) and, on \(\{\Gamma_{i,j}=1\}\),
\(\mu(dz)=h_n\nu_i^\star(dz)\), and then using
\eqref{eq:main_iv_boundary_cancel}, gives
\[
\begin{aligned}
	&\E_{j-1}[G_{i,j,n}^0(\xi_{i,j,n});
	\Gamma_{i,j}=1,K_{i,j,n}=1]\\
	&\quad=h_n\E_{j-1}[\Gamma_{i,j}e^{-\Lambda_{i,j,n}}]
	\int_{|z|>u_n/2}\Phi_{1,i,n}(z)\nu_i^\star(dz)\\
	&\quad=h_n\E_{j-1}[\Gamma_{i,j}e^{-\Lambda_{i,j,n}}]
	\int_{|z|>u_n/2}\Phi_{1,i,n}(z)\bar\nu_i(dz),
\end{aligned}
\]
while
\[
\begin{aligned}
	&\E_{j-1}[|G_{i,j,n}^0(\xi_{i,j,n})|^2;
	\Gamma_{i,j}=1,K_{i,j,n}=1]\\
	&\quad\le h_n\int_{|z|>u_n/2}
	\E|G_{i,j,n}^0(z)|^2\nu_i^\star(dz).
\end{aligned}
\]
Thus \eqref{eq:main_iv_boundary_integrals} bounds the one-jump mean and
second moment associated with \(G_{i,j,n}^0\). To control the
perturbation caused by \(R_{i,j,n}^0(z)\), put
\[
\begin{aligned}
	\Delta G_{i,j,n}(z)
	&:=G_{i,j,n}(z)-G_{i,j,n}^0(z)\\
	&=\psi_{u_n}(z+S_{i,j,n}^0+R_{i,j,n}^0(z))
	-\psi_{u_n}(z+S_{i,j,n}^0),\\
	\mathcal B_{i,j,n}(z)
	&:=\left\{\left||z+S_{i,j,n}^0|-u_n\right|
	\le |R_{i,j,n}^0(z)|\right\}.
\end{aligned}
\]
Apply the two deterministic bounds for
\(\psi_{u_n}(x+r)-\psi_{u_n}(x)\) displayed before
\eqref{eq:main_iv_boundary_integrals}, with
\[
	x=z+S_{i,j,n}^0,
	\qquad r=R_{i,j,n}^0(z).
\]
\darkgreenrev{On \(\mathcal B_{i,j,n}(z)^c\), the two cutoff indicators
coincide. Hence}
\[
	\Delta G_{i,j,n}(z)
	=R_{i,j,n}^0(z)
	\mathbf1_{\{|z+S_{i,j,n}^0|>u_n\}}.
\]
On \(\mathcal B_{i,j,n}(z)\),
\[
	|z+S_{i,j,n}^0|
	\le u_n+|R_{i,j,n}^0(z)|.
\]
Consequently, for \(k\in\{1,2\}\),
\begin{equation}
\begin{aligned}
	|\Delta G_{i,j,n}(z)|^k
	\lesssim |R_{i,j,n}^0(z)|^k
	+\{u_n^k+|R_{i,j,n}^0(z)|^k\}
	\mathbf1_{\mathcal B_{i,j,n}(z)}.
\end{aligned}
\label{eq:main_iv_one_jump_perturbation_pointwise}
\end{equation}

Let
\[
	r_n:=h_nu_n^{1-\beta}.
\]
By \eqref{eq:main_iv_split_bounds} and
Assumption~\ref{ass:iv_sampling},
\[
	\frac{r_n}{u_n}=h_nu_n^{-\beta}\longrightarrow0,
\]
so \(Cr_n\le u_n/2\) for every fixed \(C>0\) and all sufficiently
large \(n\). On
\(\mathcal B_{i,j,n}(z)\cap\{|R_{i,j,n}^0(z)|\le r_n\}\), the reverse
triangle inequality gives
\[
	\left|\left|z+S_{i,j,n}^0+R_{i,j,n}^0(z)\right|-u_n\right|
	\le2r_n.
\]
Adding the bounded shift \(h_n\kappa_{i,n}\) only enlarges this
deterministic width by \(O(h_n)\), which is absorbed by \(r_n\).
Therefore the boundary estimate in
Proposition~\ref{prop:iv_off_diagonal_density}, with \(v=u_n\) and
width \(Cr_n\), and the one-point Campbell--Mecke formula give
\begin{align}
	&h_n\int_{|z|>u_n/2}
	\E_{j-1}\!\left[
	\{u_n^k+|R_{i,j,n}^0(z)|^k\}
	\mathbf1_{\mathcal B_{i,j,n}(z)}
	\mathbf1_{\{|R_{i,j,n}^0(z)|\le r_n\}};
	\Gamma_{i,j}=1\right]\nu_i^\star(dz)\nonumber\\
	&\qquad\lesssim
	(u_n^k+r_n^k)R_{j-1}
	\left\{h_nr_nu_n^{-1-\beta}
	+r_nh_n^{-1/2}e^{-u_n^2/(Ch_n)}\right\}\nonumber\\
	&\qquad\lesssim h_n^2u_n^{k-2\beta}R_{j-1}.
	\label{eq:main_iv_one_jump_boundary_small_remainder}
\end{align}
Here the Gaussian term is smaller than every polynomial rate because
\(u_n/\sqrt{h_n}\to\infty\).

It remains to control \(|R_{i,j,n}^0(z)|>r_n\). Split the remainder
into its drift--compensator part and its compensated
\(\bar\nu_i\)-small-jump part. The compensation formula and
\eqref{eq:main_iv_split_bounds} give, for \(k\in\{1,2\}\),
\[
	\nu_i^\star(|z|>u_n/2)\lesssim u_n^{-\beta},
	\qquad
	\bar\nu_i(r_n<|w|\le u_n/2)\lesssim r_n^{1-\beta},
\]
\[
	\int_{r_n<|w|\le u_n/2}|w|^k\bar\nu_i(dw)
	\lesssim u_n^{k+1-\beta}.
\]
The first two factors produce
\(u_n^{k-\beta}r_n^{1-\beta}\), while the first and third produce
\(u_n^{k+1-2\beta}\). The drift and compensator contributions are
bounded by the remaining \(u_n^{-\beta}+1\) terms. Consequently,
\begin{align}
	&h_n\int_{|z|>u_n/2}
	\E_{j-1}\!\left[|R_{i,j,n}^0(z)|^k;
	\Gamma_{i,j}=1\right]\nu_i^\star(dz)
	\lesssim h_n^2
	\{u_n^{k+1-2\beta}+u_n^{-\beta}+1\}R_{j-1},
	\label{eq:main_iv_one_jump_remainder_integral}\\
	&h_n\int_{|z|>u_n/2}
	\E_{j-1}\!\left[
	\{u_n^k+|R_{i,j,n}^0(z)|^k\}
	\mathbf1_{\{|R_{i,j,n}^0(z)|>r_n\}};
	\Gamma_{i,j}=1\right]\nu_i^\star(dz)\nonumber\\
	&\qquad\lesssim h_n^2
	\{u_n^{k-\beta}r_n^{1-\beta}
	+u_n^{k+1-2\beta}+u_n^{-\beta}+1\}R_{j-1}\nonumber\\
	&\qquad\lesssim h_n^2
	\{u_n^{k-2\beta}+u_n^{-\beta}+1\}R_{j-1}.
	\label{eq:main_iv_one_jump_boundary_large_remainder}
\end{align}
For the last inequality, \(u_n<1\) and the lower bound
\(\omega>(4-\beta)^{-1}\) in \eqref{eq:iv_threshold} imply
\[
	u_n^\beta r_n^{1-\beta}
	=h_n^{1-\beta+\omega\{\beta+(1-\beta)^2\}}\lesssim1,
\]
because
\[
	\frac1{4-\beta}>
	\frac{\beta-1}{\beta+(1-\beta)^2},
	\qquad 1<\beta<2.
\]
Finally, the one-point Campbell--Mecke formula, followed by
\eqref{eq:main_iv_one_jump_perturbation_pointwise}--%
\eqref{eq:main_iv_one_jump_boundary_large_remainder}, gives
\[
\begin{aligned}
	&\E_{j-1}[|G_{i,j,n}(\xi_{i,j,n})
	-G_{i,j,n}^0(\xi_{i,j,n})|;
	\Gamma_{i,j}=1,K_{i,j,n}=1]\\
	&\qquad\lesssim h_n^2(u_n^{1-2\beta}+u_n^{-\beta}+1)R_{j-1},\\
	&\E_{j-1}[|G_{i,j,n}(\xi_{i,j,n})
	-G_{i,j,n}^0(\xi_{i,j,n})|^2;
	\Gamma_{i,j}=1,K_{i,j,n}=1]\\
	&\qquad\lesssim h_n^2(u_n^{2-2\beta}+u_n^{-\beta}+1)R_{j-1}.
\end{aligned}
\]
Combining the bounds for \(G_{i,j,n}^0(\xi_{i,j,n})\) and
\(G_{i,j,n}(\xi_{i,j,n})-G_{i,j,n}^0(\xi_{i,j,n})\) yields
\[
\begin{aligned}
	|\E_{j-1}[D_{i,j,n};\Gamma_{i,j}=1,K_{i,j,n}=1]|
	&\lesssim h_n(u_n^{2-\beta}+h_nu_n^{1-2\beta}
	+h_nu_n^{-\beta}+h_n)R_{j-1},\\
	\E_{j-1}[D_{i,j,n}^2;\Gamma_{i,j}=1,K_{i,j,n}=1]
	&\lesssim h_n(u_n^{2-\beta}+h_nu_n^{2-2\beta}
	+h_nu_n^{-\beta}+h_n)R_{j-1}.
\end{aligned}
\]

For the multiple-count event, put
\[
	\mathcal J_{i,j,n}
	:=\int_{t_{j-1}}^{t_j}\!\!\int_{|z|>u_n/2}|z|N_i(ds,dz).
\]
On \(\Gamma_{i,j}=1\),
\[
	|D_{i,j,n}|
	\le2\mathcal J_{i,j,n}
	+|S_{i,j,n}^0+R_{i,j,n}^0|,
	\qquad
	D_{i,j,n}^2
	\lesssim \{\mathcal J_{i,j,n}^2
	+(S_{i,j,n}^0+R_{i,j,n}^0)^2\}.
\]
The IV L\'evy bounds imply
\[
	\nu_i^\star(|z|>u_n/2)\lesssim u_n^{-\beta},
	\qquad
	\int_{|z|>u_n/2}|z|\nu_i^\star(dz)\lesssim u_n^{1-\beta},
	\qquad
	\int_{|z|>u_n/2}z^2\nu_i^\star(dz)\lesssim 1.
\]
Since
\(\E[K_{i,j,n}(K_{i,j,n}-1)\mid Z]=\Lambda_{i,j,n}^2\), the
conditional compound-Poisson moment formula gives
\[
\begin{aligned}
	&\E_{j-1}[\mathcal J_{i,j,n};
	\Gamma_{i,j}=1,K_{i,j,n}\ge2]\\
	&\quad\lesssim h_n^2\nu_i^\star(|z|>u_n/2)
	\int_{|z|>u_n/2}|z|\nu_i^\star(dz)
	\lesssim h_n^2u_n^{1-2\beta},\\
	&\E_{j-1}[\mathcal J_{i,j,n}^2;
	\Gamma_{i,j}=1,K_{i,j,n}\ge2]\\
	&\quad\lesssim h_n^2\left\{
	\left(\int_{|z|>u_n/2}|z|\nu_i^\star(dz)\right)^2
	+\nu_i^\star(|z|>u_n/2)
	\int_{|z|>u_n/2}z^2\nu_i^\star(dz)\right\}\\
	&\quad\lesssim h_n^2(u_n^{2-2\beta}+u_n^{-\beta}).
\end{aligned}
\]
The moment bounds for \(S_{i,j,n}^0\) and \(R_{i,j,n}^0\), together with
\(\sqrt{h_n}/u_n\to0\), also give
\[
\begin{aligned}
	&\E_{j-1}[|S_{i,j,n}^0+R_{i,j,n}^0|
	+(S_{i,j,n}^0+R_{i,j,n}^0)^2;
	\Gamma_{i,j}=1,K_{i,j,n}\ge2]\\
	&\qquad\lesssim h_n^2
	(u_n^{1-2\beta}+u_n^{2-2\beta}+u_n^{-\beta}+1)R_{j-1}.
\end{aligned}
\]
These formulas yield
\begin{equation}
	\label{eq:main_iv_count_bounds}
\begin{aligned}
	&|\E_{j-1}[D_{i,j,n};\Gamma_{i,j}=1,K_{i,j,n}=0]|
	+\E_{j-1}[D_{i,j,n}^2;\Gamma_{i,j}=1,K_{i,j,n}=0]\\
	&\qquad\lesssim h_n(u_n^{2-\beta}+h_nu_n^{-\beta}+h_n)R_{j-1},\\
	&|\E_{j-1}[D_{i,j,n};\Gamma_{i,j}=1,K_{i,j,n}=1]|\\
	&\qquad\lesssim h_n(u_n^{2-\beta}+h_nu_n^{1-2\beta}
	+h_nu_n^{-\beta}+h_n)R_{j-1},\\
	&\E_{j-1}[D_{i,j,n}^2;\Gamma_{i,j}=1,K_{i,j,n}=1]\\
	&\qquad\lesssim h_n(u_n^{2-\beta}+h_nu_n^{2-2\beta}
	+h_nu_n^{-\beta}+h_n)R_{j-1},\\
	&\E_{j-1}[|D_{i,j,n}|+D_{i,j,n}^2;
	\Gamma_{i,j}=1,K_{i,j,n}\ge2]\\
	&\qquad\lesssim h_n(h_nu_n^{1-2\beta}+h_nu_n^{2-2\beta}
	+h_nu_n^{-\beta}+h_n)R_{j-1}.
\end{aligned}
\end{equation}
For intervals with a regime mismatch,
\[
	\{\Gamma_{i,j}=0,\ D_{i,j,n}\ne0\}
	\subset
	\{N_{j,n}^Z\ge1,K_{i,j,n}\ge1\}
	\cup\{N_{j,n}^Z\ge2\}.
\]
\darkgreenrev{The first event contains a transition and an active jump. The
second is the leave-and-return event. Therefore,}
\[
	\Pb_{j-1}(N_{j,n}^Z\ge1,K_{i,j,n}\ge1)
	\lesssim h_n^2u_n^{-\beta},
	\qquad
	\Pb_{j-1}(N_{j,n}^Z\ge2)\lesssim h_n^2.
\]
Using the bounded signed first jump characteristic and the finite second
jump moment, the corresponding detection error satisfies
\[
	|\E_{j-1}[D_{i,j,n};\Gamma_{i,j}=0]|
	+\E_{j-1}[D_{i,j,n}^2;\Gamma_{i,j}=0]
	\lesssim h_n^2R_{j-1}.
\]
Therefore, with
\[
	b_n^{\rm det}:=u_n^{2-\beta}+h_nu_n^{1-2\beta}
	+h_nu_n^{-\beta}+h_n,
	\qquad
	v_n^{\rm det}:=u_n^{2-\beta}+h_nu_n^{2-2\beta}
	+h_nu_n^{-\beta}+h_n,
\]
\begin{equation}
	\label{eq:main_iv_detection_error}
	\begin{aligned}
	|\E_{j-1}D_{i,j,n}|
	&\le\sum_{k=0}^1
	|\E_{j-1}[D_{i,j,n};\Gamma_{i,j}=1,K_{i,j,n}=k]|\\
	&\quad+|\E_{j-1}[D_{i,j,n};
	\Gamma_{i,j}=1,K_{i,j,n}\ge2]|
	+|\E_{j-1}[D_{i,j,n};\Gamma_{i,j}=0]|\\
	&\lesssim h_nb_n^{\rm det}R_{j-1},\\[3pt]
	\E_{j-1}D_{i,j,n}^2
	&=\sum_{k=0}^1
	\E_{j-1}[D_{i,j,n}^2;\Gamma_{i,j}=1,K_{i,j,n}=k]\\
	&\quad+\E_{j-1}[D_{i,j,n}^2;
	\Gamma_{i,j}=1,K_{i,j,n}\ge2]
	+\E_{j-1}[D_{i,j,n}^2;\Gamma_{i,j}=0]\\
	&\lesssim h_nv_n^{\rm det}R_{j-1}.
	\end{aligned}
\end{equation}
The IV exponent restrictions give
\(\sqrt{T_n}b_n^{\rm det}\to0\) and \(v_n^{\rm det}\to0\).
The same count decomposition, now using the fourth-moment assumption for
large jumps, gives
\(\E_{j-1}|H_{i,j,n}|^4\lesssim h_nR_{j-1}\).

\proofstep{Step 2: centering and predictable covariance}
The preceding estimates can be written as
\begin{equation}
	\label{eq:main_unified_detection_error}
	|\E_{j-1}D_{i,j,n}|\lesssim h_n\varepsilon_{1,n}R_{j-1},
	\qquad
	\E_{j-1}D_{i,j,n}^2\lesssim h_n\varepsilon_{2,n}R_{j-1}.
\end{equation}
Here \(\sqrt{T_n}\varepsilon_{1,n}\to0\) and
\(\varepsilon_{2,n}\to0\). The compensator formula yields
\[
	\E_{j-1}L_{i,j,n}
	=\kappa_{i,n}\E_{j-1}\!\int_{t_{j-1}}^{t_j}
	\mathbf1_{\{Z_{s-}=i\}}ds.
\]
Lemma~\ref{lem:endpoint_exposure_main}, the uniform boundedness of
\(\kappa_{i,n}\), \eqref{eq:main_unified_detection_error}, and the
sampled ergodic bound in Lemma~\ref{lem:moment_ergodic_main} imply
\begin{equation}
	\label{eq:proof_tail_centering}
	\frac1{\sqrt{T_n}}
	\left\{\sum_{j=1}^n\mu_{i,j,n}-T_{i,n}\kappa_{i,n}\right\}
	=o_p(1).
\end{equation}

Put \(m_{2,i}(v)=\int_{|z|>v}z^2\nu_i^\star(dz)\). The
compound-Poisson formula gives
\[
	\E_{j-1}L_{i,j,n}^2
	=m_{2,i}(u_n)\E_{j-1}\!\int_{t_{j-1}}^{t_j}
	\mathbf1_{\{Z_{s-}=i\}}ds+O(h_n^2)R_{j-1}.
\]
By the conditional Cauchy--Schwarz inequality and
\eqref{eq:main_unified_detection_error}, replacing \(L_{i,j,n}\) by
\(H_{i,j,n}\) contributes \(o_p(T_n)\) after summation. Hence
\begin{equation}
	\label{eq:proof_tail_variance}
	\frac1{T_n}\sum_{j=1}^n
	\E_{j-1}(H_{i,j,n}-\mu_{i,j,n})^2
	\cip\pi_i\int_{\mathbb R}z^2\nu_i^\star(dz).
\end{equation}
The occupation statement used here is the explicit predictable limit
\[
	\frac1{T_n}\sum_{j=1}^n
	\E_{j-1}\!\int_{t_{j-1}}^{t_j}
	\mathbf1_{\{Z_{s-}=i\}}\,ds
	=
	\frac1n\sum_{j=1}^n\mathbf1_{\{Z_{j-1}=i\}}+O(h_n)
	\cip\pi_i.
\]
The equality follows from the one-step transition bound for the finite-state
Markov chain, and the convergence follows from
Lemma~\ref{lem:moment_ergodic_main} applied to
\(f(x,k)=\mathbf1_{\{k=i\}}\). Also,
\(m_{2,i}(u_n)\uparrow\int_{\mathbb R}z^2\nu_i^\star(dz)\) by the
monotone convergence theorem as \(u_n\downarrow0\). These two facts give
\eqref{eq:proof_tail_variance}.
For \(i\ne k\), \(H_{i,j,n}H_{k,j,n}=0\), while
\(|\mu_{i,j,n}\mu_{k,j,n}|\lesssim h_n^2R_{j-1}\); consequently the
normalized predictable cross covariance tends to zero. Moreover,
\begin{equation}
	\label{eq:proof_tail_fourth}
	\frac1{T_n^2}\sum_{i\in S}\sum_{j=1}^n
	\E_{j-1}|H_{i,j,n}-\mu_{i,j,n}|^4\cip0.
\end{equation}

The same bounds also record the global square-error estimate needed
below. Indeed, the tower property, Markov's inequality, and
\eqref{eq:main_unified_detection_error} give
\begin{equation}
	\label{eq:main_global_detected_jump}
	\max_{i\in S}\frac1{T_n}\sum_{j=1}^nD_{i,j,n}^2\cip0.
\end{equation}

\proofstep{Step 3: martingale central limit theorem}
Define
\[
	M_{n,j}:=\frac1{\sqrt{T_n}}
	(H_{i,j,n}-\mu_{i,j,n}:i\in S)^\top,
	\qquad M_n:=\sum_{j=1}^nM_{n,j}.
\]
\darkgreenrev{Equation \eqref{eq:proof_tail_variance} gives the predictable
covariance condition, and \eqref{eq:proof_tail_fourth} gives the conditional
Lindeberg condition.} The Cram\'er--Wold martingale
central limit theorem therefore gives
\[
	M_n\cil N\!\left(0,
	\operatorname{diag}\left(
	\pi_i\int_{\mathbb R}z^2\nu_i^\star(dz):i\in S\right)\right).
\]
Since \(T_{i,n}/T_n\cip\pi_i\), \eqref{eq:proof_tail_centering} yields
\[
	\sqrt{T_n}(\widehat\kappa_n-\kappa_n)
	\cil N(0,\Sigma_\kappa).
\]

\proofstep{Step 4: replacement of \(\kappa_n\) by \(m^\star\)}
\greenrev{The standing locally stable specification in the Introduction} gives
\(|\kappa_{i,n}-m_i^\star|\lesssim u_n^{2-\beta}\), and
the rate comparison in Remark~\ref{rem:iv_span_interpretation} gives
\(\sqrt{T_n}u_n^{2-\beta}\to0\). Slutsky's theorem proves the stated
central limit theorem around \(m^\star\); consistency follows.
\end{proof}
\subsection{Proof of Theorem~\ref{thm:parametric_inference}}

\begin{proof}
\begin{darkgreenrevision}
The proof has four parts. First, we compare the retained oracle-score
increments with their Brownian counterparts and prove consistency. Second,
we combine the Brownian, quadratic, and detected-tail terms in one augmented
martingale array and derive the joint score limit. Third, we prove the local
Hessian limits under the two parameter normalizations. Finally, a Taylor
expansion of the first-order conditions combines the score and Hessian limits.
\end{darkgreenrevision}

\proofstep{Step 1: retained-score reductions}
Define
\[
	\mathcal R_{j,n}^{(1)}
	:=I_{n,j}U_j^\star-\E_{j-1}[I_{n,j}U_j^\star],
	\qquad
	\mathcal W_{j,n}^{(1)}
	:=a_{j-1}(\gamma^\star)\Delta_jW.
\]
Finally, for \(i\in S\), define
\[
	H_{i,j,n}:=A_{i,j}\Delta_jX
	\mathbf1_{\{|\Delta_jX|>u_n\}},
	\quad
	\mu_{i,j,n}:=\E_{j-1}H_{i,j,n},
	\quad
	\widetilde H_{i,j,n}:=H_{i,j,n}-\mu_{i,j,n}.
\]

We first establish the conditional \(L^2\) comparison used below.
Write \(a_{j-1}^\star=a_{j-1}(\gamma^\star)\).
The retained increment must be split according to the number of jumps
larger than \(u_n/2\). Since
\(\E_{j-1}\mathcal W_{j,n}^{(1)}=0\),
\[
\begin{aligned}
	\mathcal R_{j,n}^{(1)}-\mathcal W_{j,n}^{(1)}
	={}&I_{n,j}U_j^\star-\mathcal W_{j,n}^{(1)}
	-\E_{j-1}[I_{n,j}U_j^\star],\\
	I_{n,j}U_j^\star-\mathcal W_{j,n}^{(1)}
	={}&I_{n,j}\{U_j^\star-\mathcal W_{j,n}^{(1)}\}
	-(1-I_{n,j})\mathcal W_{j,n}^{(1)}.
\end{aligned}
\]
Consequently,
\[
\begin{aligned}
	&\E_{j-1}\!\left[
	\{\mathcal R_{j,n}^{(1)}
	-\mathcal W_{j,n}^{(1)}\}^2\right]\\
	&\quad\lesssim
	\E_{j-1}\!\left[
	I_{n,j}|U_j^\star-\mathcal W_{j,n}^{(1)}|^2\right]
	+\E_{j-1}\!\left[
	(\mathcal W_{j,n}^{(1)})^2;|\Delta_jX|>u_n\right]
	+|\E_{j-1}[I_{n,j}U_j^\star]|^2.
\end{aligned}
\]
On a no-transition interval with \(Z_{j-1}=i\), the zero-, one-, and
multiple-count calculations underlying
\eqref{eq:main_iv_count_bounds} apply after omitting the endpoint
selector \(A_{i,j}\). Omitting that selector adds only transition
intervals; Brownian independence and the conditional moment bounds in
Lemma~\ref{lem:iv_drift_moments_main} give an additional
\(O(h_n^2R_{j-1})\) contribution. The same count calculations and the
random-boundary estimate in
Proposition~\ref{prop:iv_off_diagonal_density} therefore give
\[
\begin{aligned}
	\E_{j-1}\!\left[
	I_{n,j}|U_j^\star-\mathcal W_{j,n}^{(1)}|^2\right]
	&\lesssim h_n\left\{
	u_n^{2-\beta}+h_nu_n^{-\beta}+h_n
	+e^{-u_n^2/(Ch_n)}\right\}R_{j-1},\\
	\E_{j-1}\!\left[
	(\mathcal W_{j,n}^{(1)})^2;|\Delta_jX|>u_n\right]
	&\lesssim h_n\left\{
	h_nu_n^{-\beta}+e^{-u_n^2/(Ch_n)}\right\}R_{j-1}.
\end{aligned}
\]
Moreover, the first retained-moment bound in
Lemma~\ref{lem:iv_drift_moments_main} yields
\[
	|\E_{j-1}[I_{n,j}U_j^\star]|^2
	\lesssim h_n^2(b_n^{\rm det})^2R_{j-1}
	\lesssim h_n(b_n^{\rm det})^2R_{j-1},
\]
where the polynomial envelope is enlarged if necessary. Combining
these estimates proves
\begin{equation}
	\E_{j-1}\!\left[
	\{\mathcal R_{j,n}^{(1)}-\mathcal W_{j,n}^{(1)}\}^2\right]
	\lesssim h_n\epsilon_{1,n}^{\rm iv}R_{j-1},
	\label{eq:main_retained_l2_linear_iv}
\end{equation}
where
\[
	\epsilon_{1,n}^{\rm iv}
	\lesssim u_n^{2-\beta}+h_nu_n^{-\beta}
	+h_n+e^{-u_n^2/(Ch_n)}+(b_n^{\rm det})^2\to0.
\]
After the outer factor \(h_n\), the term \(u_n^{2-\beta}\) is the
retained small-jump variance; \(h_nu_n^{-\beta}\) comes from an
interval containing a jump larger than \(u_n/2\), including deletion
of its Brownian increment; the exponential term is the Gaussian part
of the random boundary strip; \(h_n\) collects the drift and
transition remainders; and \((b_n^{\rm det})^2\) is the conditional
centering contribution.

Let \(w_{j-1}\) be a predictable scalar weight bounded by a polynomial
envelope. Put
\(d_{j,n}^{(1)}:=\mathcal R_{j,n}^{(1)}
-\mathcal W_{j,n}^{(1)}\). This is a martingale difference. Conditional
isometry and \eqref{eq:main_retained_l2_linear_iv} give
\[
\begin{aligned}
	&\E\left|
	\frac1{\sqrt{T_n}}\sum_{j=1}^n w_{j-1}d_{j,n}^{(1)}
	\right|^2
	=\frac1{T_n}\sum_{j=1}^n
	\E\!\left[w_{j-1}^2\E_{j-1}(d_{j,n}^{(1)})^2\right]\\
	&\quad\lesssim\epsilon_{1,n}\frac1n
	\sum_{j=1}^n\E R_{j-1}
	\lesssim\epsilon_{1,n}\to0.
\end{aligned}
\]
The last bound follows from the uniform polynomial moment bound in
Lemma~\ref{lem:moment_ergodic_main}. Markov's inequality therefore yields
\begin{equation}
	\frac1{\sqrt{T_n}}\sum_{j=1}^n
	w_{j-1}\{\mathcal R_{j,n}^{(1)}
	-\mathcal W_{j,n}^{(1)}\}=o_p(1).
	\label{eq:main_retained_linear_reduction}
\end{equation}

Since \(I_{n,j}H_{i,j,n}=0\), expanding the conditional centering gives
\[
	\E_{j-1}[\mathcal R_{j,n}^{(1)}\widetilde H_{i,j,n}]
	=-\E_{j-1}[I_{n,j}U_j^\star]\mu_{i,j,n}.
\]
Since \(D_{i,j,n}=H_{i,j,n}-L_{i,j,n}\),
\eqref{eq:main_unified_detection_error} and the compensator formula
used immediately before \eqref{eq:proof_tail_centering} give
\[
\begin{aligned}
	|\mu_{i,j,n}|
	&\le|\E_{j-1}D_{i,j,n}|+|\E_{j-1}L_{i,j,n}|\\
	&\lesssim h_n\varepsilon_{1,n}R_{j-1}
	+|\kappa_{i,n}|\,
	\E_{j-1}\!\int_{t_{j-1}}^{t_j}
	\mathbf1_{\{Z_{s-}=i\}}\,ds
	\lesssim h_nR_{j-1}.
\end{aligned}
\]
Lemma~\ref{lem:iv_drift_moments_main} gives
\[
	|\E_{j-1}[I_{n,j}U_j^\star]|
	\lesssim h_n\delta_{1,n}R_{j-1},
	\qquad
	\delta_{1,n}:=b_n^{\rm det},
	\qquad \delta_{1,n}\to0.
\]
Hence
\begin{equation}
	\max_{i\in S}\left|
	\frac1{T_n}\sum_{j=1}^n
	\E_{j-1}\!\left[w_{j-1}\mathcal R_{j,n}^{(1)}
	\widetilde H_{i,j,n}\right]\right|=o_p(1).
	\label{eq:main_retained_linear_tail}
\end{equation}

\proofstep{Step 2: consistency}
Lemma~\ref{lem:iv_debiasing_main} gives
\[
	\max_{i\in S}
	|\widehat a_{i,n}^{\,2,\rm db}-a_i^2(\gamma^\star)|\cip0.
\]
Since \(T_{i,n}/T_n\cip\pi_i\), uniformly on \(\Theta_\gamma\),
\[
	\mathbb H_{\gamma,n}(\gamma)
	-\mathbb H_{\gamma,n}(\gamma^\star)
	=Q_\gamma(\gamma)+o_p(1),
\]
where
\[
	Q_\gamma(\gamma)
	=-\frac12\sum_{i=1}^m\pi_i\left[
	\log\frac{a_i^2(\gamma)}{a_i^2(\gamma^\star)}
	+\frac{a_i^2(\gamma^\star)}{a_i^2(\gamma)}-1
	\right].
\]
Thus \(\bar\gamma_n\cip\gamma^\star\). Moreover,
Lemma~\ref{lem:iv_drift_moments_main},
\eqref{eq:main_retained_linear_reduction}, and
Proposition~\ref{prop:uniform_ergodic_main}, including
\eqref{eq:uniform_retained_ergodic}, imply
\[
	\sup_{\alpha\in\Theta_\alpha}\left|
	\frac{\mathbb H_{\alpha,n}
	(\alpha,\bar\gamma_n)
	-\mathbb H_{\alpha,n}
	(\alpha^\star,\bar\gamma_n)}{T_n}
	-Q_\alpha(\alpha)\right|\cip0,
\]
where
\[
	Q_\alpha(\alpha)
	=-\frac12\int
	\frac{\{b(x,i,\alpha)-b(x,i,\alpha^\star)\}^2}
	{a_i^2(\gamma^\star)}\mu(dx,di).
\]
The tail-mean replacement is uniform because
\(\max_i|\widehat\kappa_{i,n}-\kappa_{i,n}|=O_p(T_n^{-1/2})\).
Assumption~\ref{ass:ident} and the argmax theorem now give
\(\bar\theta_n\cip\theta^\star\).

\proofstep{Step 3: augmented martingale array}
Define the common Brownian and detected-tail components by
\begin{align*}
	\xi_{n,j}^{W}
	&:=\frac1{\sqrt{T_n}}
	\frac{\partial_\alpha b(X_{t_{j-1}},Z_{j-1},\alpha^\star)}
	{a_{Z_{j-1}}(\gamma^\star)}\,\Delta_jW,\\
	\xi_{n,j}^{N,i}
	&:=\frac1{\sqrt{T_n}\pi_i}
	\int_{t_{j-1}}^{t_j}\!\!\int_{|z|>u_n}
	z\mathbf1_{\{Z_{s-}=i\}}\widetilde N_i(ds,dz),
	\qquad i\in S.
\end{align*}
Put
\[
	\xi_{n,j}^{Q,i}
	:=\frac{\Gamma_{i,j}a_i^2(\gamma^\star)}
	{\sqrt n\,\pi_i}
	\left\{\frac{(\Delta_jW)^2}{h_n}-1\right\},
	\qquad i\in S,
\]
\[
	\Xi_{n,j}
	:=\bigl((\xi_{n,j}^{W})^\top,
	(\xi_{n,j}^{Q,i}:i\in S),
	(\xi_{n,j}^{N,i}:i\in S)\bigr)^\top.
\]
Set
\[
\begin{aligned}
	\Sigma
	&:=\operatorname{diag}\left\{
	I_\alpha(\theta^\star),
	\operatorname{diag}\left(
	\frac{2a_i^4(\gamma^\star)}{\pi_i}:i\in S\right),
	\Sigma_\kappa\right\}.
\end{aligned}
\]

Every component of \(\Xi_{n,j}\) is an
\(\mathcal F_{t_j}\)-measurable martingale difference. For the quadratic
component, this follows from independence of the complete
regime path and the Brownian increment. Conditional Gaussian moments,
the compensated-Poisson isometry, and
Lemma~\ref{lem:moment_ergodic_main} give
\begin{align*}
	\sum_{j=1}^n\E_{j-1}
	[\xi_{n,j}^{W}(\xi_{n,j}^{W})^\top]
	&=\frac1n\sum_{j=1}^n
	\frac{\{\partial_\alpha b_{j-1}(\alpha^\star)\}^{\otimes2}}
	{a_{j-1}^2(\gamma^\star)}
	\cip I_\alpha(\theta^\star),\\
	\sum_{j=1}^n\E_{j-1}
	[\xi_{n,j}^{Q,i}\xi_{n,j}^{Q,k}]
	&=\frac{2a_i^4(\gamma^\star)}{n\pi_i^2}
	\mathbf1_{\{i=k\}}\sum_{j=1}^n\E_{j-1}\Gamma_{i,j}
	\cip\mathbf1_{\{i=k\}}
	\frac{2a_i^4(\gamma^\star)}{\pi_i},\\
	\sum_{j=1}^n\E_{j-1}
	[\xi_{n,j}^{N,i}\xi_{n,j}^{N,k}]
	&=\mathbf1_{\{i=k\}}
	\frac{\int_{|z|>u_n}z^2\nu_i^\star(dz)}
	{T_n\pi_i^2}
	\sum_{j=1}^n\E_{j-1}
	\int_{t_{j-1}}^{t_j}\mathbf1_{\{Z_{s-}=i\}}\,ds\\
	&\cip\mathbf1_{\{i=k\}}
	\frac1{\pi_i}\int_{\mathbb R}z^2\nu_i^\star(dz).
\end{align*}
For the last display, the one-step transition bound gives
\[
	\E_{j-1}\!\int_{t_{j-1}}^{t_j}
	\mathbf1_{\{Z_{s-}=i\}}\,ds
	=h_n\mathbf1_{\{Z_{j-1}=i\}}+O(h_n^2).
\]
For the quadratic block, it gives
\[
	\E_{j-1}\Gamma_{i,j}
	=\mathbf1_{\{Z_{j-1}=i\}}\{1+O(h_n)\}.
\]
The sampled limits then follow from
Lemma~\ref{lem:moment_ergodic_main}; the replacement of endpoint
exposure by no-switch exposure is justified by
Lemma~\ref{lem:endpoint_exposure_main}.

All off-diagonal covariance blocks vanish. First, the Gaussian third
moment gives
\[
	\E_{j-1}[\xi_{n,j}^{W}\xi_{n,j}^{Q,i}]=0,
	\qquad i\in S,
\]
because the expectation contains the factor
\[
	\E\!\left[
	\Delta_jW\left\{\frac{(\Delta_jW)^2}{h_n}-1\right\}\right]=0.
\]
Moreover, \(\Gamma_{i,j}\Gamma_{k,j}=0\) for \(i\ne k\), so the
off-diagonal quadratic blocks are zero.

For the remaining blocks, let
\(\mathcal Z_n:=\sigma(Z_s:0\le s\le T_n)\) and initially enlarge the
filtration to \(\mathcal G_t:=\mathcal F_t\vee\mathcal Z_n\). By the
independence of the regime path from \(W\) and the Poisson random
measures, the Brownian and compensated-Poisson integrals remain
square-integrable \(\mathcal G_t\)-martingales. Denote the interval
martingales whose terminal increments form the Brownian linear,
centered Brownian quadratic, and regime-\(i\) compensated-Poisson
components by \(M_j^W,M_j^Q,M_j^{N,i}\), respectively. Continuity of
\(M_j^W,M_j^Q\), pure discontinuity of \(M_j^{N,i}\), and independent
scattering for \(i\ne k\) give
\[
\begin{aligned}
	\langle M_j^W,M_j^{N,i}\rangle
	&=0,&
	\langle M_j^Q,M_j^{N,i}\rangle
	&=0,\\
	\langle M_j^{N,i},M_j^{N,k}\rangle
	&=0,\qquad i\ne k.
\end{aligned}
\]
The integration-by-parts formula therefore yields, for each displayed
pair \((a,b)\),
\[
\begin{aligned}
	\E\!\left[
	\Delta M_j^a\Delta M_j^b\mid\mathcal G_{t_{j-1}}\right]
	&=\E\!\left[
	\Delta\langle M_j^a,M_j^b\rangle
	\mid\mathcal G_{t_{j-1}}\right]=0.
\end{aligned}
\]
Taking conditional expectations again with respect to
\(\mathcal F_{t_{j-1}}\) proves that all Brownian--Poisson and
cross-regime Poisson covariance blocks are zero.

For the Lindeberg condition, conditional Gaussian moments and the
fourth-moment formula for compensated Poisson integrals yield
\[
	\sum_j\E_{j-1}\|\xi_{n,j}^{W}\|^4=O_p(n^{-1}),
\]
\[
	\sum_{i,j}\E_{j-1}|\xi_{n,j}^{Q,i}|^4=O_p(n^{-1}),
	\qquad
	\sum_j\E_{j-1}|\xi_{n,j}^{N,i}|^4
	=O_p(T_n^{-1}+n^{-1}).
\]
Indeed, if \(\Delta M_{i,j}\) is the unnormalized Poisson integral,
then
\[
	\E_{j-1}|\Delta M_{i,j}|^4
	\lesssim\left\{
	h_nR_{j-1}\int_{|z|>u_n}z^4\nu_i^\star(dz)
	+h_n^2R_{j-1}
	\left(\int_{|z|>u_n}z^2\nu_i^\star(dz)\right)^2\right\}.
\]
The vector dimension is fixed, so the preceding bounds imply the
conditional vector Lindeberg condition. The Cram\'er--Wold martingale
central limit theorem gives
\begin{equation}
	\sum_{j=1}^n\Xi_{n,j}
	\cil N(0,\Sigma).
	\label{eq:main_joint_base_clt}
\end{equation}

\proofstep{Step 4: score representation}
Define
\[
	C_\gamma(x_i:i\in S)
	:=\frac12\sum_{i=1}^m
	\pi_i\frac{\partial_\gamma\log a_i^2(\gamma^\star)}
	{a_i^2(\gamma^\star)}x_i,
\]
\[
	A:=
	\begin{pmatrix}I_{p_\alpha}&0&J(\theta^\star)\\
	0&C_\gamma&0\end{pmatrix}.
\]
Also set
\[
	S_n
	:=\begin{pmatrix}
	T_n^{-1/2}\partial_\alpha
	\mathbb H_{\alpha,n}
	(\alpha^\star,\gamma^\star)\\
	\sqrt n\,\partial_\gamma
	\mathbb H_{\gamma,n}(\gamma^\star)
	\end{pmatrix}.
\]

Equations~\eqref{eq:proof_tail_centering} and
\eqref{eq:main_global_detected_jump}, proved in
Proposition~\ref{prop:tail_functional}, together with
\(T_{i,n}/T_n\cip\pi_i\) from
Lemma~\ref{lem:endpoint_exposure_main}, establish
\begin{equation}
	\sqrt{T_n}(\widehat\kappa_{i,n}-\kappa_{i,n})
	=\sum_{j=1}^n\xi_{n,j}^{N,i}+o_p(1).
	\label{eq:main_score_kappa_representation}
\end{equation}
Indeed, centering the detected increments produces the compensated
Poisson integral, the sum of their conditional-mean remainders is
\(o_p(\sqrt{T_n})\), and \(T_{i,n}/T_n\cip\pi_i\).

Put \(e_{i,n}=\widehat\kappa_{i,n}-\kappa_{i,n}\). At the true
parameter, the feasible residual satisfies the exact identity
\[
	\widehat U_j=U_j^\star+h_ne_{Z_{j-1},n}.
\]

For the drift score, differentiation, conditional centering,
\eqref{eq:main_retained_linear_reduction}, and
\(\sqrt{T_n}b_n^{\rm det}\to0\) yield
\[
\begin{aligned}
	\frac1{\sqrt{T_n}}\partial_\alpha
	\mathbb H_{\alpha,n}
	(\alpha^\star,\gamma^\star)
	=\sum_{j=1}^n\xi_{n,j}^{W}
	+\sum_{i=1}^m
	\left\{\frac1n\sum_{j=1}^n
	I_{n,j}\mathbf1_{\{Z_{j-1}=i\}}
	\frac{\partial_\alpha b_{j-1}(\alpha^\star)}
	{a_{j-1}^2(\gamma^\star)}\right\}
	\sqrt{T_n}e_{i,n}+o_p(1).
\end{aligned}
\]
By Lemmas~\ref{lem:moment_ergodic_main} and
\ref{lem:iv_drift_moments_main}, together with
\eqref{eq:main_retained_linear_reduction}, the braces converge to
\(J_i(\theta^\star)\). Equation
\eqref{eq:main_score_kappa_representation} then proves the first row of
the score representation.

Direct differentiation of the diffusion contrast gives
\[
	\sqrt n\,\partial_\gamma
	\mathbb H_{\gamma,n}(\gamma^\star)
	=\frac12\sum_{i=1}^m\widehat\pi_{i,n}
	\frac{\partial_\gamma\log a_i^2(\gamma^\star)}
	{a_i^2(\gamma^\star)}
	\sqrt n\{\widehat a_{i,n}^{\,2,\rm db}
	-a_i^2(\gamma^\star)\}.
\]
Lemma~\ref{lem:iv_debiasing_main} gives
\[
	\sqrt n\{\widehat a_{i,n}^{\,2,\rm db}
	-a_i^2(\gamma^\star)\}
	=\sum_{j=1}^n\xi_{n,j}^{Q,i}+o_p(1),
\]
and \(\widehat\pi_{i,n}\cip\pi_i\). This proves the second row of the
score representation. Hence
\begin{equation}
	S_n
	=A\sum_{j=1}^n\Xi_{n,j}+o_p(1).
	\label{eq:main_joint_score_representation}
\end{equation}
Since
\(A\Sigma A^\top=\Omega(\theta^\star)\),
\eqref{eq:main_joint_base_clt},
\eqref{eq:main_joint_score_representation}, and the continuous mapping
theorem give
\[
	S_n\cil N\{0,\Omega(\theta^\star)\}.
\]

\proofstep{Step 5: local Hessian laws}
Let
\[
	\widehat U_j(\alpha)
	:=\Delta_jX-h_n\{b_{j-1}(\alpha)
	-\widehat\kappa_{Z_{j-1},n}\},
	\qquad
	s_{j-1}(\gamma):=\partial_\gamma\log a_{j-1}^2(\gamma).
\]
Direct differentiation of the drift contrast gives
\begin{align}
	-\frac1{T_n}\partial_\alpha^2\mathbb H_{\alpha,n}(\alpha,\gamma)
	={}&\frac1n\sum_jI_{n,j}
	\frac{\{\partial_\alpha b_{j-1}(\alpha)\}^{\otimes2}}
	{a_{j-1}^2(\gamma)}
	-\frac1{T_n}\sum_jI_{n,j}
	\frac{\widehat U_j(\alpha)\partial_\alpha^2b_{j-1}(\alpha)}
	{a_{j-1}^2(\gamma)},
	\label{eq:main_hessian_alpha}\\
	\frac1{\sqrt{nT_n}}\partial_{\alpha\gamma}^2
	\mathbb H_{\alpha,n}(\alpha,\gamma)
	={}&-\frac1{\sqrt{nT_n}}\sum_jI_{n,j}
	\frac{\widehat U_j(\alpha)\partial_\alpha b_{j-1}(\alpha)}
	{a_{j-1}^2(\gamma)}s_{j-1}(\gamma)^\top.
	\label{eq:main_hessian_cross}
\end{align}

The drift contrast has the derivatives in
\eqref{eq:main_hessian_alpha} and
\eqref{eq:main_hessian_cross}. Lemma~\ref{lem:iv_drift_moments_main}
and \eqref{eq:main_retained_linear_reduction} therefore give the
\(I_\alpha\) and mixed-block limits. Direct differentiation of the IV
diffusion contrast gives
\[
\begin{aligned}
	-\partial_\gamma^2\mathbb H_{\gamma,n}(\gamma)
	=\frac12\sum_i\widehat\pi_{i,n}
	\frac{\widehat a_{i,n}^{\,2,\rm db}}{a_i^2(\gamma)}
	\{\partial_\gamma\log a_i^2(\gamma)\}^{\otimes2}
	-\frac12\sum_i\widehat\pi_{i,n}
	\left\{\frac{\widehat a_{i,n}^{\,2,\rm db}}
	{a_i^2(\gamma)}-1\right\}
	\partial_\gamma^2\log a_i^2(\gamma).
\end{aligned}
\]
Uniformly for \(|\gamma-\gamma^\star|\le\delta_n\),
\(\widehat a_{i,n}^{\,2,\rm db}\cip a_i^2(\gamma^\star)\) and
\(\widehat\pi_{i,n}\cip\pi_i\). The first line converges to
\(I_\gamma(\theta^\star)\), and the second converges to zero.

The deterministic coefficient integrands in
\eqref{eq:main_hessian_alpha} and
\eqref{eq:main_hessian_cross} contain parameter derivatives of
\(b\) or \(a\) through order at most two. The first coordinate
derivatives required by Proposition~\ref{prop:uniform_ergodic_main}
therefore involve coefficient derivatives through order at most three,
exactly as provided by Assumption~\ref{ass:smooth}(iii), together with
the required \(x\)-derivatives and a common polynomial envelope. Thus,
for \(|\theta-\theta^\star|\le\delta_n\), the mean-value theorem,
Proposition~\ref{prop:uniform_ergodic_main}, and
\eqref{eq:uniform_retained_ergodic} control every deterministic
coefficient average, while
\eqref{eq:main_retained_l2_linear_iv} controls the centered residual
terms through conditional isometry. Thus, for every deterministic
\(\delta_n\downarrow0\), the preceding diagonal limits hold locally
uniformly and the normalized mixed Hessians are \(o_p(1)\).

\proofstep{Step 6: Taylor expansion}
Consistency and the local Hessian laws imply that the maximizers are
interior with probability tending to one and satisfy their first-order
conditions. Expanding the diffusion condition first gives
\[
	\sqrt n(\bar\gamma_n-\gamma^\star)=O_p(1).
\]
The mixed-Hessian law then yields
\[
	\frac1{\sqrt{T_n}}\left\{
	\partial_\alpha\mathbb H_{\alpha,n}
	(\alpha^\star,\bar\gamma_n)
	-\partial_\alpha\mathbb H_{\alpha,n}
	(\alpha^\star,\gamma^\star)\right\}=o_p(1).
\]
The two sequential expansions give
\[
	D_n(\bar\theta_n-\theta^\star)
	=I(\theta^\star)^{-1}S_n+o_p(1).
\]
Combining this identity with the score limit proves
\[
	D_n(\bar\theta_n-\theta^\star)
	\cil N\{0,V(\theta^\star)\}.
\]
\end{proof}
\subsection{Proof of Corollary~\ref{cor:studentization}}

\begin{proof}
\begin{darkgreenrevision}
The proof has two parts. First, we show that each feasible truncated
second-moment estimator is asymptotically equivalent to
\(\mathcal Q_{i,n}/T_{i,n}\). We replace detected jumps by ideal jumps,
remove the oracle drift correction, control the feasible residual, and apply
a Poisson law of large numbers. Second, we combine this result with the
limits for the information, exposure, and tail-coupling terms. This gives
covariance consistency, after which Slutsky's theorem yields the joint limit.
\end{darkgreenrevision}

Fix \(i\in S\), and work on
\(\{T_{i,n}\ge\pi_iT_n/2\}\). Its probability tends to one by
\eqref{eq:Tin_convergence}.

\proofstep{Step 1: detected jumps versus ideal large jumps}
Retain \(H_{i,j,n}\), \(L_{i,j,n}\), and
\(D_{i,j,n}=H_{i,j,n}-L_{i,j,n}\) from the proof of
Proposition~\ref{prop:tail_functional}, and define
\[
	Q_{i,j,n}
	:=\int_{t_{j-1}}^{t_j}\!\!\int_{|z|>u_n}
	z^2\mathbf1_{\{Z_{s-}=i\}}N_i(ds,dz).
\]
By additivity of Poisson integrals and
\eqref{eq:studentization_Q_definition},
\[
	\mathcal Q_{i,n}=\sum_{j=1}^nQ_{i,j,n},
	\qquad
	A_{i,j}g_n(\Delta_jX)=H_{i,j,n}^2.
\]
The detection calculation in the proof of
Proposition~\ref{prop:tail_functional}, specifically
\eqref{eq:main_global_detected_jump}, gives
\begin{equation}
	\label{eq:main_tail_second_D_sum}
	\frac1{T_n}\sum_{j=1}^nD_{i,j,n}^2\cip0.
\end{equation}

For \(r=1,2,4\), let
\(m_{r,i}(v)=\int_{|z|>v}|z|^r\nu_i^\star(dz)\). Uniformly in
\(i,n\), \(m_{2,i}(u_n)+m_{4,i}(u_n)\le C\); moreover,
\(m_{1,i}(u_n)\le C(1+u_n^{1-\beta})\). Put
\[
	\tau_{i,j,n}:=\int_{t_{j-1}}^{t_j}
	\mathbf1_{\{Z_{s-}=i\}}\,ds.
\]
Conditional on the complete regime path, enumerate the regime-\(i\)
jumps larger than \(u_n\) on \((t_{j-1},t_j]\) as
\(\zeta_{i,j,n,1},\ldots,\zeta_{i,j,n,K_{i,j,n}^{(u_n)}}\), where
\(K_{i,j,n}^{(u_n)}\) denotes their number. Pointwise,
\[
\begin{aligned}
	L_{i,j,n}
	&=\sum_{k=1}^{K_{i,j,n}^{(u_n)}}\zeta_{i,j,n,k},
	&Q_{i,j,n}
	&=\sum_{k=1}^{K_{i,j,n}^{(u_n)}}\zeta_{i,j,n,k}^2,\\
	L_{i,j,n}^2-Q_{i,j,n}
	&=2\sum_{1\le k<r\le K_{i,j,n}^{(u_n)}}
	\zeta_{i,j,n,k}\zeta_{i,j,n,r}.
\end{aligned}
\]
The conditional compound-Poisson second-moment formula and the
conditional two-point Campbell formula therefore give
\begin{equation}
	\label{eq:main_tail_second_multiple}
\begin{aligned}
	\E_{j-1}L_{i,j,n}^2
	&=\E_{j-1}\!\left[
	\tau_{i,j,n}m_{2,i}(u_n)
	+\tau_{i,j,n}^2\kappa_{i,n}^2\right]
	\lesssim h_n,\\
	\E_{j-1}|L_{i,j,n}^2-Q_{i,j,n}|
	&\le\E_{j-1}[\tau_{i,j,n}^2]m_{1,i}(u_n)^2
	\lesssim h_n^2m_{1,i}(u_n)^2.
\end{aligned}
\end{equation}
Thus the pair-product identity is used precisely to obtain the second
bound in \eqref{eq:main_tail_second_multiple}. Moreover,
\[
	h_nu_n^{2-2\beta}=h_n^{1-2\omega(\beta-1)}O(1)\to0
\]
by \(\beta<8/5\) and \(\omega<4/(8+\beta)\). Hence
\begin{equation}
	\label{eq:main_tail_second_L_sum}
	\frac1{T_n}\sum_jL_{i,j,n}^2=O_p(1),
	\qquad
	\frac1{T_n}\sum_j|L_{i,j,n}^2-Q_{i,j,n}|\cip0.
\end{equation}
Since \(H_{i,j,n}=L_{i,j,n}+D_{i,j,n}\), the
Cauchy--Schwarz inequality and
\eqref{eq:main_tail_second_D_sum}--\eqref{eq:main_tail_second_L_sum}
give
\begin{equation}
	\label{eq:main_tail_second_raw}
\begin{aligned}
	\frac1{T_n}\left|\sum_jH_{i,j,n}^2-\mathcal Q_{i,n}\right|
	&\le2\left(\frac1{T_n}\sum_jL_{i,j,n}^2\right)^{1/2}
	\left(\frac1{T_n}\sum_jD_{i,j,n}^2\right)^{1/2}\\
	&\quad+\frac1{T_n}\sum_jD_{i,j,n}^2
	+\frac1{T_n}\sum_j|L_{i,j,n}^2-Q_{i,j,n}|\cip0.
\end{aligned}
\end{equation}

\proofstep{Step 2: oracle drift correction}
For \(|y|\le u_n/2\), direct expansion of
\(g_n(x+y)-g_n(x)\) yields
\begin{equation}
	\label{eq:main_tail_second_shift}
\begin{aligned}
	|g_n(x+y)-g_n(x)|
	&\le C(|x||y|+y^2)\mathbf1_{\{|x|>u_n/2\}}
	+Cu_n^2\mathbf1_{\{u_n/2<|x|\le2u_n\}}.
\end{aligned}
\end{equation}
The zero-, one-, and multiple-jump decomposition from
\eqref{eq:main_iv_count_bounds}, applied at \(u_n/2\) and
\(2u_n\), gives
\begin{equation}
	\label{eq:main_tail_second_boundary_counts}
\begin{aligned}
	\frac1{T_n}\sum_jA_{i,j}|\Delta_jX|
	\mathbf1_{\{|\Delta_jX|>u_n/2\}}
	&=O_p(1+u_n^{1-\beta}),\\
	\frac{u_n^2}{T_n}\sum_jA_{i,j}
	\mathbf1_{\{u_n/2<|\Delta_jX|\le2u_n\}}&\cip0.
\end{aligned}
\end{equation}
For clarity, the one-jump and zero-jump terms are bounded, respectively, by
\[
	Ch_nm_{1,i}(u_n/4),\qquad
	Ch_n\{u_n^{1-\beta}+h_nu_n^{-1-\beta}+h_n/u_n\}R_{j-1}.
\]
After multiplication by \(u_n^2\), the annulus bound is
\[
	u_n^{2-\beta}+h_nu_n^{-\beta}+h_n,
\]
which vanishes. Transition and multiple-jump terms are of smaller order
under Assumption~\ref{ass:iv_sampling}.

Let
\[
	c_{i,j,n}^\star:=-h_n\{b(X_{t_{j-1}},i,\alpha^\star)
	-\kappa_{i,n}\},
	\qquad c_{i,n}^{\max}:=\max_{j\le n}|c_{i,j,n}^\star|.
\]
For any fixed \(q>4\), the uniform moment bound, a union bound, and
Markov's inequality give
\[
 \max_{j\le n}R_{j-1}=O_p(n^{1/q}).
\]
Since \(T_n=nh_n\), \(nh_n^2\to0\), and \(\sqrt {h_n}/u_n\to0\),
\begin{equation}
	\label{eq:main_tail_second_oracle_shift}
	\frac{c_{i,n}^{\max}}{\sqrt {h_n}}=o_p(1),
	\qquad \frac{c_{i,n}^{\max}}{u_n}=o_p(1).
\end{equation}
Apply \eqref{eq:main_tail_second_shift} with
\(x=\Delta_jX\), \(y=c_{i,j,n}^\star\), and use
\eqref{eq:main_tail_second_boundary_counts}--%
\eqref{eq:main_tail_second_oracle_shift}. This gives
\begin{equation}
	\label{eq:main_tail_second_oracle_replacement}
	\frac1{T_n}\sum_jA_{i,j}
	|g_n(U_j^\star)-g_n(\Delta_jX)|\cip0.
\end{equation}
Together with \eqref{eq:main_tail_second_raw} and
\(T_{i,n}/T_n\cip\pi_i\), this proves
\begin{equation}
	\label{eq:main_tail_second_oracle}
	M_{2,i,n}^\star-\frac{\mathcal Q_{i,n}}{T_{i,n}}\cip0.
\end{equation}

\proofstep{Step 3: feasible residual}
On \(\{A_{i,j}=1\}\), put
\[
\begin{aligned}
	\delta_{i,j,n}
	:=\widehat Y_{i,j}-U_j^\star
	=-h_n\{b(X_{t_{j-1}},i,\bar\alpha_n)
	-b(X_{t_{j-1}},i,\alpha^\star)\}
	+h_n(\widehat\kappa_{i,n}-\kappa_{i,n}).
\end{aligned}
\]
Theorem~\ref{thm:parametric_inference} and
Proposition~\ref{prop:tail_functional} imply
\[
	L_n:=|\bar\alpha_n-\alpha^\star|
	+\max_i|\widehat\kappa_{i,n}-\kappa_{i,n}|
	=O_p(T_n^{-1/2}).
\]
By the mean-value theorem,
\(|\delta_{i,j,n}|\le Ch_nL_nR_{j-1}\). Taking the same
\(q>4\) as above gives
\begin{equation}
	\label{eq:main_max_plugin_displacement}
	\max_{i,j}\frac{|\delta_{i,j,n}|}{\sqrt {h_n}}\cip0,
	\qquad
	\max_{i,j}\frac{|\delta_{i,j,n}|}{u_n}\cip0.
\end{equation}
On the event where the two maxima in \eqref{eq:main_max_plugin_displacement}
are at most \(1/2\), apply \eqref{eq:main_tail_second_shift} with
\(x=U_j^\star\) and \(y=\delta_{i,j,n}\). Equations
\eqref{eq:main_tail_second_boundary_counts},
\eqref{eq:main_tail_second_oracle_shift}, and
\eqref{eq:main_max_plugin_displacement} show that each term on the
right is \(o_p(T_n)\) after summation. Thus
\begin{equation}
	\label{eq:main_tail_second_feasible}
	\widehat M_{2,i,n}-M_{2,i,n}^\star\cip0.
\end{equation}

\proofstep{Step 4: Poisson law of large numbers}
Let
\[
	\Lambda_{i,n}:=\int_0^{T_n}\mathbf1_{\{Z_{s-}=i\}}ds,
	\qquad
	\widetilde{\mathcal Q}_{i,n}:=
	\int_0^{T_n}\!\!\int_{|z|>u_n}z^2
	\mathbf1_{\{Z_{s-}=i\}}\widetilde N_i(ds,dz).
\]
Then
\[
	\mathcal Q_{i,n}=\widetilde{\mathcal Q}_{i,n}
	+\Lambda_{i,n}m_{2,i}(u_n),
	\qquad
	\E|\widetilde{\mathcal Q}_{i,n}|^2
	\le CT_n.
\]
Hence \(\widetilde{\mathcal Q}_{i,n}/T_{i,n}\cip0\). Applying
\eqref{eq:ergodic_thm} to
\(f(x,k)=\mathbf1_{\{k=i\}}\), and using
\eqref{eq:Tin_convergence}, gives
\[
	\frac{\Lambda_{i,n}}{T_n}\cip\pi_i,
	\qquad
	\frac{T_{i,n}}{T_n}\cip\pi_i,
	\qquad
	\frac{\Lambda_{i,n}}{T_{i,n}}\cip1.
\]
The last convergence follows from Slutsky's theorem.
Moreover, since \(u_n\downarrow0\), the monotone convergence theorem
gives
\(m_{2,i}(u_n)\to\int_{\mathbb R}z^2\nu_i^\star(dz)\),
and therefore
\begin{equation}
	\label{eq:main_tail_second_poisson}
	\frac{\mathcal Q_{i,n}}{T_{i,n}}
	\cip\int_{\mathbb R}z^2\nu_i^\star(dz).
\end{equation}
The regime set is finite, so
\eqref{eq:main_tail_second_oracle},
\eqref{eq:main_tail_second_feasible}, and
\eqref{eq:main_tail_second_poisson} hold uniformly over \(i\). Hence
\begin{equation}
	\label{eq:main_tail_second_conclusion}
	\max_{i\in S}\left|
	\widehat M_{2,i,n}
	-\int_{\mathbb R}z^2\nu_i^\star(dz)\right|\cip0.
\end{equation}

\proofstep{Step 5: covariance}
The integrands defining \(\widehat I_{\alpha,n}\),
\(\widehat I_{\gamma,n}\), and \(\widehat J_n\) contain
coefficient parameter derivatives through order at most one. Their
first coordinate derivatives therefore use derivatives through order at
most two, and Assumption~\ref{ass:smooth}(iii) gives the common
polynomial envelope required by
Proposition~\ref{prop:uniform_ergodic_main}. In particular,
\eqref{eq:uniform_retention_removal} gives
\begin{align*}
	&\sup_{(\alpha,\gamma)\in\Theta}
	\left\|\frac1n\sum_{j=1}^n(1-I_{n,j})
	\frac{\{\partial_\alpha b_{j-1}(\alpha)\}^{\otimes2}}
	{a_{j-1}^2(\gamma)}\right\|
	\lesssim\frac1n\sum_{j=1}^n(1-I_{n,j})R_{j-1}\cip0,\\
	&\sup_{\gamma\in\Theta_\gamma}
	\left\|\frac1{2n}\sum_{j=1}^n(1-I_{n,j})
	\{\partial_\gamma\log a_{j-1}^2(\gamma)\}^{\otimes2}\right\|
	\lesssim\frac1n\sum_{j=1}^n(1-I_{n,j})R_{j-1}\cip0,\\
	&\max_{i\in S}\sup_{(\alpha,\gamma)\in\Theta}
	\left\|\frac1n\sum_{j=1}^n(1-I_{n,j})
	\mathbf1_{\{Z_{j-1}=i\}}
	\frac{\partial_\alpha b_{j-1}(\alpha)}{a_{j-1}^2(\gamma)}\right\|
	\lesssim\frac1n\sum_{j=1}^n(1-I_{n,j})R_{j-1}\cip0.
\end{align*}
Thus the retention indicators do not change any of the three uniform
limits. Applying the unretained part of
Proposition~\ref{prop:uniform_ergodic_main} to their scalar entries,
and then using Theorem~\ref{thm:parametric_inference} and the continuous
mapping theorem, gives
\[
	\widehat I_{\alpha,n}\cip I_\alpha(\theta^\star),
	\quad
	\widehat I_{\gamma,n}\cip I_\gamma(\theta^\star),
	\quad
	\widehat J_n\cip J(\theta^\star),
	\quad
	\widehat\pi_{i,n}\cip\pi_i.
\]
Equation \eqref{eq:main_tail_second_conclusion} now implies
\[
	\widehat\Sigma_{\kappa,n}\cip\Sigma_\kappa,
	\qquad
	\widehat V_n\cip V(\theta^\star).
\]
Assumption~\ref{ass:info_nondeg} makes the empirical information
matrices nonsingular with probability tending to one. Theorem
\ref{thm:parametric_inference} and Slutsky's theorem prove the displayed
studentized limit; the coordinatewise intervals follow.
\end{proof}
}
% END INLINED FILE: proof_strategy_main.tex
}

		\section{Simulation study}
		\label{sec:sim}
		\raggedbottom

		{
		% BEGIN INLINED FILE: simulation/revised_section.tex
	\subsection{Model}
\label{sec:sim_models}

We simulate a switching Ornstein--Uhlenbeck process:
\[
	dX_t
	=-\alpha_{Z_t}X_t\,dt
	+\gamma_{Z_t}\,dW_t
	+\sum_{i=1}^2\mathbf 1_{\{Z_{t-}=i\}}
	\int_{\mathbb R}z\,\widetilde N_i(dt,dz).
\]
The common structural parameters are
\[
	Q=\begin{pmatrix}-0.8&0.8\\0.5&-0.5\end{pmatrix},
	\qquad
	(\alpha_1,\alpha_2)=(1.2,0.7),
	\qquad
	(\gamma_1,\gamma_2)=(0.35,0.55).
\]
\begin{darkgreenrevision}
We generate the Markov chain from its exact exponential holding times. If a
transition occurs between two observation times, we split the interval at
the transition and advance \(X\) on each constant-regime segment.

In each replication, we compute the tail means, coefficient estimates, and
feasible covariance matrix from Sections~\ref{sec:estimation}
and~\ref{sec:studentization}. The study examines the two-step correction and
the extra rate restrictions caused by \(T_n\to\infty\).
\end{darkgreenrevision}

% BEGIN INLINED FILE: simulation/case_iv_main_results.tex
\subsection{Infinite-variation numerical study}
\label{sec:sim_case_iv}

\subsubsection{Design and rate audit}

The jump component is specified by the two-sided tempered-stable L\'evy
densities
\[
	s_i^\star(z)
	=\frac{c_i^+e^{-\lambda_i^+z}\mathbf1_{\{z>0\}}
	+c_i^-e^{-\lambda_i^-|z|}\mathbf1_{\{z<0\}}}
	{|z|^{1+\beta}},
	\qquad \beta=\frac54,
\]
with
\[
\begin{array}{c|ccc}
	i&(c_i^+,c_i^-)&\lambda_i^+&\lambda_i^-\\ \hline
	1&(0.10,0.10)&4&6\\
	2&(0.08,0.08)&3&5.
\end{array}
\]
\begin{darkgreenrevision}
The leading singularity is symmetric in each regime. Tempering adds the
lower-order signed asymmetry allowed by the standing specification in the
Introduction. We simulate the infinite-activity jump component with the
small-jump Gaussian approximation of \citet{AsmussenRosinski2001}; see also
\citet{CohenRosinski2007} for tempered-stable processes. We split each L\'evy
increment at \(\epsilon_n=0.05\sqrt{h_n}\). Jumps larger than \(\epsilon_n\)
form a finite-activity compound-Poisson component and are generated
individually. The compensated sum of jumps at or below \(\epsilon_n\) is
replaced by a centered Gaussian variable with variance
\end{darkgreenrevision}
\[
	h_n\int_{|z|\leq\epsilon_n}z^2s_i^\star(z)\,dz.
\]
The chain is initialized from its stationary distribution, and a burn-in of
10 time units is discarded.

We use
\[
	h_n=1.2\times10^{-5}
	\left(\frac{n}{2{,}500{,}000}\right)^{-0.85},
	\qquad
	u_n=1.3h_n^{32/77}.
\]
\begin{darkgreenrevision}
Table~\ref{tab:sim_iv_settings} gives the implemented settings. S2 contains
\(2.5\) million increments per path and uses \(M=50\) replications. The mesh
diagnostic S3 contains \(10\) million increments per path and uses \(M=10\).
We therefore treat S3 and the reported coverage proportions as descriptive.
\end{darkgreenrevision}

\begin{currentrevision}
\begin{table}[H]
	\centering
	\caption{Infinite-variation sampling designs and replication counts.}
	\label{tab:sim_iv_settings}
	\small
	% BEGIN INLINED FILE: simulation/results/tables/case_iv_sampling_settings.tex
\begin{tabular}{@{}lrrrrrr@{}}\toprule
Setting & $M$ & $n$ & $h_n$ & $T_n$ & $u_n$ & $\eta_n$\\
\midrule
S1 & 50 & 625000 & 0.000039 & 24.4 & 0.0191 & 0.1584\\
S2 & 50 & 2500000 & 0.000012 & 30.0 & 0.0117 & 0.1519\\
S3 & 10 & 10000000 & 0.000004 & 36.9 & 0.0072 & 0.1458\\
\bottomrule
\end{tabular}
% END INLINED FILE: simulation/results/tables/case_iv_sampling_settings.tex
\end{table}
\end{currentrevision}

For \(\beta=5/4\), the choice \(\omega=32/77\) maximizes the sharp remainder
exponent in Remark~\ref{rem:iv_span_interpretation}:
\[
	\delta_{\rm iv}^{\sharp}=\frac{23}{154}.
\]
Since \(h_n\asymp n^{-17/20}\) and \(T_n=nh_n\),
\[
	\sqrt{T_n}h_n^{\delta_{\rm iv}^{\sharp}}
	\asymp n^{-4/77}\longrightarrow0,
\]
and \(17/20>77/100\), as required by the corresponding polynomial-mesh
restriction. \begin{darkgreenrevision}Table~\ref{tab:sim_iv_rates} reports
this power proxy and the exact finite-sample bound \(\sqrt n\,\rho_n\) from
Assumption~\ref{ass:iv_sampling}. The proxy falls from \(1.084\) to \(0.938\),
and \(\sqrt n\,\rho_n\) falls from \(3.468\) to \(2.918\). Both diagnostics
improve across S1--S3, although the conservative direct bound remains
sizable. We use these values when interpreting the diffusion results below.
\end{darkgreenrevision}

\begin{table}[H]
	\centering
	\caption{Infinite-variation finite-sample rate diagnostics.}
	\label{tab:sim_iv_rates}
	\scriptsize
	\begin{adjustbox}{max width=\textwidth}
		% BEGIN INLINED FILE: simulation/results/tables/case_iv_rate_diagnostics.tex
\begin{tabular}{@{}lrrrrrrr@{}}\toprule
Setting & $nh_n^2$ & $u_n/\sqrt{h_n}$ & $u_n^{4-\beta}/h_n$ & $u_n^{(4+\beta)/2}/h_n$ & $\sqrt{T_n}u_n^{2-\beta}$ & $\sqrt{T_n}h_n^{\delta_{\rm iv}^{\sharp}}$ & $\sqrt n\,\rho_n$\\
\midrule
S1 & 0.0010 & 3.06 & 0.482 & 0.791 & 0.254 & 1.084 & 3.468\\
S2 & 0.0004 & 3.38 & 0.408 & 0.711 & 0.195 & 1.008 & 3.177\\
S3 & 0.0001 & 3.74 & 0.345 & 0.639 & 0.150 & 0.938 & 2.918\\
\bottomrule
\end{tabular}
% END INLINED FILE: simulation/results/tables/case_iv_rate_diagnostics.tex
	\end{adjustbox}
\end{table}

\subsubsection{Primary parametric results}

\begin{currentrevision}
For every replication, we compute
\eqref{eq:iv_C_debiased} for the nine fixed pairs
\((\zeta_1,\zeta_2)\in\{1.1,1.2,1.3\}^2\). Both denominator indicators use
the cutoff \(h_n\). We take the regime-wise median and maximize the diffusion
contrast over \(\Theta_\gamma=[0.1,1]^2\). This is the fixed-grid version in
Remark~\ref{rem:fixed_grid_median}. The drift step is unchanged.
\end{currentrevision}

\begin{currentrevision}
\begin{table}[H]
	\centering
	\caption{Infinite-variation parametric results under S2 (\(M=50\)).}
	\label{tab:sim_iv_parametric}
	\footnotesize
	\begin{adjustbox}{max width=\textwidth}
		% BEGIN INLINED FILE: simulation/results/tables/case_iv_parametric_S2.tex
\begin{tabular}{@{}llrrrrrrr@{}}\toprule
Setting & Parameter & Bias & SD & Mean SE & RMSE & Stud. mean & Stud. SD & Coverage\\
\midrule
S2 & \(\alpha_1\) & 0.0820 & 0.3052 & 0.3329 & 0.3131 & 0.126 & 0.922 & 1.000\\
S2 & \(\alpha_2\) & 0.0215 & 0.3339 & 0.2774 & 0.3313 & -0.136 & 1.194 & 0.900\\
S2 & \(\gamma_1\) & 0.0087 & 0.0079 & 0.0003 & 0.0117 & 31.224 & 30.009 & 0.020\\
S2 & \(\gamma_2\) & 0.0030 & 0.0029 & 0.0003 & 0.0042 & 9.350 & 9.147 & 0.040\\
\bottomrule
\end{tabular}
% END INLINED FILE: simulation/results/tables/case_iv_parametric_S2.tex
	\end{adjustbox}
\end{table}
\end{currentrevision}

\begin{darkgreenrevision}
In Table~\ref{tab:sim_iv_parametric}, the drift biases are \(0.0820\) and
\(0.0215\), both smaller than their across-replication standard deviations.
The studentized standard deviations are \(0.922\) and \(1.194\), and the
coverage proportions are \(1.000\) and \(0.900\). We treat these proportions
as descriptive because \(M=50\). The fixed-grid diffusion biases are
\(0.0087\) and \(0.0030\), with RMSEs of \(0.0117\) and \(0.0042\).

Table~\ref{tab:sim_iv_debias} compares four diffusion estimators. The raw TRV
uses \(\widehat C_{i,n}(u_n)\) without correction. The one-step estimator
uses \(\widehat C_{i,n}^{(1)}(u_n)\) with \(\zeta_1=1.2\). The fixed-pair
two-step estimator uses \(\widehat C_{i,n}^{(2)}(u_n)\) with
\((\zeta_1,\zeta_2)=(1.2,1.2)\). The grid-median estimator takes the median
of the nine two-step variance estimates from
\((\zeta_1,\zeta_2)\in\{1.1,1.2,1.3\}^2\). In this experiment,
\(a_i(\gamma)=\gamma_i\). The table therefore reports the corresponding
estimates of \(\gamma_i\) from the same diffusion contrast over
\(\Theta_\gamma=[0.1,1]^2\).
\end{darkgreenrevision}

\begin{currentrevision}
\begin{table}[H]
	\centering
	\caption{Diffusion correction comparison under S2 (\(M=50\)).}
	\label{tab:sim_iv_debias}
	\small
	\begin{adjustbox}{max width=\textwidth}
		% BEGIN INLINED FILE: simulation/results/tables/case_iv_debias_S2.tex
\begin{tabular}{@{}llrrrr@{}}\toprule
Regime & Estimator & Mean & Bias & SD & RMSE\\
\midrule
1 & Raw TRV & 0.3626 & 0.0126 & 0.0003 & 0.0126\\
1 & One-step (\(\zeta_1=1.2\)) & 0.3623 & 0.0123 & 0.0288 & 0.0310\\
1 & Two-step (\(1.2,1.2\)) & 0.3558 & 0.0058 & 0.0777 & 0.0772\\
1 & Grid-median two-step & 0.3587 & 0.0087 & 0.0079 & 0.0117\\
2 & Raw TRV & 0.5563 & 0.0063 & 0.0003 & 0.0063\\
2 & One-step (\(\zeta_1=1.2\)) & 0.5544 & 0.0044 & 0.0279 & 0.0280\\
2 & Two-step (\(1.2,1.2\)) & 0.5465 & -0.0035 & 0.0323 & 0.0322\\
2 & Grid-median two-step & 0.5530 & 0.0030 & 0.0029 & 0.0042\\
\bottomrule
\end{tabular}
% END INLINED FILE: simulation/results/tables/case_iv_debias_S2.tex
	\end{adjustbox}
\end{table}

\begin{darkgreenrevision}
The fixed-pair two-step estimator has RMSEs of \(0.0772\) and \(0.0322\).
The grid median lowers them to \(0.0117\) and \(0.0042\). It also improves on
the raw TRV RMSEs of \(0.0126\) and \(0.0063\), although the gain in regime 1
is modest. The grid therefore limits the dispersion caused by an unfavorable
correction pair without changing the two-step estimator.
\end{darkgreenrevision}
	\end{currentrevision}

\begin{darkgreenrevision}
Diffusion coverage is less accurate than point estimation. Under S2,
Table~\ref{tab:sim_iv_rates} gives \(\sqrt n\,\rho_n=3.177\) and
\(\sqrt{T_n}h_n^{\delta_{\rm iv}^{\sharp}}=1.008\). Higher-order terms in
the studentization are therefore still visible, so the diffusion intervals
need caution under this design.
\end{darkgreenrevision}

\subsubsection{\darkgreenrev{Tail-mean results and cutoff sensitivity}}

\begin{darkgreenrevision}
The signed tail-mean biases in Table~\ref{tab:sim_iv_tail} are \(-0.0111\)
and \(-0.0035\). They are small relative to the corresponding RMSEs of
\(0.0837\) and \(0.0616\).
\end{darkgreenrevision}

\begin{currentrevision}
\begin{table}[H]
	\centering
	\caption{Infinite-variation signed tail-mean estimation under S2 (\(M=50\)).}
	\label{tab:sim_iv_tail}
	\small
	% BEGIN INLINED FILE: simulation/results/tables/case_iv_tail_mean_S2.tex
\begin{tabular}{@{}lrrrr@{}}\toprule
Regime & Target $\kappa_{i,n}$ & Mean & Bias & RMSE\\
\midrule
1 & 0.0647 & 0.0536 & -0.0111 & 0.0837\\
2 & 0.0629 & 0.0594 & -0.0035 & 0.0616\\
\bottomrule
\end{tabular}
% END INLINED FILE: simulation/results/tables/case_iv_tail_mean_S2.tex
\end{table}
\end{currentrevision}

\begin{darkgreenrevision}
Table~\ref{tab:sim_iv_sensitivity} varies the cutoff constant on \(M=20\)
common-random-number S2 paths. It uses the fixed-grid median at every cutoff.
The drift RMSE changes little. The diffusion RMSE is \(0.0092\), \(0.0121\),
and \(0.0133\) as the cutoff factor rises from \(0.8\) to \(1.1\). This small
experiment does not identify a preferred constant. It does show that the
point-estimation gain is not limited to the baseline cutoff.
\end{darkgreenrevision}

\begin{currentrevision}
\begin{table}[H]
	\centering
	\caption{\darkgreenrev{Infinite-variation cutoff sensitivity under S2 (\(M=20\)).}}
	\label{tab:sim_iv_sensitivity}
	\small
	% BEGIN INLINED FILE: simulation/results/tables/case_iv_sensitivity_S2.tex
\begin{tabular}{@{}lrr@{}}\toprule
\darkgreenrev{Cutoff} & $\mathrm{RMSE}_{\alpha}$ & $\mathrm{RMSE}_{\gamma}$\\
\midrule
$0.8u_0$ & 0.4828 & 0.0092\\
$u_0$ & 0.4867 & 0.0121\\
$1.1u_0$ & 0.4875 & 0.0133\\
\bottomrule
\end{tabular}
% END INLINED FILE: simulation/results/tables/case_iv_sensitivity_S2.tex
\end{table}
\end{currentrevision}

\subsection{\darkgreenrev{Practical checks and numerical conclusion}}

\begin{darkgreenrevision}
For a fixed cutoff grid, all tail sums and truncated variations can be
computed in one pass through the increments. An application should report
regime exposures, detected-increment counts, safeguard activations, the share
of positive grid estimates, and sensitivity to nearby cutoffs.
Section~\ref{sec:empirical}
reports these checks for the cattle data. For a planned sampling design, the
rate diagnostics in Table~\ref{tab:sim_iv_rates} are also useful.

The fixed-grid median improves diffusion point estimation in this study.
Diffusion studentization stabilizes more slowly, so Gaussian intervals need
caution until the cutoff and rate diagnostics are stable. The theory assumes
observed regimes and does not model measurement noise. Latent regimes or
noisy observations require an additional filtering or noise-robust step.
\end{darkgreenrevision}

\FloatBarrier
% END INLINED FILE: simulation/case_iv_main_results.tex

\section{Empirical illustration}
\label{sec:empirical}

\begin{darkgreenrevision}
This section applies the method to public cattle sensor data. It shows how
the model, tuning checks, and estimates can be used for a long record in which
a separate sensor measures the regimes. We treat the analysis as an empirical
fit; we do not claim that every assumption holds exactly for these signals.
\end{darkgreenrevision}

\subsection{Data and model}

\begin{darkgreenrevision}
The Precision Beef data set contains synchronized sensor records for 18
cattle \citep{Pavlovic2021PrecisionBeef}. Each animal wore two devices. An
Afimilk collar recorded three-axis acceleration at about 10 Hz. A RumiWatch
halter classified jaw pressure at 10 Hz as eating, rumination, or other. We
use the collar's \(x\)-axis, measured in milli-\(g\), as \(X\). Thus the
regime label comes from another device and is not constructed from \(X\). We
combine eating and rumination into a ``jaw activity'' regime because both
labels record sustained jaw movement. All remaining labels form the
``other'' regime. Both regimes have substantial exposure.
\end{darkgreenrevision}

\begin{darkgreenrevision}
We use animal 05, the smallest complete collar--halter file pair in the
archive. This choice was made before the model was fitted. We match each
collar observation to the nearest halter time, with a maximum gap of 0.051
seconds. The matched record contains 4,486,227 observations over 127.50
hours. Its mean mesh is \(h=0.102313\) seconds, and every step differs from
this value by less than \(10^{-4}\) seconds. We therefore use the mean mesh in
the approximately equidistant formulas. We drop increments whose endpoint
labels differ. The remaining exposure is 56.72 hours in the other regime and
70.74 hours in jaw activity, with 1,337 observed regime changes.
\end{darkgreenrevision}

For this illustration, the coefficient model is
\begin{equation}
\label{eq:empirical_ou_model}
 b(x,i,\alpha)=\alpha_i(\mu_i-x),
 \qquad
 a_i(\gamma)=\gamma_i,
 \qquad i\in\{0,1\}.
\end{equation}
\begin{darkgreenrevision}
Here \(1/\alpha_i\) measures the adjustment time of the collar coordinate,
and \(\gamma_i\) measures its continuous local variability. The jump part
captures large abrupt movements. Figure~\ref{fig:precision_beef_path} shows
the first 45 minutes, averaged within one-second bins. The colors show the
matched halter labels.
\end{darkgreenrevision}

\begin{figure}[H]
 \centering
 \includegraphics[width=0.88\textwidth]{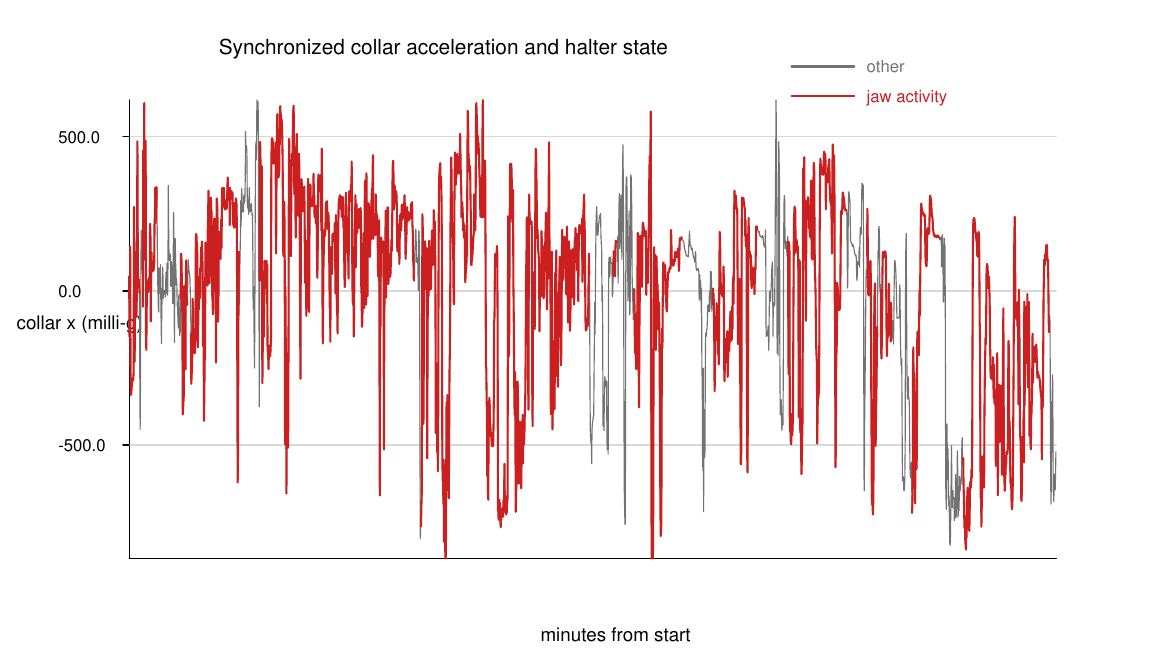}
 \caption{Collar acceleration and independently recorded jaw-activity labels
 for animal 05. Gray denotes other behavior and red denotes eating or
 rumination. The plotted collar values are one-second averages; estimation
 uses the original observations.}
 \label{fig:precision_beef_path}
\end{figure}

\subsection{Implementation and results}

We first compute the pooled robust increment scale
\[
 \widehat s=
 \frac{\operatorname{median}(|\Delta_jX-
 \operatorname{median}(\Delta X)|)}{0.67449}
 =11.861\ {\rm milli}\text{-}g.
\]
The cutoff candidates are
\(u\in\{2.5,3.0,3.5,4.0\}\widehat s\). For a simple tail check, let
\[
 \widehat\beta_i^{\rm tail}(u)
 =\frac{\log\{N_i(u)/N_i(1.5u)\}}{\log(1.5)},
 \qquad
 N_i(v)=\sum_j A_{i,j}\mathbf 1_{\{|\Delta_jX|>v\}}.
\]
\begin{darkgreenrevision}
The primary cutoff is \(u=3.5\widehat s=41.513\) milli-\(g\). It is the only
candidate for which both regime slopes, 1.527 and 1.565, lie in the
theoretical range \((1,8/5)\). This choice uses only the tail counts, not the
drift or diffusion estimates. At this cutoff, 69,927 increments (3.5\%) are
detected in the other regime and 249,702 (10.0\%) in the jaw-activity regime.
\end{darkgreenrevision}

\begin{darkgreenrevision}
We use the nine correction pairs
\((\zeta_1,\zeta_2)\in\{1.1,1.2,1.3\}^2\) and take the median as in
Remark~\ref{rem:fixed_grid_median}. The signed tail means and the drift
parameters are then estimated from \eqref{eq:kappa_estimator} and the
adaptive drift contrast. Table~\ref{tab:empirical_parameters} also reports
raw truncated-realized-variation (TRV) estimates, obtained without the two
debiasing steps. Brackets are percentile ranges from 500 resamples of
30-minute time blocks. They describe variation within this record. They are
not theorem-based confidence intervals because the cutoff is selected from a
finite grid.
\end{darkgreenrevision}

\begin{table}[H]
 \centering
 \caption{Parametric estimates for the Precision Beef record.}
 \label{tab:empirical_parameters}
 \footnotesize
 \begin{adjustbox}{max width=\textwidth}
  \begin{tabular}{@{}lrrrrrr@{}}\toprule
Regime & Exposure & $\widehat\alpha_i$ & $\widehat\mu_i$ & $\widehat\kappa_i$ & Raw $\widehat\gamma_i$ & Debiased $\widehat\gamma_i$\\
\midrule
Other & 56.72 h & 0.0027 [0.0022, 0.0033] & 18.7 [-24.1, 65.2] & 0.134 & 26.76 & 45.75 [8.95, 52.48]\\
Jaw activity & 70.74 h & 0.0133 [0.0119, 0.0148] & -14.2 [-41.7, 16.6] & -0.742 & 47.01 & 75.90 [65.63, 86.50]\\
\bottomrule
\end{tabular}

 \end{adjustbox}
 \parbox{0.97\textwidth}{\scriptsize\raggedright\emph{Note.}
 Brackets are descriptive 95\% block-bootstrap ranges.
 Units are seconds\(^{-1}\) for \(\widehat\alpha_i\), milli-\(g\) for
 \(\widehat\mu_i\), milli-\(g\)/second for \(\widehat\kappa_i\), and
 milli-\(g\)/\(\sqrt{\text{second}}\) for the two \(\widehat\gamma_i\)
 estimates.}
\end{table}

\begin{darkgreenrevision}
The adjustment estimate is larger during jaw activity. Its half-life,
\(\log(2)/\widehat\alpha_i\), is 51.9 seconds, compared with 253.3 seconds in
the other regime. The debiased diffusion estimate is also larger during jaw
activity: 75.90 rather than 45.75 milli-\(g\)/\(\sqrt{\text{second}}\). The
two descriptive ranges do not overlap, and the raw TRV estimates have the
same ordering. The estimated signed tail means have opposite signs. This
difference supports keeping the tail correction in the drift step.
\end{darkgreenrevision}

\begin{darkgreenrevision}
Table~\ref{tab:empirical_sensitivity} repeats the calculation over the full
cutoff grid. The superscript ``sp'' denotes a sign-constrained pilot version
of the same two correction steps. We use it only as a finite-sample check.
For each regime, six of the nine primary grid variances are positive, and
every denominator passes the safeguard in \eqref{eq:iv_debias_operator}.
The grid-median estimate is higher during jaw activity at three of the four
cutoffs. The raw TRV and pilot estimates have this ordering at all four. The
level of the debiased estimate changes with the cutoff, so the regime
contrast is more stable than its exact size. These checks show where the
method is informative in this application.
\end{darkgreenrevision}

\begin{table}[H]
 \centering
 \caption{Cutoff sensitivity of the empirical diffusion estimates.}
 \label{tab:empirical_sensitivity}
 \footnotesize
 \begin{adjustbox}{max width=\textwidth}
  \begin{tabular}{@{}lrrrrrr@{}}\toprule
Cutoff & $\widehat\beta_0^{\rm tail}$ & $\widehat\beta_1^{\rm tail}$ & $\widehat\gamma_0^{\rm db}$ & $\widehat\gamma_1^{\rm db}$ & $\widehat\gamma_0^{\rm sp}$ & $\widehat\gamma_1^{\rm sp}$\\
\midrule
$2.5\widehat s$ & 1.438 & 1.639 & 13.78 & 33.67 & 21.71 & 37.02\\
$3.0\widehat s$ & 1.665 & 1.810 & 43.43 & 29.28 & 26.91 & 42.98\\
$3.5\widehat s$ & 1.527 & 1.565 & 45.75 & 75.90 & 27.12 & 47.01\\
$4.0\widehat s$ & 1.711 & 1.731 & 32.45 & 58.11 & 15.73 & 52.04\\
\bottomrule
\end{tabular}

 \end{adjustbox}
 \parbox{0.97\textwidth}{\scriptsize\raggedright\emph{Note.}
 States 0 and 1 are other and jaw activity. ``db'' is the fixed-grid-median
 two-step estimator. ``sp'' is the sign/pilot check. Diffusion estimates are
 in milli-\(g\)/\(\sqrt{\text{second}}\).}
\end{table}

\begin{darkgreenrevision}
The record has the main features needed for a useful empirical fit. It is
long and finely sampled, its labels come from a separate sensor, and both
regimes have substantial exposure. It also shows a practical limit of the
method. The regime ordering is clear, but the exact correction remains
sensitive to the cutoff. An empirical analysis should therefore report the
raw estimate, the correction grid, and nearby cutoffs.
\end{darkgreenrevision}

\FloatBarrier

% END INLINED FILE: simulation/revised_section.tex
		}
			\flushbottom

		\appendix
	{
	% BEGIN INLINED FILE: appendix_proof_lemmas.tex
\section{Preliminary lemmas for the main proofs}
\label{sec:appendix_preliminary}

This appendix states the preliminary results used in the proofs of the
main theorems and corollaries. Their proofs are omitted here because they
are lengthy and repeatedly use the same conditional-moment, jump-count,
boundary-localization, and ergodic arguments. \begin{greenrevision}The
detailed infinite-variation proofs are given in the \textit{supplementary
material}.\end{greenrevision}

\subsection{Moment, ergodic, and exposure bounds}

\begin{lem}[Moment and sampled ergodic bounds]
	\label{lem:moment_ergodic_main}
	Suppose Assumptions~\ref{ass:smooth}, \ref{ass:levy}, and
	\ref{ass:ergodic} hold. Then, for every \(p\ge2\),
	\[
		\sup_{t\ge0}\E|X_t|^p<\infty,
	\]
	and, for \(0\le s<t\),
	\[
		\E[|X_t-X_s|^p\mid\mathcal F_s]
		\le C_p\{(t-s)^p+(t-s)\}(1+|X_s|^p).
	\]
	If \(f:\mathbb R\times S\to\mathbb R\) is continuously
	differentiable in \(x\), and \(f\) and \(\partial_xf\) have a common
	polynomial envelope, then
	\[
		\frac1n\sum_{j=1}^n
		f(X_{t_{j-1}},Z_{j-1})
		\cip\int_{\mathbb R\times S}f(x,i)\,\mu(dx,di).
	\]
\end{lem}
\begin{darkgreenrevision}
The proof adapts the moment and sampled-ergodic arguments of
\citet[Chapter~2]{yin2009hybrid} and \citet[Lemma~6.1]{Yuzhong2025} to the
switching jump diffusion. Kunita's inequality controls the compensated jump
martingale. We then compare sampled averages with continuous-time occupation
averages. The resulting bounds control the polynomial envelopes and one-step
errors used below.
\end{darkgreenrevision}

Let \(\mathcal O_\ell\subset\mathbb R^{d_\ell}\), \(1\le\ell\le L\),
be bounded open domains with Lipschitz boundaries, and put
\[
	\mathcal O_\vartheta:=\prod_{\ell=1}^L\mathcal O_\ell,
	\qquad
	\Theta_\vartheta:=\overline{\mathcal O_\vartheta}
	=\prod_{\ell=1}^L\overline{\mathcal O_\ell},
	\qquad
	d_\vartheta:=\sum_{\ell=1}^Ld_\ell.
\]

\begin{prop}[Uniform discrete ergodic law]
	\label{prop:uniform_ergodic_main}
	Suppose Assumptions~\ref{ass:smooth}, \ref{ass:levy}, and
	\ref{ass:ergodic} hold. Let
	\(f:\mathbb R\times S\times\Theta_\vartheta\to\mathbb R\).
	Enumerate the scalar coordinates of \(\vartheta\) and write
	\(\partial_k\) for the derivative in coordinate \(k\). Suppose that
	\(f\) is continuous on
	\(\mathbb R\times S\times\Theta_\vartheta\), that its first coordinate
	derivatives \(\partial_1f,\ldots,\partial_{d_\vartheta}f\) exist on
	\(\mathbb R\times S\times\mathcal O_\vartheta\) and extend continuously
	to \(\mathbb R\times S\times\Theta_\vartheta\), that the \(x\)-derivatives
	through order two of \(f\) and these extensions exist and are continuous,
	and that all these functions have a common polynomial envelope. Then
	\[
		\sup_{\vartheta\in\Theta_\vartheta}
		\left|
		\frac1n\sum_{j=1}^n
		f(X_{t_{j-1}},Z_{j-1},\vartheta)
		-\int_{\mathbb R\times S}
		f(x,i,\vartheta)\,\mu(dx,di)
		\right|\cip0.
	\]
	The same conclusion holds coordinatewise for finite-dimensional
	vector- and matrix-valued \(f\).
	If, in addition, Assumption~\ref{ass:iv_sampling} holds, then every predictable
	polynomial envelope \(R_{j-1}\) satisfies
	\begin{equation}
		\frac1n\sum_{j=1}^n(1-I_{n,j})R_{j-1}\cip0.
		\label{eq:uniform_retention_removal}
	\end{equation}
	Consequently,
	\begin{equation}
		\sup_{\vartheta\in\Theta_\vartheta}
		\left|
		\frac1n\sum_{j=1}^n I_{n,j}
		f(X_{t_{j-1}},Z_{j-1},\vartheta)
		-
		\frac1n\sum_{j=1}^n
		f(X_{t_{j-1}},Z_{j-1},\vartheta)
		\right|\cip0.
		\label{eq:uniform_retained_ergodic}
	\end{equation}
	\end{prop}
\begin{darkgreenrevision}
We prove this proposition from Lemma~\ref{lem:moment_ergodic_main}. The
sampled ergodic law is applied only to the integrand and its first coordinate
derivatives. A first-order Sobolev embedding on
\(\mathcal O_\vartheta\), followed by the continuous extension to
\(\Theta_\vartheta\), then gives uniformity. Thus the required smoothness does
not increase with the parameter dimension. The result gives uniform limits
for the contrast, Hessian, information, and studentization matrices.
Assumption~\ref{ass:iv_sampling} also allows us to remove the retention
indicator through a polynomial-envelope bound.
\end{darkgreenrevision}

Define
\[
	\Gamma_{i,j}
	:=\mathbf1_{\{Z_s=i\ \mathrm{for\ all}\ s\in[t_{j-1},t_j]\}},
	\qquad
	T_{i,n}^{\circ}:=h_n\sum_{j=1}^n\Gamma_{i,j}.
\]

\begin{lem}[Endpoint/no-switch equivalence and exposure law]
	\label{lem:endpoint_exposure_main}
	Uniformly over \(i\in S\) and \(j\le n\),
	\[
		\Pb_{j-1}(A_{i,j}\ne\Gamma_{i,j})\le Ch_n^2.
	\]
	Consequently, for
	\(\Omega_n^{\rm ns}:=\{A_{i,j}=\Gamma_{i,j}\text{ for all }
	i\in S,\ j\le n\}\),
	\[
		\Pb\{(\Omega_n^{\rm ns})^c\}\le Cnh_n^2=o(1),
	\]
	and, on \(\Omega_n^{\rm ns}\),
	\[
		A_{i,j}=\Gamma_{i,j}\quad(i\in S,j\le n),
		\qquad T_{i,n}=T_{i,n}^{\circ}\quad(i\in S).
	\]
	If Assumption~\ref{ass:ergodic} also holds, then, uniformly over
	\(i\in S\),
	\[
		\frac{T_{i,n}}{T_n}\cip\pi_i,
		\qquad
		\frac{T_{i,n}^{\circ}}{T_n}\cip\pi_i.
	\]
\end{lem}
\begin{darkgreenrevision}
The proof uses finite-state Markov-chain uniformization and the continuous
occupation law. Equal endpoint labels without a constant regime require a
leave-and-return path and hence at least two transitions. This gives the
order-\(h_n^2\) bound. The occupation law then identifies the limiting regime
exposures and justifies replacing endpoint selectors by their no-switch
counterparts.
\end{darkgreenrevision}

\subsection{Conditional moment and detected-tail bounds}

For the infinite-variation analysis, define
\[
	\mathcal V_n
	:=\{\zeta_1^k\zeta_2^\ell u_n:k,\ell\in\{0,1,2\}\},
	\qquad
	\mathcal W_n:=\{u_n,\zeta_2u_n,\zeta_2^2u_n\},
\]
\[
	G_v(y):=y^2\mathbf1_{\{|y|\le v\}},
	\qquad
	\Delta_\zeta F(v):=F(\zeta v)-F(v),
	\qquad
	\Delta_\zeta^2F(v):=F(\zeta^2v)-2F(\zeta v)+F(v),
\]
\[
	\eta_s(\zeta):=\zeta^s-1,
	\qquad
	\lambda_{s,t}(\zeta)
	:=\left\{\frac{\eta_s(\zeta)-\eta_t(\zeta)}
	{\eta_s(\zeta)}\right\}^2,
\]
\[
	Y_{i,j,n}^{\circ}:=\Delta_jX+h_n\kappa_{i,n},
	\qquad
	d_{i,1}:=\frac{2c_i}{2-\beta},
	\qquad
	d_{i,2}:=-\frac{c_i(\beta+1)(\beta+2)}{\beta}
	a_i^2(\gamma^\star).
\]
Fix \(\epsilon_0=(3\beta-2)/4\) and put
\[
	\rho_n
	:=h_n^2u_n^{-\beta-2}
	+h_n^{1/2}u_n^{1-\beta/2}
	+h_nu_n^{2-2\beta}
	+u_n^{2-\epsilon_0}+h_n,
\]
\[
	\mathcal B_{i,j,n}(v)
	:=\Gamma_{i,j}\left\{G_v(Y_{i,j,n}^{\circ})
	-a_i^2(\gamma^\star)(\Delta_jW)^2\right\}.
\]

For \(i\in S\) and \(j\le n\), put
\[
	H_{i,j,n}:=A_{i,j}\Delta_jX
	\mathbf1_{\{|\Delta_jX|>u_n\}},
	\qquad
	\mu_{i,j,n}:=\E_{j-1}H_{i,j,n}.
\]

Define
\[
\chi_n^{\rm dr}
	:=h_n\left(h_n^2u_n^{2-\beta}
	+h_nu_n^{4-\beta}+h_n^3\right)^{1/2}
	+h_n^2u_n^{3-\beta}+h_n^{5/2}+h_n^3.
\]

The fixed-regime truncated expansion is adapted from
\citet{BonieceFigueroaLopezHan2024}. Because their assumptions exclude zero one-sided first-order coefficients, the
supplement first proves that the required likelihood and remainder bounds
remain uniform when \(\ell_i^+=0\) or \(\ell_i^-=0\). It then localizes the
fixed-regime expansion by matching the L\'evy density near zero, treating the
remaining finite measure as an independent compound-Poisson component, and
transferring the frozen expansion through the state-dependent drift. The
conditional bridge supplies the no-switching reduction, uniformity over the
finite threshold grid, and summability over the growing horizon. These steps
yield the retained first, second, and fourth moments required for the IV
drift-score and studentization arguments.

\begin{lem}[Conditional truncated moments for the IV drift score]
	\label{lem:iv_drift_moments_main}
	Suppose the \greenrev{standing specification in the Introduction} and
	Assumptions~\ref{ass:smooth}, \ref{ass:levy}, and
	\ref{ass:iv_sampling} hold. Uniformly over \(j\le n\),
	\begin{align*}
		|\E_{j-1}[I_{n,j}U_j^\star]|
		&\le Ch_nb_n^{\rm det}R_{j-1},\\
		\left|\E_{j-1}[I_{n,j}(U_j^\star)^2]
		-h_na_{Z_{j-1}}^2(\gamma^\star)\right|
		&\le Ch_n(u_n^{2-\beta}+h_nu_n^{-\beta})R_{j-1},\\
		\left|\E_{j-1}[I_{n,j}(U_j^\star)^4]
		-3h_n^2a_{Z_{j-1}}^4(\gamma^\star)\right|
		&\le C\chi_n^{\rm dr}R_{j-1}.
	\end{align*}
	Moreover, \(\chi_n^{\rm dr}/h_n^2\to0\).
\end{lem}

\subsection{Uniform IV density and localized moments}

Let \(X^{i,x}\) be the nonswitching regime-\(i\) process
started from \(x\), driven by \(W\) and \(N_i\). Put
\[
	L_t^i:=\int_0^t\int_{\mathbb R}z\,\widetilde N_i(ds,dz),
\]
and define
\[
	\xi_{h_n}^{i,x}
	:=a_i(\gamma^\star)W_{h_n}+L_{h_n}^i
	+h_n\{b(x,i,\alpha^\star)+\kappa_{i,n}\},
	\qquad
	D_{h_n}^{i,x}
	:=\int_0^{h_n}
	\{b(X_s^{i,x},i,\alpha^\star)-b(x,i,\alpha^\star)\}\,ds,
\]
\[
	Y_{h_n}^{i,x}:=X_{h_n}^{i,x}-x+h_n\kappa_{i,n}
	=\xi_{h_n}^{i,x}+D_{h_n}^{i,x}.
\]
Let \(p_{i,h_n}^{\mathrm{fr},x}\) and
\(p_{i,h_n}^{\mathrm{sd},x}\) denote the densities of
\(\xi_{h_n}^{i,x}\) and \(Y_{h_n}^{i,x}\), respectively.

\begin{prop}[Uniform IV off-diagonal densities and boundary estimates]
	\label{prop:iv_off_diagonal_density}
	Suppose the \greenrev{standing specification in the Introduction} and
	Assumptions~\ref{ass:smooth}, \ref{ass:levy}, and
	\ref{ass:iv_sampling} hold. The densities
	\(p_{i,h_n}^{\mathrm{fr},x}\) and
	\(p_{i,h_n}^{\mathrm{sd},x}\) belong to \(C^1(\mathbb R)\).
	Uniformly over \(i\in S\), \(x\in\mathbb R\),
	\(v\in\mathcal V_n\), and \(|y|\in[v/2,3v/2]\),
	\begin{align*}
		p_{i,h_n}^{\mathrm{fr},x}(y)
		+p_{i,h_n}^{\mathrm{sd},x}(y)
		&\le CR(x)\left\{
		h_n|y|^{-1-\beta}
		+h_n^{-1/2}e^{-y^2/(Ch_n)}\right\},\\
		|\partial_y p_{i,h_n}^{\mathrm{fr},x}(y)|
		+|\partial_y p_{i,h_n}^{\mathrm{sd},x}(y)|
		&\le CR(x)\left\{
		h_n^{1/2}|y|^{-1-\beta}
		+h_n^{-1}e^{-y^2/(Ch_n)}\right\}.
	\end{align*}
	For \(0<r\le v/2\),
	\[
		\Pb\!\left(\left||\xi_{h_n}^{i,x}|-v\right|\le r\right)
		+\Pb_x^i\!\left(\left||Y_{h_n}^{i,x}|-v\right|\le r\right)
		\le CR(x)\left\{
		h_nrv^{-1-\beta}
		+r h_n^{-1/2}e^{-v^2/(Ch_n)}\right\}.
	\]
\end{prop}

\begin{darkgreenrevision}
The proof combines the Gaussian density and gradient estimates of
\citet{MenozziPesceZhang2021} with the small-/large-jump decompositions of
\citet{figueroalopezhoudre2009small} and
\citet{FigueroaLopezLuoOuyang2014}. We make these estimates uniform on the
shrinking off-diagonal band and obtain a random boundary-strip bound. This
bound controls cutoff perturbations in jump detection, IV drift moments, and
the feasible--oracle replacement.
\end{darkgreenrevision}

\subsection{Two-step IV debiasing and denominator stability}

Set \(p=2-\beta\) and \(q=-\beta\).

\begin{darkgreenrevision}
The two rational corrections are adapted from
\citet{BonieceFigueroaLopezHan2024}. We add a deterministic perturbation
expansion for both random denominators and prove the feasible--oracle
replacement under observed regimes. The exposure law and a vector martingale
central limit theorem then give denominator stability and the joint Gaussian
representation used for the diffusion estimator.
\end{darkgreenrevision}

\begin{lem}[Two-step IV debiasing and its denominators]
	\label{lem:iv_debiasing_main}
	Suppose the \greenrev{standing specification in the Introduction} and
	Assumptions~\ref{ass:smooth}, \ref{ass:levy},
	\ref{ass:iv_sampling}, and \ref{ass:ergodic} hold. Uniformly over
	\(i\in S\),
	\begin{align*}
		\Delta_{\zeta_1}^2\widehat C_{i,n}(v)
		&=d_{i,1}(\zeta_1^{2-\beta}-1)^2
		v^{2-\beta}\{1+o_p(1)\},
		\qquad v\in\mathcal W_n,\\
		\Delta_{\zeta_2}^2\widehat C_{i,n}^{(1)}(u_n)
		&=\lambda_{p,q}(\zeta_1)d_{i,2}
		(\zeta_2^{-\beta}-1)^2
		h_nu_n^{-\beta}\{1+o_p(1)\}.
	\end{align*}
	\[
		\Pb\!\left(
		\min\left\{
		\min_{i\in S}\min_{v\in\mathcal W_n}
		|\Delta_{\zeta_1}^2\widehat C_{i,n}(v)|,
		\min_{i\in S}|\Delta_{\zeta_2}^2
		\widehat C_{i,n}^{(1)}(u_n)|\right\}\ge h_n
		\right)\to1.
	\]
	There are random variables \((\mathcal Z_{i,n}:i\in S)\) such that
	\[
		\widehat a_{i,n}^{\,2,\rm db}
		=a_i^2(\gamma^\star)
		+\frac{\mathcal Z_{i,n}}{\sqrt n}+o_p(n^{-1/2}),
	\]
	\[
		\mathcal Z_{i,n}
		=\frac1{\sqrt n\,\pi_i}\sum_{j=1}^n
		\Gamma_{i,j}a_i^2(\gamma^\star)
		\left\{\frac{(\Delta_jW)^2}{h_n}-1\right\}+o_p(1),
	\]
	and
	\[
		(\mathcal Z_{i,n}:i\in S)^\top
		\cil N\!\left(
		0,\operatorname{diag}\left(
		\frac{2a_i^4(\gamma^\star)}{\pi_i}:i\in S\right)\right).
	\]
\end{lem}

% END INLINED FILE: appendix_proof_lemmas.tex
	}

	\section*{Statements and Declarations}

	\subsection*{Funding}
	No funding was received for conducting this study.

	\subsection*{Competing interests}
	The author has no relevant financial or non-financial interests to disclose.

	\subsection*{Data availability}
		The cattle sensor data analyzed in this study are publicly available in
		the Zenodo repository as \textit{Precision Beef---Animal Behaviour
		Classification} \citep{Pavlovic2021PrecisionBeef}, at DOI
		\href{https://doi.org/10.5281/zenodo.4064802}{10.5281/zenodo.4064802}.
		The empirical illustration
		uses the files \texttt{accel-05.csv} and \texttt{halter-05.csv}. The
		simulation and empirical-analysis code and generated numerical outputs are
		available from the author upon reasonable request.

		\subsection*{Supplementary information}
	The supplementary material following the references contains detailed
	proofs for the infinite-variation results.

	%%%%%
	%%%%%
	\bigskip

	% BEGIN INLINED FILE: main.bbl

	% END INLINED FILE: main.bbl

	%%%%%
	%%%%%

\clearpage
\phantomsection
\pdfbookmark[0]{Supplementary Material}{supplement-title}
\setcounter{section}{0}
\setcounter{subsection}{0}
\setcounter{subsubsection}{0}
\setcounter{equation}{0}
\setcounter{thm}{0}
\renewcommand{\thesection}{S\arabic{section}}
\renewcommand{\theHsection}{supp.\arabic{section}}
\renewcommand{\theHsubsection}{\theHsection.\arabic{subsection}}
\renewcommand{\theHsubsubsection}{\theHsubsection.\arabic{subsubsection}}
\raggedbottom
\setlength{\emergencystretch}{1em}

\begin{center}
{\Large\bfseries Supplementary Material for\\[0.4em]
\textit{Debiased parametric inference for regime-switching jump diffusions
with infinite variation jumps}}\\[0.8em]
{\large Yuzhong Cheng}\\[0.25em]
{\small Institute of Mathematics for Industry, Kyushu University,
744 Motooka, Fukuoka, Japan}\\
{\small \texttt{cheng.yuzhong.451@m.kyushu-u.ac.jp}}
\end{center}

\begin{quote}
\small\textbf{Abstract.}
This supplement gives the detailed proofs for the infinite-variation results.
It contains the rate calculations, common probabilistic tools, local density
bounds, drift-score expansion, and two-step debiasing proof.
\end{quote}
\medskip

\section{Conventions}
\label{sec:supp_notation}
\begingroup
\small

\begin{revision}
The model, observation scheme, and common notation are those of the main
paper. The notation below is used in the detailed infinite-variation proofs.
\end{revision}

Throughout, write \(Z_{j-1}:=Z_{t_{j-1}}\). Also write
\[
 \E_{j-1}[\cdot]:=\E[\cdot\mid\mathcal F_{t_{j-1}}],
 \qquad
 \Gamma_{i,j}:=\mathbf1_{\{Z_s=i\text{ for all }s\in[t_{j-1},t_j]\}},
 \qquad
 T_{i,n}^{\circ}:=h_n\sum_{j=1}^n\Gamma_{i,j},
\]
\[
\begin{aligned}
 I_{n,j}&:=\mathbf1_{\{|\Delta_jX|\le u_n\}},
 &Y_{i,j,n}^{\circ}&:=\Delta_jX+h_n\kappa_{i,n},\\
 \mathcal V_n&:=\{\zeta_1^k\zeta_2^\ell u_n:k,\ell\in\{0,1,2\}\},
 &\mathcal W_n&:=\{u_n,\zeta_2u_n,\zeta_2^2u_n\}.
\end{aligned}
\]
For an integrable random variable \(H\) and an event \(A\), use the
semicolon notation
\[
	\E_{j-1}[H;A]
	:=\E_{j-1}[H\mathbf1_A].
\]
For a function \(F:(0,\infty)\to\mathbb R\), put
\[
 \Delta_\zeta F(v):=F(\zeta v)-F(v),
 \qquad
 \Delta_\zeta^2F(v):=F(\zeta^2v)-2F(\zeta v)+F(v),
\]
\[
 \eta_s(\zeta):=\zeta^s-1,
 \qquad
 \lambda_{s,t}(\zeta)
 :=\left\{\frac{\eta_s(\zeta)-\eta_t(\zeta)}
 {\eta_s(\zeta)}\right\}^2.
\]
The symbol \(R_{j-1}\) denotes a generic predictable polynomial
envelope \(C(1+|X_{t_{j-1}}|^c)\), and \(R(x)\) denotes the corresponding
envelope for a deterministic initial state.  Its value may increase
from line to line.  All constants are uniform over the finite regime
set and the displayed finite threshold grids.

Assumption, theorem, and appendix-lemma numbers cited from the main paper
are imported from its auxiliary file and remain synchronized with it.

\endgroup

\begin{revision}
This supplement proves the increasing-horizon rate bounds and the common
exposure results. It also proves the local density bounds, the drift-score
expansion, and the two-step debiasing result. These arguments control the
remainders and denominators as the observation horizon grows.
\end{revision}

\section{Common perturbation and IV rate lemmas}

\begin{lem}[Quadratic-truncation perturbation]
	\label{lem:supp_truncation_perturbation}
	For \(v>0\), put
	\[
		G_v(y):=y^2\mathbf1_{\{|y|\le v\}},
		\qquad
		g_v(y):=y^2\mathbf1_{\{|y|>v\}}.
	\]
	There is a universal constant \(C\) such that, for all \(y,d\in
	\mathbb R\) with \(|d|\le v/2\),
	\begin{align}
		|G_v(y+d)-G_v(y)-2dy\mathbf1_{\{|y|\le v\}}|
		&\le Cd^2+Cv^2
		\mathbf1_{\{\lvert |y|-v\rvert\le |d|\}},
		\label{eq:supp_G_perturbation}\\
		|g_v(y+d)-g_v(y)-2dy\mathbf1_{\{|y|>v\}}|
		&\le Cd^2+Cv^2
		\mathbf1_{\{\lvert |y|-v\rvert\le |d|\}}.
		\label{eq:supp_g_perturbation}
	\end{align}
\end{lem}

\begin{proof}[Proof of Lemma~\ref{lem:supp_truncation_perturbation}]
	If \(y\) and \(y+d\) lie on the same side of the cutoff, the relevant
	indicator does not change and
	\((y+d)^2-y^2-2dy=d^2\).  If they lie on different sides, then
	\(\lvert |y|-v\rvert\le |d|\), while \(|y|\le3v/2\) and
	\(|y+d|\le2v\).  The absolute value of every quadratic term is then
	bounded by \(Cv^2\).  This proves \eqref{eq:supp_G_perturbation}.
	Since \(g_v(y)=y^2-G_v(y)\), subtracting
	\eqref{eq:supp_G_perturbation} from
	\((y+d)^2-y^2=2dy+d^2\) proves
	\eqref{eq:supp_g_perturbation}.
\end{proof}

For the Case IV exponent audit, set
\begin{equation}
	\epsilon_0:=\frac{3\beta-2}{4},
	\qquad
	\rho_n:=h_n^2u_n^{-\beta-2}
	+h_n^{1/2}u_n^{1-\beta/2}
	+h_nu_n^{2-2\beta}
	+u_n^{2-\epsilon_0}+h_n,
	\label{eq:supp_iv_rho_definition}
\end{equation}
and put
\begin{equation}
	\delta_{\rm IV}^{\sharp}:=\min\left\{
	\frac32-\omega(\beta+2),
	\ \omega\left(1-\frac\beta2\right),
	\ \frac12-2\omega(\beta-1),
	\ -\frac12+\frac{\omega(10-3\beta)}4,
	\ \frac12\right\}.
	\label{eq:supp_iv_sharp_delta}
\end{equation}
For the remainder \(\rho_n\) in
\eqref{eq:supp_iv_rho_definition}, substitution of
\(u_n=c_{\rm IV}h_n^\omega\) gives
\begin{equation}
\begin{aligned}
	\sqrt n\,\rho_n\asymp\sqrt{T_n}\big\{&
	h_n^{3/2-\omega(\beta+2)}
	+h_n^{\omega(1-\beta/2)}
	+h_n^{1/2-2\omega(\beta-1)}\\
	&+h_n^{-1/2+\omega(10-3\beta)/4}
	+h_n^{1/2}\big\}
	\asymp\sqrt{T_n}h_n^{\delta_{\rm IV}^{\sharp}}.
\end{aligned}
\label{eq:supp_iv_rho_audit}
\end{equation}
All five exponents in \eqref{eq:supp_iv_sharp_delta} are strictly
positive. The upper bound \(\omega<4/(8+\beta)\), together with
\(\beta<8/5\), controls the first and third exponents; the second is
positive because \(\beta<2\). The lower bound
\(\omega>(4-\beta)^{-1}\) controls the fourth exponent, and the last is
positive. Consequently, the direct IV-horizon condition
\(\sqrt n\,\rho_n\to0\) is equivalent to
\(\sqrt{T_n}h_n^{\delta_{\rm IV}^{\sharp}}\to0\).
It also implies the two increasing-horizon comparisons used in the main proof.
Indeed, the four normalized powers in
\(\sqrt{T_n}b_n^{\rm det}\) strictly dominate powers already present in
\eqref{eq:supp_iv_rho_audit}, because
\begin{align*}
	\omega(2-\beta)-\omega(1-\beta/2)
	&=\omega(1-\beta/2)>0,\\
	1+\omega(1-2\beta)-\{1/2-2\omega(\beta-1)\}
	&=1/2-\omega>0,\\
	1-\omega\beta-\omega(1-\beta/2)
	&=1-\omega(1+\beta/2)>0,\\
	1-1/2&>0.
\end{align*}
The third inequality follows from
\(\omega<4/(8+\beta)<2/(2+\beta)\). Hence
\(\sqrt{T_n}b_n^{\rm det}\to0\), and its first summand also gives
\(\sqrt{T_n}u_n^{2-\beta}\to0\).

The remaining consequences of the IV interval are collected in one
place below.  Subsequent IV proofs cite this audit instead of repeating
the substitution $u_n=c_{\rm IV}h_n^\omega$.
\begin{table}[H]
\centering
\small
\renewcommand{\arraystretch}{1.2}
\begin{adjustbox}{max width=\textwidth}
\begin{tabular}{@{}lll@{}}
\toprule
Required quantity & Power after substitution & Restriction used \\
\midrule
\(\delta_{\rm IV}^{\sharp}>0\)
& equation~\eqref{eq:supp_iv_sharp_delta}
& IV threshold interval and \(\beta<8/5\) \\
\(h_n/u_n^2\to0\)
& \(h_n^{1-2\omega}\)
& \(\omega<4/(8+\beta)<1/2\) \\
\(u_n^{4-\beta}/h_n\to0\)
& \(h_n^{\omega(4-\beta)-1}\)
& \(\omega>(4-\beta)^{-1}\) \\
\(u_n^{(4+\beta)/2}/h_n\to0\)
& \(h_n^{\omega(4+\beta)/2-1}\)
& \(\omega>2/(4+\beta)\) \\
\(\sqrt{h_n}/u_n\to0\)
& \(h_n^{1/2-\omega}\)
& \(\omega<1/2\) \\
\(h_nu_n^{-\beta}=o(u_n^{2-\beta})\)
& ratio \(h_n/u_n^2\)
& \(\omega<1/2\) \\
\(\sqrt{T_n}u_n^{2-\beta}\to0\)
& exponent \(\omega(2-\beta)>
\omega(1-\beta/2)\)
& direct IV-horizon condition \\
\(1-2\omega(\beta-1)>0\)
& \(h_nu_n^{2-2\beta}\to0\)
& \(\beta<8/5\), \(\omega<4/(8+\beta)\) \\
\(v_n^{\rm det}\to0\)
& \(O\{h_n^{\omega(2-\beta)}+h_n^{1+2\omega(1-\beta)}
+h_n^{1-\omega\beta}+h_n\}\)
& preceding IV interval consequences \\
\(\sqrt{T_n}b_n^{\rm det}\to0\)
& normalized exponents
\(\omega(2-\beta)\), \(1+\omega(1-2\beta)\),
\(1-\omega\beta\), and \(1\)
& direct IV-horizon condition and the exponent comparisons above \\
\(\rho_n/u_n^{2-\beta}\to0\)
& \(O\{h_n^{2-4\omega}
+h_n^{\{1-\omega(2-\beta)\}/2}
+h_n^{1-\omega\beta}
+h_n^{\omega(\beta-\epsilon_0)}
+h_n^{1-\omega(2-\beta)}\}\)
& IV interval and \(\epsilon_0=(3\beta-2)/4\) \\
\(\rho_n/(h_nu_n^{-\beta})\to0\)
& \(O\{h_n^{1-2\omega}
+h_n^{\{\omega(2+\beta)-1\}/2}
+h_n^{\omega(2-\beta)}
+h_n^{\omega(2+\beta-\epsilon_0)-1}
+h_n^{\omega\beta}\}\)
& IV interval and \(\epsilon_0=(3\beta-2)/4\) \\
\(\dfrac{u_n^{(4-\beta)/2}}{\sqrt{T_n}}
=o\{u_n^{2-\beta}\}\) and
\(o\{h_nu_n^{-\beta}\}\)
& ratios \(u_n^{\beta/2}/\sqrt{T_n}\) and
\(u_n^{(4+\beta)/2}/(h_n\sqrt{T_n})\)
& \(T_n\to\infty\) and \(u_n^{(4+\beta)/2}/h_n\to0\) \\
\(\sqrt n\,\rho_n\to0\)
& equation~\eqref{eq:supp_iv_rho_audit}
& direct IV-horizon condition \\
\bottomrule
\end{tabular}
\end{adjustbox}
\caption{Consequences of the IV threshold and increasing-horizon conditions.}
\label{tab:supp_iv_rate_audit}
\end{table}

\section{Moment, ergodic, and exposure bounds}
\label{sec:supp_moment_ergodic_exposure}

This section proves Appendix Lemma~\ref{lem:moment_ergodic_main},
Proposition~\ref{prop:uniform_ergodic_main}, and
Lemma~\ref{lem:endpoint_exposure_main}. The first result extends
standard switching-diffusion moment and sampled-ergodic arguments to the
compensated jump martingale, the second upgrades pointwise convergence
to parameter-uniform convergence by ordinary first-order Sobolev
embedding on the open product parameter domain and continuous extension
to its closure, and the
third combines finite-state-chain uniformization with the occupation
law. Together they provide the integrability, uniform deterministic
limits, retention-indicator removal, and regime-exposure replacements
used throughout the main proofs.

\begin{lem}[Moment and ergodic bounds]
	\label{lem:supp_moment_bounds}
	Suppose Assumptions~\ref{ass:smooth}, \ref{ass:levy}, and
	\ref{ass:ergodic} hold. Then:
	\begin{enumerate}[label=(\roman*),leftmargin=1.8em]
		\item for every \(p\ge2\),
		\(\sup_{t\ge0}\E|X_t|^p<\infty\);
		\item for every \(p\ge2\) and \(0\le s<t\),
		\[
			\E|X_t-X_s|^p
			\le C_p\{(t-s)^p+(t-s)\};
		\]
		\item for every \(p\ge2\) and \(0\le s<t\),
		\[
			\E[|X_t-X_s|^p\mid\mathcal F_s]
			\le C_p\{(t-s)^p+(t-s)\}(1+|X_s|^p);
		\]
		\item if \(f:\mathbb R\times S\to\mathbb R\) is continuously
		differentiable in \(x\), and \(f\) and \(\partial_xf\) have a
		common polynomial envelope, then
		\[
			\frac1n\sum_{j=1}^n
			f(X_{t_{j-1}},Z_{j-1})
			\cip\int_{\mathbb R\times S}f(x,i)\,\mu(dx,di).
		\]
	\end{enumerate}
\end{lem}

	\begin{proof}[Proof of Lemma~\ref{lem:supp_moment_bounds}]
		The argument is analogous to the switching-diffusion moment and
		sampling-ergodic arguments in
		\citet[Chapter~2]{yin2009hybrid} and
		\citet[Lemma~6.1]{Yuzhong2025}. We give the details because the present
		model contains the additional compensated jump martingale.
		Parts~(i), (iii), and (iv) prove
		Lemma~\ref{lem:moment_ergodic_main}; part~(ii) is recorded
		because it is repeatedly used to pass from conditional to
		unconditional one-step bounds.

		\proofstep{Step 1: uniform state moments}
		Let \(g_p(x,i):=1+|x|^p\). Assumption~\ref{ass:ergodic}(i) entails
		\(\mu(g_p)<\infty\) and
		\[
			P_tg_p(x,i)
			\le \mu(g_p)+C_pe^{-\varkappa_pt}g_p(x,i)
			\le C_pg_p(x,i),\qquad t\ge0.
		\]
		Integrating this inequality with respect to the law of
		\((X_0,Z_0)\) and using Assumption~\ref{ass:ergodic}(ii) proves (i).
		Moreover, by the Markov property, for \(u\ge s\),
		\begin{equation}
			\label{eq:conditional_moment_lemma}
			\E\!\left[1+|X_u|^p\mid\mathcal F_s\right]
			=P_{u-s}g_p(X_s,Z_s)
			\le C_p(1+|X_s|^p).
		\end{equation}

		\proofstep{Step 2: conditional increment moments}
		Put \(\delta:=t-s\) and decompose
		\[
			X_t-X_s=B_{s,t}+M_{s,t}+J_{s,t},
		\]
		where
		\[
			B_{s,t}:=\int_s^t b(X_u,Z_u,\alpha^\star)\,du,\qquad
			M_{s,t}:=\int_s^t a_{Z_u}(\gamma^\star)\,dW_u,
		\]
		and
		\[
			J_{s,t}:=\sum_{i=1}^m\int_s^t\!\!\int_{\mathbb R}
			z\mathbf 1_{\{Z_{u-}=i\}}\,\widetilde N_i(du,dz).
		\]
		By Hölder's inequality, the linear-growth condition in
		Assumption~\ref{ass:smooth}, and
		\eqref{eq:conditional_moment_lemma},
		\[
			\E\!\left[|B_{s,t}|^p\mid\mathcal F_s\right]
			\le \delta^{p-1}\int_s^t
			\E\!\left[|b(X_u,Z_u,\alpha^\star)|^p\mid\mathcal F_s\right]du
			\le C_p\delta^p(1+|X_s|^p).
		\]
		The Burkholder--Davis--Gundy inequality and
		\eqref{eq:conditional_moment_lemma} give
		\[
			\E\!\left[|M_{s,t}|^p\mid\mathcal F_s\right]
			\le C_p\E\!\left[
			\left(\int_s^t|a_{Z_u}(\gamma^\star)|^2du\right)^{p/2}
			\biggm|\mathcal F_s\right]
			\le C_p\delta^{p/2}(1+|X_s|^p).
		\]
		For every \(i\in S\) and \(p\ge2\),
		Assumption~\ref{ass:levy}(i) and the defining property of a L\'evy
		measure imply
			$\int_{\mathbb R}|z|^p\nu_i^\star(dz)<\infty$.
		Hence Kunita's inequality yields
		\[
			\E\!\left[|J_{s,t}|^p\mid\mathcal F_s\right]
			\le C_p\left(\delta^{p/2}+\delta\right).
		\]
		Since \(\delta^{p/2}\le\delta+\delta^p\), the elementary inequality
		\(|x+y+z|^p\le C_p(|x|^p+|y|^p+|z|^p)\) gives
		\[
			\E\!\left[|X_t-X_s|^p\mid\mathcal F_s\right]
			\le C_p\{\delta^p+\delta\}(1+|X_s|^p),
		\]
		which proves (iii). Taking expectations and applying (i) proves (ii).

		\proofstep{Step 3: sampled ergodic approximation}
		Set
		\[
			A_n:=\frac1n\sum_{j=1}^n
			f(X_{t_{j-1}},Z_{j-1}),\qquad
			\bar A_n:=\frac1{T_n}\int_0^{T_n}f(X_u,Z_u)\,du.
		\]
		For \(u\in[t_{j-1},t_j]\), the mean-value formula, Hölder's inequality,
		(i), and (ii) imply
		\[
			\E\!\left[
			\left|f(X_u,Z_u)-f(X_{t_{j-1}},Z_u)\right|
			\right]
			\le C|u-t_{j-1}|^{1/2}.
		\]
		Since \(S\) is finite and \(f\) has polynomial growth,
		\[
			\left|f(X_{t_{j-1}},Z_u)
			-f(X_{t_{j-1}},Z_{j-1})\right|
			\le C(1+|X_{t_{j-1}}|^C)
			\mathbf 1_{\{Z_u\ne Z_{j-1}\}}.
		\]
		The finite-state Markov-chain small-time transition expansion gives,
		uniformly in \(i\in S\),
		\[
			\mathbb P\!\left(Z_u\ne i\mid Z_{j-1}=i\right)
			\le C|u-t_{j-1}|.
		\]
		Consequently, by (i),
		\[
			\E\!\left[
			\left|f(X_{t_{j-1}},Z_u)
			-f(X_{t_{j-1}},Z_{j-1})\right|
			\right]
			\le C|u-t_{j-1}|.
		\]
		Combining the preceding estimates,
		\[
			\E\!\left[|A_n-\bar A_n|\right]
			\le \frac{C}{T_n}\sum_{j=1}^n
			\int_{t_{j-1}}^{t_j}
			\left\{|u-t_{j-1}|^{1/2}+|u-t_{j-1}|\right\}du
			\le C(h_n^{1/2}+h_n)\to0.
		\]
		Thus \(A_n-\bar A_n\to0\) in \(L^1\). The continuous-time
		ergodic theorem from Assumption~\ref{ass:ergodic} gives
		\(\bar A_n\cip\int_{\mathbb R\times S}f\,d\mu\), and (iv) follows.
	\end{proof}

	Let \(L\ge1\), let
	\(\mathcal O_\ell\subset\mathbb R^{d_\ell}\), \(1\le\ell\le L\),
	be bounded open domains with Lipschitz boundaries, and put
	\[
		\mathcal O_\vartheta:=\prod_{\ell=1}^L\mathcal O_\ell,
		\qquad
		\Theta_\vartheta:=\overline{\mathcal O_\vartheta}
		=\prod_{\ell=1}^L\overline{\mathcal O_\ell},
		\qquad
		d_\vartheta:=\sum_{\ell=1}^Ld_\ell.
	\]

	\begin{prop}[Uniform ergodic law]
		\label{prop:supp_uniform_ergodic}
		Suppose Assumptions~\ref{ass:smooth}, \ref{ass:levy}, and
		\ref{ass:ergodic} hold. Let
		\(f:\mathbb R\times S\times\Theta_\vartheta\to\mathbb R\).
		Enumerate the scalar coordinates of \(\vartheta\) and write
		\(\partial_k\) for the derivative in coordinate \(k\). Suppose that
		\(f\) is continuous on
		\(\mathbb R\times S\times\Theta_\vartheta\), that its first coordinate
		derivatives \(\partial_1f,\ldots,\partial_{d_\vartheta}f\) exist on
		\(\mathbb R\times S\times\mathcal O_\vartheta\) and extend continuously
		to \(\mathbb R\times S\times\Theta_\vartheta\), that the \(x\)-derivatives
		through order two of \(f\) and these extensions exist and are continuous,
		and that all these functions have a common polynomial envelope. Then
		\[
			\sup_{\vartheta\in\Theta_\vartheta}
			\left|
			\frac1n\sum_{j=1}^n
			f(X_{t_{j-1}},Z_{j-1},\vartheta)
			-\int_{\mathbb R\times S}
			f(x,i,\vartheta)\,\mu(dx,di)
			\right|\cip0.
		\]
		The same conclusion holds coordinatewise for finite-dimensional
		vector- and matrix-valued \(f\).
		If, in addition, \(h_n/u_n^2\to0\), then every
		predictable polynomial envelope \(R_{j-1}\) satisfies
		\begin{equation}
			\frac1n\sum_{j=1}^n(1-I_{n,j})R_{j-1}\cip0.
			\label{eq:supp_uniform_retention_removal}
		\end{equation}
		Consequently,
		\begin{equation}
			\sup_{\vartheta\in\Theta_\vartheta}
			\left|
			\frac1n\sum_{j=1}^n I_{n,j}
			f(X_{t_{j-1}},Z_{j-1},\vartheta)
			-
			\frac1n\sum_{j=1}^n
			f(X_{t_{j-1}},Z_{j-1},\vartheta)
			\right|\cip0.
			\label{eq:supp_uniform_retained_ergodic}
		\end{equation}
	\end{prop}

	\begin{proof}[Proof of Proposition~\ref{prop:supp_uniform_ergodic}]
		Proposition~\ref{prop:uniform_ergodic_main} is obtained by a
		first-order Sobolev upgrade of
		Lemma~\ref{lem:moment_ergodic_main}. Fix \(q>d_\vartheta\) and
		write \(f_k\) for the continuous extension of \(\partial_kf\) to
		\(\mathbb R\times S\times\Theta_\vartheta\), \(1\le k\le d_\vartheta\).
		Put \(f_0:=f\) and
		\[
			\mu f_k(\vartheta)
			:=\int_{\mathbb R\times S}f_k(x,i,\vartheta)\,\mu(dx,di),
			\qquad 0\le k\le d_\vartheta.
		\]
		The common polynomial envelope and
		Assumption~\ref{ass:ergodic} imply that every
		\(\mu f_k(\vartheta)\) is finite. The dominated convergence
		theorem gives
		\[
			\partial_k
			\int_{\mathbb R\times S}f(x,i,\vartheta)\,\mu(dx,di)
			=\mu f_k(\vartheta),
			\qquad \vartheta\in\mathcal O_\vartheta,
			\quad 1\le k\le d_\vartheta.
		\]

		Define
		\[
			A_n(\vartheta)
			:=
			\frac1n\sum_{j=1}^n
			f(X_{t_{j-1}},Z_{j-1},\vartheta)
			-\int_{\mathbb R\times S}
			f(x,i,\vartheta)\,\mu(dx,di).
		\]
		For \(1\le k\le d_\vartheta\), also put
		\[
			B_{n,k}(\vartheta)
			:=\frac1n\sum_{j=1}^n
			f_k(X_{t_{j-1}},Z_{j-1},\vartheta)
			-\mu f_k(\vartheta),
			\qquad \vartheta\in\Theta_\vartheta.
		\]
		For each fixed \(\vartheta\in\Theta_\vartheta\), part~(iv) of
		Lemma~\ref{lem:supp_moment_bounds}, applied successively to
		\(f_0(\cdot,\cdot,\vartheta)\) and
		\(f_k(\cdot,\cdot,\vartheta)\), yields
		\[
			A_n(\vartheta)\cip0,
			\qquad
			B_{n,k}(\vartheta)\cip0,
			\qquad 1\le k\le d_\vartheta.
		\]
		For \(\vartheta\in\mathcal O_\vartheta\), differentiation of the
		finite sum and the preceding dominated-differentiation identity give
		\(\partial_kA_n(\vartheta)=B_{n,k}(\vartheta)\).
		The asserted \(x\)-regularity verifies, in particular, the
		hypotheses of Lemma~\ref{lem:supp_moment_bounds}(iv).

		Let \(G(x):=C(1+|x|^C)\) be a common envelope for the
		\(f_k\)'s, including \(f_0\). Jensen's inequality and
		Lemma~\ref{lem:supp_moment_bounds}(i), used with any moment exponent
		\(r>\max\{2qC,2\}\), give
		\[
			\sup_{n\ge1}\sup_{\vartheta\in\Theta_\vartheta}
			\E|A_n(\vartheta)|^{2q}\lesssim1,
			\qquad
			\sup_{n\ge1}\sup_{\vartheta\in\Theta_\vartheta}
			\E|B_{n,k}(\vartheta)|^{2q}\lesssim1,
			\quad 1\le k\le d_\vartheta.
		\]
		Thus the \(q\)-th powers are uniformly integrable. The preceding
		pointwise convergence in probability implies
		\[
			\E|A_n(\vartheta)|^q\to0,
			\qquad
			\E|B_{n,k}(\vartheta)|^q\to0,
			\quad 1\le k\le d_\vartheta.
		\]
		The uniform \(2q\)-moment bound, Fubini's theorem, and the dominated
		convergence theorem now yield
		\[
			\E\!\left[
			\|A_n\|_{W^{1,q}(\mathcal O_\vartheta)}^q\right]
			=
			\int_{\mathcal O_\vartheta}
			\left\{\E|A_n(\vartheta)|^q
			+\sum_{k=1}^{d_\vartheta}
			\E|B_{n,k}(\vartheta)|^q\right\}d\vartheta
			\to0.
		\]

		The product \(\mathcal O_\vartheta\subset\mathbb R^{d_\vartheta}\)
		is a bounded open Lipschitz domain. Since \(q>d_\vartheta\), the
		Sobolev embedding theorem gives a continuous representative
		\(\widetilde A_n\in C(\Theta_\vartheta)\) satisfying
		\[
			\|\widetilde A_n\|_{C(\Theta_\vartheta)}
			\lesssim
			\|A_n\|_{W^{1,q}(\mathcal O_\vartheta)}.
		\]
		The definition of \(A_n\), the continuity assumptions on \(f\), and
		the dominated convergence theorem show that \(A_n\) is continuous on
		\(\Theta_\vartheta\). On \(\mathcal O_\vartheta\), it agrees almost
		everywhere with \(\widetilde A_n\); since both functions are continuous
		there, they agree everywhere on \(\mathcal O_\vartheta\), and hence on
		\(\Theta_\vartheta=\overline{\mathcal O_\vartheta}\). Consequently,
		taking \(q\)-th moments proves
		\[
			\E\!\left[
			\sup_{\vartheta\in\Theta_\vartheta}|A_n(\vartheta)|^q\right]
			\to0,
		\]
		and hence the asserted uniform convergence in probability.

		It remains to prove the retained-average conclusions. Put
		\(A_{n,j}:=\{|\Delta_jX|>u_n\}\), so that
		\(1-I_{n,j}=\mathbf1_{A_{n,j}}\). By Markov's inequality and
		Lemma~\ref{lem:supp_moment_bounds}(ii),
		\begin{equation}
			\sup_{j\le n}\Pb(A_{n,j})
			\le \frac1{u_n^2}\sup_{j\le n}\E|\Delta_jX|^2
			\lesssim\frac{h_n}{u_n^2}.
			\label{eq:supp_large_increment_probability}
		\end{equation}
		By the extra condition in the proposition, \(h_n/u_n^2\to0\).
		For any predictable polynomial
		envelope \(R_{j-1}\), the Cauchy--Schwarz inequality,
		\eqref{eq:supp_large_increment_probability}, and
		Lemma~\ref{lem:supp_moment_bounds}(i) yield
		\begin{align}
			\E\!\left[\frac1n\sum_{j=1}^n
			(1-I_{n,j})R_{j-1}\right]
			&\le \frac1n\sum_{j=1}^n
			\Pb(A_{n,j})^{1/2}\{\E R_{j-1}^2\}^{1/2}\nonumber\\
			&\lesssim \left(\frac{h_n}{u_n^2}\right)^{1/2}\longrightarrow0.
			\label{eq:supp_retention_envelope_l1}
		\end{align}
		This proves \eqref{eq:supp_uniform_retention_removal} in \(L^1\).
		Since the common polynomial envelope of \(f\) can be absorbed into
		\(R_{j-1}\),
		\begin{align*}
			&\sup_{\vartheta\in\Theta_\vartheta}
			\left|\frac1n\sum_{j=1}^n(I_{n,j}-1)
			f(X_{t_{j-1}},Z_{j-1},\vartheta)\right|\\
			&\qquad\le\frac1n\sum_{j=1}^n
			(1-I_{n,j})R_{j-1}\cip0,
		\end{align*}
		which proves \eqref{eq:supp_uniform_retained_ergodic}.

		Applying the scalar argument to each coordinate proves the vector- and
		matrix-valued versions. In particular, the Sobolev argument never
		differentiates the integrand in more than one parameter coordinate at
		a time; it requires no higher-order mixed derivative whose order
		depends on \(d_\vartheta\).
	\end{proof}

Define
\[
	\Omega_n^{\rm ns}
	:=\{A_{i,j}=\Gamma_{i,j}\text{ for every }i\in S,\ j\le n\}.
\]

\begin{lem}[Endpoint and exposure law]
	\label{lem:supp_endpoint_exposure}
	Uniformly over \(i\in S\) and \(j\le n\),
	\begin{equation}
		\Pb_{j-1}(A_{i,j}\ne\Gamma_{i,j})\le Ch_n^2.
		\label{eq:supp_endpoint_no_switch_local}
	\end{equation}
	Consequently,
	\begin{equation}
		\Pb\{(\Omega_n^{\rm ns})^c\}\le Cnh_n^2=o(1),
		\label{eq:supp_endpoint_no_switch_global}
	\end{equation}
	and, on \(\Omega_n^{\rm ns}\), simultaneously for every \(i\in S\),
	\begin{equation}
		A_{i,j}=\Gamma_{i,j}\quad(j\le n),
		\qquad T_{i,n}=T_{i,n}^{\circ}.
		\label{eq:supp_endpoint_exact_identification}
	\end{equation}
	If Assumption~\ref{ass:ergodic} also holds, then, uniformly over
	\(i\in S\),
	\begin{equation}
		\frac{T_{i,n}}{T_n}\cip\pi_i,
		\qquad
		\frac{T_{i,n}^{\circ}}{T_n}\cip\pi_i.
		\label{main-eq:Tin_convergence}
	\end{equation}
\end{lem}

\begin{proof}[Proof of Lemma~\ref{lem:supp_endpoint_exposure}]
The proof of Lemma~\ref{lem:endpoint_exposure_main} has two
separate components. Finite-state-chain uniformization shows that
endpoint equality differs from no switching only through a
leave-and-return interval, while the continuous occupation law
identifies the limiting regime exposure. The condition \(nh_n^2\to0\)
makes the first replacement simultaneous over the complete sample.

\(\Pb_{j-1}\) denotes conditional probability given
\(\mathcal F_{t_{j-1}}\).  Let
\(\bar q:=\max_{k\in S}(-q_{kk})\).  Conditional on
\(\mathcal F_{t_{j-1}}\), uniformization dominates the number
\(N_{j,n}^Z\) of transitions on \((t_{j-1},t_j]\) by a Poisson
random variable with mean \(\bar qh_n\).  The event
\(\{A_{i,j}\ne\Gamma_{i,j}\}\) requires the chain to start and end in
state \(i\), but to leave that state in between; it therefore requires
at least two transitions.  Hence
\[
	\Pb_{j-1}(A_{i,j}\ne\Gamma_{i,j})
	\le\Pb\{\operatorname{Pois}(\bar qh_n)\ge2\}
	\le Ch_n^2,
\]
which proves \eqref{eq:supp_endpoint_no_switch_local}.  The union bound
over the finite regime set and all observation intervals proves
\eqref{eq:supp_endpoint_no_switch_global};
\eqref{eq:supp_endpoint_exact_identification} follows directly from
the definition of \(\Omega_n^{\rm ns}\).

Fix \(i\in S\) and put
\[
	\Lambda_{i,n}:=\int_0^{T_n}\mathbf1_{\{Z_{s-}=i\}}\,ds.
\]
If \(N_{j,n}^Z=0\), then
\[
	h_nA_{i,j}
	=\int_{t_{j-1}}^{t_j}\mathbf1_{\{Z_{s-}=i\}}\,ds.
\]
On every interval the absolute difference between these two quantities is
at most \(h_n\). Consequently,
\[
	\left|T_{i,n}-\Lambda_{i,n}\right|
	\le h_n\sum_{j=1}^n\mathbf1_{\{N_{j,n}^Z\ge1\}}.
\]
More generally, for \(s\in[t_{j-1},t_j]\), let \(N_{j,s}^Z\) denote
the number of transitions on \((t_{j-1},s]\), so
\(N_{j,t_j}^Z=N_{j,n}^Z\). Uniformization gives
\begin{equation}
	\Pb(N_{j,s}^Z\ge1\mid\mathcal F_{t_{j-1}})
	\le \bar q(s-t_{j-1}).
	\label{eq:supp_one_transition_probability}
\end{equation}
Therefore
\[
	\E\left[\frac{|T_{i,n}-\Lambda_{i,n}|}{T_n}\right]
	\le \frac{h_n}{T_n}\sum_{j=1}^n
	\Pb(N_{j,n}^Z\ge1)
	\le \bar qh_n\to0.
\]
Thus \((T_{i,n}-\Lambda_{i,n})/T_n\to0\) in \(L^1\). Applying the
continuous-time ergodic theorem in
Assumption~\ref{ass:ergodic} to
\(f(x,k)=\mathbf1_{\{k=i\}}\) yields
\(\Lambda_{i,n}/T_n\cip\pi_i\), and hence
\(T_{i,n}/T_n\cip\pi_i\).  Finally,
\eqref{eq:supp_endpoint_no_switch_global} and
\eqref{eq:supp_endpoint_exact_identification} give the same convergence
for \(T_{i,n}^{\circ}\). Since \(S\) is finite, both convergences are
uniform over \(i\in S\).
\end{proof}

\section{Uniform IV off-diagonal densities and boundary estimates}
\label{sec:supp_iv_off_diagonal}

This section proves Appendix Proposition~\ref{prop:iv_off_diagonal_density}.
It combines Gaussian density and gradient bounds with a small-/large-jump
count decomposition, and then adapts them to the shrinking off-diagonal band
and random boundary strips generated by the IV threshold.  These estimates
are the input used later for detected-tail errors, the IV drift-score moments,
and feasible--oracle replacement.

Fix \(i\in S\) and \(x\in\mathbb R\). Let \(X^{i,x}\) be the
nonswitching regime-\(i\) process started from \(x\). Put
\[
	L_t^i:=\int_0^t\int_{\mathbb R}z\,\widetilde N_i(ds,dz),
\]
and define
\[
	\xi_{h_n}^{i,x}
	:=a_i(\gamma^\star)W_{h_n}+L_{h_n}^i
	+h_n\{b(x,i,\alpha^\star)+\kappa_{i,n}\},
	\qquad
	D_{h_n}^{i,x}
	:=\int_0^{h_n}
	\{b(X_s^{i,x},i,\alpha^\star)-b(x,i,\alpha^\star)\}\,ds,
\]
\[
	Y_{h_n}^{i,x}:=X_{h_n}^{i,x}-x+h_n\kappa_{i,n}
	=\xi_{h_n}^{i,x}+D_{h_n}^{i,x}.
\]
Let \(p_{i,h_n}^{\mathrm{fr},x}\) and
\(p_{i,h_n}^{\mathrm{sd},x}\) denote the densities of
\(\xi_{h_n}^{i,x}\) and \(Y_{h_n}^{i,x}\), respectively.

\begin{proof}[Proof of Proposition~\ref{prop:iv_off_diagonal_density}]
\sloppy
We use the Gaussian density and gradient estimates of
\citet{MenozziPesceZhang2021}. We combine them with the classical
small-/large-jump decomposition and conditioning on zero, one, or
multiple large jumps. These tools are used in the small-time density
expansions of \citet{figueroalopezhoudre2009small} for L\'evy processes
and \citet[Sections~2, 4, and 5]{FigueroaLopezLuoOuyang2014} for local
jump diffusions. We adapt these ingredients to the shrinking
off-diagonal band \(|y|\asymp v\), uniformly over the regimes,
thresholds, and initial states, and derive the boundary-strip estimate
required below.

Write \(b_i(z)=b(z,i,\alpha^\star)\),
\(a_i=a_i(\gamma^\star)\). The constants below are
uniform in \(i\) because \(S\) is finite. We first isolate the Gaussian
smoothing that remains after conditioning on the complete jump path.

\proofstep{Step 1: conditional diffusion kernel}
Let \(B:[0,h_n]\times\mathbb R\to\mathbb R\) be measurable in time and
globally Lipschitz in space with constant \(L\), and consider
\[
	dU_t=B(t,U_t)\,dt+a\,dW_t,\qquad U_0=u,
\]
where \(0<\underline a\le a\le\overline a\). Let \(\phi\) solve
\(\dot\phi_t=B(t,\phi_t)\), \(\phi_0=u\). The transition density
\(q_{h_n}^B(u,\cdot)\) exists, is continuously differentiable, and satisfies
\begin{equation}
	q_{h_n}^B(u,z)
	\le Ch_n^{-1/2}
	\exp\!\left\{-\frac{|z-\phi_{h_n}|^2}{Ch_n}\right\},
	\qquad
	|\partial_zq_{h_n}^B(u,z)|
	\le Ch_n^{-1}
	\exp\!\left\{-\frac{|z-\phi_{h_n}|^2}{Ch_n}\right\}.
	\label{eq:supp_conditional_diffusion_kernel}
\end{equation}
The constants depend only on \(L,\underline a,\overline a\) for
\(0<h_n\le h_0\), and not on \(B(t,0)\).

Put
\[
	V_t:=U_t-\phi_t,
	\qquad
	c_t(z):=B(t,z+\phi_t)-B(t,\phi_t).
\]
Subtracting the ordinary differential equation for \(\phi\) from the
stochastic differential equation for \(U\) gives
\begin{equation}
	dV_t=c_t(V_t)\,dt+a\,dW_t,\qquad V_0=0,
	\label{eq:supp_centered_conditional_diffusion}
\end{equation}
Writing \(\sigma(t,z):=a\), the coefficients in
\eqref{eq:supp_centered_conditional_diffusion} satisfy, for every
\(t,z,z'\),
\begin{equation}
\begin{aligned}
	&c_t(0)=0,\qquad
	|c_t(z)-c_t(z')|\le L|z-z'|,\qquad
	|c_t(z)|\le L|z|,\\
	&\kappa_0^{-1}|\theta|^2
	\le \sigma(t,z)^2|\theta|^2\le\kappa_0|\theta|^2,
	\qquad
	|\sigma(t,z)-\sigma(t,z')|
	+|\partial_z\sigma(t,z)|=0,
\end{aligned}
\label{eq:supp_mpz_hypotheses}
\end{equation}
where
\(\kappa_0:=1\vee\overline a^{\,2}\vee\underline a^{-2}\).
Thus the drift hypothesis of \citet[Theorem~1.2(i),(iv)]%
{MenozziPesceZhang2021} holds with spatial exponent one and constant
\(\kappa_1=1\vee L\); its diffusion H\"older and forward-gradient
hypotheses hold with any exponent in \((0,1)\), because the diffusion
coefficient is constant. The deterministic flow associated with \(c\)
and initial value zero is
\[
	\theta_{0,t}^c(0)=0,
\]
since \(c_t(0)=0\). Hence Theorem~1.2(i),(iv) and Remark~1.3 of the cited
paper give, with
\(g_\lambda(t,y):=t^{-1/2}\exp\{-\lambda |y|^2/t\}\),
\[
	q_t^c(0,y)\le Cg_\lambda(t,y),
	\qquad
	|\partial_yq_t^c(0,y)|
	\le Ct^{-1/2}g_\lambda(t,y).
\]
Equivalently, uniformly for \(0<t\le h_0\),
\begin{equation}
	q_t^c(0,y)
	\le Ct^{-1/2}\exp\!\left\{-\frac{|y|^2}{Ct}\right\},
	\qquad
	|\partial_yq_t^c(0,y)|
	\le Ct^{-1}\exp\!\left\{-\frac{|y|^2}{Ct}\right\}.
	\label{eq:supp_centered_heat_kernel_bound}
\end{equation}
Because \(U_t=V_t+\phi_t\), translation of the density gives
\[
	q_t^B(u,z)=q_t^c(0,z-\phi_t),
	\qquad
	\partial_zq_t^B(u,z)
	=\partial_yq_t^c(0,z-\phi_t).
\]
Substitution in \eqref{eq:supp_centered_heat_kernel_bound} proves
\eqref{eq:supp_conditional_diffusion_kernel}. The constants in
\eqref{eq:supp_mpz_hypotheses} depend on \(B\) only through \(L\); in
particular, the translation by \(\phi\) removes \(B(t,0)\). Therefore
\eqref{eq:supp_conditional_diffusion_kernel} remains uniform when \(B\)
is random and we condition on a realization with the same spatial
Lipschitz constant.

Condition now on a deterministic c\`adl\`ag path \(\ell\) of \(L^i\).
After the change of variable \(U_t=X_t^{i,x}-\ell_t\), set
\[
	B_\ell(t,u):=b_i(u+\ell_t),
	\qquad
	\varphi_t^\ell:=\Phi_t^\ell-\ell_t,
\]
where \(\Phi^\ell\) is defined by
\begin{equation}
	\Phi_t^\ell
	=x+\ell_t+\int_0^t b_i(\Phi_s^\ell)\,ds,
	\label{eq:supp_iv_deterministic_flow}
\end{equation}
so that
\(\dot\varphi_t^\ell=B_\ell(t,\varphi_t^\ell)\) and
\(\varphi_0^\ell=x\). Since
\[
	Y_{h_n}^{i,x}=U_{h_n}+\ell_{h_n}-x+h_n\kappa_{i,n},
\]
its conditional density is
\begin{equation}
\begin{aligned}
	p_{i,h_n}^{\mathrm{sd},x}(y\mid\ell)
	&=q_{h_n}^{B_\ell}
	\left(x,x+y-\ell_{h_n}-h_n\kappa_{i,n}\right),\\
	p_{i,h_n}^{\mathrm{sd},x}(y\mid\ell)
	&\le Ch_n^{-1/2}
	\exp\!\left\{-\frac{|y-m_{h_n}^{\mathrm{sd}}(\ell)|^2}
	{Ch_n}\right\},\\
	|\partial_yp_{i,h_n}^{\mathrm{sd},x}(y\mid\ell)|
	&\le Ch_n^{-1}
	\exp\!\left\{-\frac{|y-m_{h_n}^{\mathrm{sd}}(\ell)|^2}
	{Ch_n}\right\},
\end{aligned}
\label{eq:supp_conditional_sd_kernel}
\end{equation}
where
\[
	m_{h_n}^{\mathrm{sd}}(\ell)
	:=\Phi_{h_n}^\ell-x+h_n\kappa_{i,n}.
\]
For the frozen increment, conditioning on \(L^i=\ell\) gives the exact
Gaussian density
\begin{equation}
\begin{aligned}
	p_{i,h_n}^{\mathrm{fr},x}(y\mid\ell)
	&=\frac{1}{\sqrt{2\pi a_i^2h_n}}
	\exp\!\left\{-\frac{|y-m_{h_n}^{\mathrm{fr}}(\ell)|^2}
	{2a_i^2h_n}\right\},\\
	p_{i,h_n}^{\mathrm{fr},x}(y\mid\ell)
	&\le Ch_n^{-1/2}
	\exp\!\left\{-\frac{|y-m_{h_n}^{\mathrm{fr}}(\ell)|^2}
	{Ch_n}\right\},\\
	|\partial_yp_{i,h_n}^{\mathrm{fr},x}(y\mid\ell)|
	&\le Ch_n^{-1}
	\exp\!\left\{-\frac{|y-m_{h_n}^{\mathrm{fr}}(\ell)|^2}
	{Ch_n}\right\},
\end{aligned}
\label{eq:supp_conditional_frozen_kernel}
\end{equation}
where
\[
	m_{h_n}^{\mathrm{fr}}(\ell)
	:=\ell_{h_n}+h_n\{b_i(x)+\kappa_{i,n}\}.
\]

\proofstep{Step 2: L\'evy bounds and the small-jump aggregate}
Assumption~\ref{ass:levy} and the
\greenrev{standing IV specification of the main paper} imply, uniformly
for \(0<r\le1\),
\begin{equation}
	\nu_i^\star(|z|>r)\le Cr^{-\beta},\qquad
	s_i^\star(z)\le C|z|^{-1-\beta}\quad(0<|z|\le1),
	\qquad
	\left|\int_{|z|>r}z\,\nu_i^\star(dz)\right|\le C.
	\label{eq:supp_iv_levy_uniform_bounds}
\end{equation}
For the last bound, the integral over \(|z|>1\) is finite by
Assumption~\ref{ass:levy}; on \(r<|z|\le1\), the symmetric leading
terms cancel, while the two linear terms are bounded by
\(C\int_r^1z^{1-\beta}\,dz\) and the remainders by
\(C\int_r^1z^{\rho_q}\,dz\).

For \(v\in\mathcal V_n\), put
\(\delta_{h_n,v}=h_n v^{-\beta}\). Assumption~\ref{ass:iv_sampling}
gives, uniformly on the finite threshold grid,
\begin{equation}
	\delta_{h_n,v}\to0,\qquad \frac{\sqrt{h_n}}{v}\to0.
	\label{eq:supp_iv_scale_separation_density}
\end{equation}
Fix an integer \(N\ge1\). Choose an integer \(K_N\) sufficiently large
and split \(L^i\) at \(r_v=v/K_N\):
\[
	L_t^i=M_t^{(v)}
	+\int_0^t\!\int_{|z|>r_v}z\,N_i(ds,dz)
	-t\int_{|z|>r_v}z\,\nu_i^\star(dz),
\]
where \(M^{(v)}\) is the compensated martingale with jumps bounded by
\(r_v\). The exponential supermartingale inequality gives
\begin{equation}
	\Pb\!\left(\sup_{s\le h_n}|M_s^{(v)}|>v/16\right)
	\le C_N\delta_{h_n,v}^{\,N}.
	\label{eq:supp_iv_small_jump_aggregate}
\end{equation}
Indeed, for every \(\lambda>0\),
\[
	\int_{|z|\le r_v}
	(e^{\lambda z}-1-\lambda z)\nu_i^\star(dz)
	\le C\lambda^2e^{\lambda r_v}r_v^{2-\beta}.
\]
Apply the exponential supermartingale separately to \(M^{(v)}\) and
\(-M^{(v)}\), take
\(\lambda r_v=\tfrac12\log(\delta_{h_n,v}^{-1})\), and then choose
\(K_N\) so that the negative term \(-\lambda v/16\) dominates
\(N\log(\delta_{h_n,v}^{-1})\). This proves
\eqref{eq:supp_iv_small_jump_aggregate}. This is the step that controls
the aggregate of the infinite-activity small jumps on the
zero-large-jump event.

\proofstep{Step 3: zero, one, and multiple large jumps}
Let \(J_{h_n,v}\) be the number of jumps larger than \(r_v\). It is
Poisson with parameter
\[
	\lambda_{h_n,v}=h_n\nu_i^\star(|z|>r_v)\le C_N\delta_{h_n,v}.
\]
First suppose that
\begin{equation}
	h_n(1+|x|)\le v/(64C).
	\label{eq:supp_iv_moderate_start}
\end{equation}
On \(\{J_{h_n,v}=0\}\) and the complement of the event in
\eqref{eq:supp_iv_small_jump_aggregate},
\eqref{eq:supp_iv_levy_uniform_bounds}, Gronwall's inequality applied to
\eqref{eq:supp_iv_deterministic_flow}, and
\(\sup_{i,n}|\kappa_{i,n}|<\infty\) give
\begin{equation}
\begin{aligned}
	|m_{h_n}^{\mathrm{sd}}(L^i)|
	&=|\Phi_{h_n}^{L^i}-x+h_n\kappa_{i,n}|\le v/4,\\
	|m_{h_n}^{\mathrm{fr}}(L^i)|
	&=|L_{h_n}^i+h_n\{b_i(x)+\kappa_{i,n}\}|\le v/4.
\end{aligned}
\label{eq:supp_iv_zero_jump_centers}
\end{equation}
Therefore
\eqref{eq:supp_conditional_sd_kernel} and
\eqref{eq:supp_conditional_frozen_kernel} give, for
\(|y|\in[v/2,3v/2]\),
\begin{align}
	p_{0,h_n}(y)&\le
	Ch_n^{-1/2}e^{-y^2/(Ch_n)}
	+Ch_n^{-1/2}\delta_{h_n,v}^{\,N},
	\label{eq:supp_iv_zero_jump_density}\\
	|\partial_yp_{0,h_n}(y)|&\le
	Ch_n^{-1}e^{-y^2/(Ch_n)}
	+Ch_n^{-1}\delta_{h_n,v}^{\,N}.
	\label{eq:supp_iv_zero_jump_derivative}
\end{align}

On \(\{J_{h_n,v}=1\}\), let \(\tau\) and \(z\) be the jump time and
size, and let \(\ell\) denote the remaining compensated small-jump
path, including the deterministic compensator in
\eqref{eq:supp_iv_levy_uniform_bounds}. For the state-dependent
increment define
\[
	F_{\tau,\ell}(z)
	:=\Phi_{h_n}^{\ell+z\mathbf1_{[\tau,h_n]}}-x+h_n\kappa_{i,n}.
\]
Since \(b_i\) is continuously differentiable and globally Lipschitz,
the derivative of the flow with respect to the jump size is
\begin{equation}
	\partial_zF_{\tau,\ell}(z)
	=\exp\!\left\{\int_\tau^{h_n}
	\partial_xb_i(\Phi_s^{\ell+z\mathbf1_{[\tau,h_n]}})\,ds\right\},
	\qquad
	e^{-Lh_n}\le\partial_zF_{\tau,\ell}(z)\le e^{Lh_n}.
	\label{eq:supp_iv_flow_bilipschitz}
\end{equation}
For the frozen increment the corresponding map is \(z\mapsto z+c\), and
\[
	\partial_z(z+c)=1\in[e^{-Lh_n},e^{Lh_n}],
\]
which is the frozen analogue of \eqref{eq:supp_iv_flow_bilipschitz}.
On the good small-jump event and under
\eqref{eq:supp_iv_moderate_start},
\(|F_{\tau,\ell}(0)|\le v/4\). Hence, whenever
\(|F_{\tau,\ell}(z)-y|\le v/8\),
\(|z|\asymp v\). By \eqref{eq:supp_iv_levy_uniform_bounds}, the change
of variables \(w=F_{\tau,\ell}(z)\), and
\eqref{eq:supp_conditional_sd_kernel}--%
\eqref{eq:supp_conditional_frozen_kernel},
\begin{align*}
	\int_{|z|>r_v}h_n^{-1/2}
	e^{-|y-F_{\tau,\ell}(z)|^2/(Ch_n)}\nu_i^\star(dz)
	&\le Cv^{-1-\beta}
	+Ch_n^{-1/2}v^{-\beta}e^{-v^2/(Ch_n)},\\
	\int_{|z|>r_v}h_n^{-1}
	e^{-|y-F_{\tau,\ell}(z)|^2/(Ch_n)}\nu_i^\star(dz)
	&\le Ch_n^{-1/2}v^{-1-\beta}
	+Ch_n^{-1}v^{-\beta}e^{-v^2/(Ch_n)}.
\end{align*}
The terms with \(|z|>1\) in the off-center region are bounded in the
same way using \(\nu_i^\star(|z|>1)<\infty\). The bad small-jump event
is bounded by \eqref{eq:supp_iv_small_jump_aggregate} and the global
conditional bounds in \eqref{eq:supp_conditional_sd_kernel}--%
\eqref{eq:supp_conditional_frozen_kernel}. The compensation formula
therefore gives
\begin{align}
	p_{1,h_n}(y)&\le Ch_nv^{-1-\beta}
	+Ch_n^{-1/2}e^{-y^2/(Ch_n)}
	+Ch_n^{-1/2}\delta_{h_n,v}^{\,N},
	\label{eq:supp_iv_one_jump_density}\\
	|\partial_yp_{1,h_n}(y)|&\le Ch_n^{1/2}v^{-1-\beta}
	+Ch_n^{-1}e^{-y^2/(Ch_n)}
	+Ch_n^{-1}\delta_{h_n,v}^{\,N}.
	\label{eq:supp_iv_one_jump_derivative}
\end{align}

Finally, the second Poisson factorial moment and the global conditional
bounds in \eqref{eq:supp_conditional_sd_kernel}--%
\eqref{eq:supp_conditional_frozen_kernel} yield
\begin{align}
	p_{\ge2,h_n}(y)&\le
	Ch_n^{-1/2}\delta_{h_n,v}^{\,2},
	\label{eq:supp_iv_multiple_jump_density}\\
	|\partial_yp_{\ge2,h_n}(y)|&\le
	Ch_n^{-1}\delta_{h_n,v}^{\,2}.
	\label{eq:supp_iv_multiple_jump_derivative}
\end{align}
Because \(\omega(\beta-1)<1/2\),
\[
	h_n^{-1/2}\delta_{h_n,v}^{\,2}
	\le Ch_nv^{-1-\beta},\qquad
	h_n^{-1}\delta_{h_n,v}^{\,2}
	\le Ch_n^{1/2}v^{-1-\beta}.
\]
Since \(1-\omega\beta>0\), choose \(N\) in
\eqref{eq:supp_iv_small_jump_aggregate} large enough that
\[
	h_n^{-1/2}\delta_{h_n,v}^{\,N}
	\le Ch_nv^{-1-\beta},\qquad
	h_n^{-1}\delta_{h_n,v}^{\,N}
	\le Ch_n^{1/2}v^{-1-\beta}.
\]
Combining \eqref{eq:supp_iv_zero_jump_density}--%
\eqref{eq:supp_iv_multiple_jump_derivative} proves
the off-diagonal density and derivative bounds under
\eqref{eq:supp_iv_moderate_start}.

If \eqref{eq:supp_iv_moderate_start} fails, the global conditional
bounds in \eqref{eq:supp_conditional_sd_kernel}--%
\eqref{eq:supp_conditional_frozen_kernel} give
\(p\le Ch_n^{-1/2}\) and
\(|\partial_yp|\le Ch_n^{-1}\). Since \(v\asymp h_n^\omega\), a fixed
sufficiently large power of \(1+|x|\) absorbs these bounds into the
right-hand sides of the off-diagonal density and derivative bounds.
This is the envelope
\(R(x)\). The dominating bounds
\eqref{eq:supp_iv_zero_jump_density}--%
\eqref{eq:supp_iv_multiple_jump_derivative} also justify
differentiation under the conditional expectation and the Poisson
decomposition, proving that both unconditional densities belong to
\(C^1(\mathbb R)\).

For \(0<r\le v/2\), the two boundary strips are contained in
\(\{|y|\in[v/2,3v/2]\}\) and have total length \(4r\). Integrating
the off-diagonal density bound over these strips proves
the off-diagonal boundary bound.
\end{proof}

\section{Conditional truncated moments for the IV drift score}
\label{sec:supp_iv_drift_moments}

This section proves Appendix
Lemma~\ref{lem:iv_drift_moments_main} and records the localized IV
estimates later used in Appendix
Lemma~\ref{lem:iv_debiasing_main}. The fixed-regime truncated
expansion is adapted from \citet{BonieceFigueroaLopezHan2024}. Since the
cited result assumes nonzero one-sided first-order coefficients, we first
prove its extension for a globally admissible fixed-regime density with
possibly zero coefficients. Proposition~\ref{prop:supp_localized_iv_moments}
then supplies the separate local-to-global argument: it constructs an
auxiliary global density, quantifies the finite-activity difference from
the model measure with the tail-centering/compensator cancellation made
explicit, and transfers the frozen bounds through the state-dependent
drift using
Proposition~\ref{prop:iv_off_diagonal_density}. The final lemma
conditions on a no-switch interval and verifies summability over the
growing observation horizon.

Put
\begin{equation}
	d_{i,1}:=\frac{2c_i}{2-\beta},
	\qquad
	d_{i,2}:=-\frac{c_i(\beta+1)(\beta+2)}{\beta}
	a_i^2(\gamma^\star),
	\label{eq:supp_iv_bias_coefficients}
\end{equation}
and recall \(\epsilon_0\) and \(\rho_n\) from
\eqref{eq:supp_iv_rho_definition}.
Put
\[
	\mathcal B_{h_n}^{i,x}(v)
	:=G_v(Y_{h_n}^{i,x})-a_i^2(\gamma^\star)W_{h_n}^2 .
\]
For a family of functions
\(\check q_i:\mathbb R\setminus\{0\}\to[0,\infty)\), call
\(\{\check q_i:i\in S\}\) globally admissible if the following bounds
hold with constants independent of \(i\): \(\check q_i\) is bounded and
Borel measurable,
\[
	|\check q_i(z)-1-\ell_i^+z|
	\lesssim |z|^{\beta+\rho_q}\quad(z\downarrow0),
	\qquad
	|\check q_i(z)-1-\ell_i^-z|
	\lesssim |z|^{\beta+\rho_q}\quad(z\uparrow0),
\]
and, for every \(p>0\),
\[
	\sup_{i\in S}\int_{|z|>1}|z|^p\check q_i(z)
	|z|^{-1-\beta}\,dz<\infty.
\]
For such a family, put
\[
	\check\nu_i(dz):=c_i\check q_i(z)|z|^{-1-\beta}\,dz,
	\qquad
	\check\kappa_{i,n}:=\int_{u_n<|z|\le1}z\,\check\nu_i(dz),
\]
let \(\check L^i\) be the compensated L\'evy process with measure
\(\check\nu_i\), and define
\[
	\check\xi_{h_n}^{i,x}:=a_i(\gamma^\star)W_{h_n}
	+\check L_{h_n}^i
	+h_n\{b(x,i,\alpha^\star)+\check\kappa_{i,n}\},
	\qquad
	\check R_{h_n}^{i,x}:=\check\xi_{h_n}^{i,x}
	-a_i(\gamma^\star)W_{h_n}.
\]

\begin{lem}[Global fixed-regime IV bounds with zero coefficients]
	\label{lem:supp_iv_zero_coefficient_extension}
	Suppose the \greenrev{standing IV specification of the main paper} and
	Assumptions~\ref{ass:smooth}
	and~\ref{ass:iv_sampling} hold, and let
	\(\{\check q_i:i\in S\}\) be a globally admissible family. The
	coefficients \(\ell_i^+\) and \(\ell_i^-\) are allowed to vanish.
	Uniformly over \(i\in S\), \(x\in\mathbb R\), and \(v<w\) in
	\(\mathcal V_n\),
	\begin{align}
		&\left|
		\E_x^iG_v(\check\xi_{h_n}^{i,x})
		-h_n\{a_i^2(\gamma^\star)+d_{i,1}v^{2-\beta}
		+d_{i,2}h_nv^{-\beta}\}\right|\nonumber\\
		&\qquad\lesssim R(x)\left\{
		h_n^3v^{-\beta-2}+h_n^{3/2}v^{1-\beta/2}
		+h_n^2v^{2-2\beta}+h_nv^{2-\epsilon_0}+h_n^2\right\},
		\label{eq:supp_zero_coefficient_mean}\\
		&\E_x^i[G_v(\check\xi_{h_n}^{i,x})^2]
		\lesssim R(x)(h_n^2+h_nv^{4-\beta}),
		\label{eq:supp_zero_coefficient_fourth}\\
		&\E_x^i\!\left[
		\{G_w(\check\xi_{h_n}^{i,x})
		-G_v(\check\xi_{h_n}^{i,x})\}^2\right]
		\lesssim h_nv^{4-\beta}R(x).
		\label{eq:supp_zero_coefficient_annulus}
	\end{align}
	Moreover,
	\begin{align}
		&\E_x^i\!\left[W_{h_n}^2(\check R_{h_n}^{i,x})^2
		\mathbf1_{\{|\check\xi_{h_n}^{i,x}|\le v\}}\right]
		\lesssim R(x)(h_n^2v^{2-\beta}+h_n^3),
		\label{eq:supp_zero_coefficient_mixed}\\
		&\E_x^i\!\left[(\check R_{h_n}^{i,x})^4
		\mathbf1_{\{|\check\xi_{h_n}^{i,x}|\le v\}}\right]
		\lesssim R(x)(h_nv^{4-\beta}
		+h_n^2v^{4-2\beta}+h_n^4),
		\label{eq:supp_zero_coefficient_jump_fourth}\\
		&\E_x^i\!\left[W_{h_n}^4
		\mathbf1_{\{|\check\xi_{h_n}^{i,x}|>v\}}\right]
		\lesssim R(x)\left\{h_n^3v^{-\beta}
		+h_n^2e^{-v^2/(Ch_n)}\right\}.
		\label{eq:supp_zero_coefficient_brownian_tail}
	\end{align}
\end{lem}

\begin{proof}[Proof of Lemma~\ref{lem:supp_iv_zero_coefficient_extension}]
The proof reconstructs the local change of measure, audits its constants
uniformly as either one-sided coefficient tends to zero, and then uses a
nonzero approximation. The global admissibility conditions isolate this
argument from the local-to-global transfer proved below.

\proofstep{Step 1: localized change of measure}
Put
\[
	\ell_i(z):=\ell_i^+\mathbf1_{\{z>0\}}
	+\ell_i^-\mathbf1_{\{z<0\}}.
\]
Choose \(r_0\in(0,1/4)\) such that
\(1/2\le\check q_i(z)\le3/2\) for
\(i\in S\) and \(0<|z|\le2r_0\). On this neighborhood, put
\[
	\widetilde\nu_i(dz):=c_i|z|^{-1-\beta}\,dz,
	\qquad
	\phi_i(z):=-\log\check q_i(z),
	\qquad
	\psi_i(z):=\phi_i(z)+\ell_i(z)z.
\]
Taylor's formula for \(-\log(1+y)\) gives
\begin{equation}
\begin{aligned}
	|\psi_i(z)|&\lesssim |z|^2+|z|^{\beta+\rho_q},
	&|\phi_i(z)|&\lesssim |z|,\\
	\sup_{i\in S}\int_{0<|z|\le2r_0}
	\{ |\psi_i(z)|+|\phi_i(z)|^2\}\widetilde\nu_i(dz)&<\infty.
\end{aligned}
	\label{eq:supp_zero_coefficient_phi_bounds}
\end{equation}
Let \(\widetilde\Pb_i\) denote the reference law with compensator
\(dt\,\widetilde\nu_i(dz)\) on this neighborhood, leaving the
finite-activity part outside unchanged. Define
\[
	\eta_i:=\int_{0<|z|\le2r_0}
	\{e^{-\phi_i(z)}-1+\phi_i(z)\}\widetilde\nu_i(dz),
\]
\begin{equation}
	U_{i,t}:=\int_0^t\int_{0<|z|\le2r_0}
	\phi_i(z)\widetilde N_i^0(ds,dz)+t\eta_i.
	\label{eq:supp_zero_coefficient_likelihood}
\end{equation}
The exponential formula for Poisson random measures yields, for every
\(p\in\mathbb R\),
\[
	\left.\frac{d\Pb_i}{d\widetilde\Pb_i}\right|_{\mathcal F_t}
	=e^{-U_{i,t}},
	\qquad
	\widetilde\E_i e^{pU_{i,t}}
	=\exp\!\left[t\left\{p\eta_i+
	\int(e^{p\phi_i}-1-p\phi_i)\,d\widetilde\nu_i\right\}\right].
\]
Consequently, uniformly over \(i\in S\) and \(0<t\le1\),
\begin{equation}
\begin{aligned}
	\widetilde\E_i|U_{i,t}|^2&\lesssim t,
	&\widetilde\E_i|e^{-U_{i,t}}-1|^2&\lesssim t,\\
	\widetilde\E_i(e^{-U_{i,t}}-1+U_{i,t})&\lesssim t,
	&\sup_{p\in[-8,8]}\widetilde\E_i e^{pU_{i,t}}&\lesssim1.
\end{aligned}
	\label{eq:supp_zero_coefficient_likelihood_moments}
\end{equation}
For
\[
	S_{i,t}^{\pm}:=
	\int_0^t\int_{\{0<\pm z\le2r_0\}}
	z\,\widetilde N_i^0(ds,dz),
\]
equations \eqref{eq:supp_zero_coefficient_phi_bounds} and
\eqref{eq:supp_zero_coefficient_likelihood} give
\begin{equation}
	U_{i,t}=U_{i,t}^{\rm fv}
	-\ell_i^+S_{i,t}^+-\ell_i^-S_{i,t}^-,
	\qquad
	\widetilde\E_i|U_{i,t}^{\rm fv}|\lesssim t,
	\label{eq:supp_zero_coefficient_decomposition}
\end{equation}
where
\[
	U_{i,t}^{\rm fv}:=\int_0^t\int_{0<|z|\le2r_0}
	\psi_i(z)\widetilde N_i^0(ds,dz)+t\eta_i.
\]
Thus a zero coefficient removes the corresponding term in
\eqref{eq:supp_zero_coefficient_decomposition}; no reciprocal
coefficient occurs.

\proofstep{Step 2: nonzero approximation and remainder audit}
Choose nonzero \(\ell_{i,m}^{\pm}\to\ell_i^{\pm}\), uniformly bounded
in \(i,m\), and put
\[
	\ell_{i,m}(z):=\ell_{i,m}^+\mathbf1_{\{z>0\}}
	+\ell_{i,m}^-\mathbf1_{\{z<0\}}.
\]
For a smooth cutoff \(\chi\) equal to one on \([-r_0,r_0]\) and zero
outside \((-2r_0,2r_0)\), set
\[
	\check q_{i,m}(z):=\check q_i(z)+\chi(z)
	\{\ell_{i,m}(z)-\ell_i(z)\}z.
\]
For all sufficiently large \(m\), this is a globally admissible family
with nonzero first-order coefficients and constants uniform in \(m\).
Hence the nonzero-coefficient results of
\citet[Propositions~4--5, Lemmas~3--4, 7, 17, and~19]{BonieceFigueroaLopezHan2024}
apply to \(\check q_{i,m}\).

On \(0<|z|\le2r_0\), define
\[
	\phi_{i,m}(z):=-\log\check q_{i,m}(z),
	\qquad
	\psi_{i,m}(z):=\phi_{i,m}(z)+\ell_{i,m}(z)z,
\]
and let \(U_{i,t}^{(m)}\) and \(U_{i,t}^{{\rm fv},(m)}\) be the
corresponding quantities in
\eqref{eq:supp_zero_coefficient_likelihood} and
\eqref{eq:supp_zero_coefficient_decomposition}. Because
\(\check q_{i,m}(z)-1-\ell_{i,m}(z)z
=\check q_i(z)-1-\ell_i(z)z\) on \(0<|z|\le r_0\), while
\(\check q_{i,m}\) is uniformly bounded above and away from zero on
\(r_0<|z|<2r_0\),
\begin{equation}
\begin{aligned}
	L_\ell&:=\sup_{i,m}
	\bigl(|\ell_{i,m}^+|+|\ell_{i,m}^-|\bigr)<\infty,\\
	M_q&:=\sup_{i,m}\int_{0<|z|\le1}
	|\check q_{i,m}(z)-1-\ell_{i,m}(z)z|
	|z|^{-1-\beta}\,dz<\infty,\\
	M_\phi&:=\sup_{i,m}\int_{0<|z|\le2r_0}
	\bigl\{|\psi_{i,m}(z)|+|\phi_{i,m}(z)|^2\bigr\}
	\widetilde\nu_i(dz)<\infty,\\
	M_U&:=\sup_{\substack{i,m,\,0<t\le1\\ |p|\le8}}
	\left\{\widetilde\E_i e^{pU_{i,t}^{(m)}}
	+t^{-1}\widetilde\E_i|U_{i,t}^{(m)}|^2
	+t^{-1}\widetilde\E_i
	|e^{-U_{i,t}^{(m)}}-1|^2\right\}<\infty.
\end{aligned}
	\label{eq:supp_zero_coefficient_uniform_inputs}
\end{equation}
The same calculation as in
\eqref{eq:supp_zero_coefficient_decomposition} now gives
\begin{equation}
	U_{i,t}^{(m)}=U_{i,t}^{{\rm fv},(m)}
	-\ell_{i,m}^+S_{i,t}^+-\ell_{i,m}^-S_{i,t}^-,
	\qquad
	\sup_{i,m,\,0<t\le1}t^{-1}
	\widetilde\E_i|U_{i,t}^{{\rm fv},(m)}|
	\lesssim M_\phi.
	\label{eq:supp_zero_coefficient_uniform_decomposition}
\end{equation}
The reference measure is independent of \(m\); hence, for every
\(r>\beta\) and \(0<v\le1\),
\begin{equation}
	\sup_i\int_{|z|\le v}|z|^r\widetilde\nu_i(dz)
	\lesssim v^{r-\beta},
	\qquad
	\sup_i\widetilde\nu_i(|z|>v)\lesssim v^{-\beta}.
	\label{eq:supp_zero_coefficient_stable_inputs}
\end{equation}
Moreover, global admissibility and the compactly supported perturbation
give, for every \(r>0\),
\begin{equation}
	\sup_{i,m}\int_{|z|>r_0}|z|^r
	\check q_{i,m}(z)|z|^{-1-\beta}\,dz<\infty.
	\label{eq:supp_zero_coefficient_finite_activity_input}
\end{equation}

The following audit identifies the quantity that fixes the constant in
each imported estimate. It replaces an appeal to inspection of the cited
proofs.
\begin{table}[H]
\centering
\small
\begin{adjustbox}{max width=\textwidth}
\begin{tabular}{@{}p{0.28\textwidth}p{0.29\textwidth}p{0.39\textwidth}@{}}
\toprule
Imported estimate & Quantity used by the cited proof
& Uniform control for \(\check q_{i,m}\) \\
\midrule
Change-of-measure factors in (C.19), (C.27), and Lemma~19
& Exponential moments of \(U_{i,t}^{(m)}\) and the
\(L^2\)-moments of \(U_{i,t}^{(m)}\) and
\(e^{-U_{i,t}^{(m)}}-1\)
& The constant is bounded by \(M_U\) in
\eqref{eq:supp_zero_coefficient_uniform_inputs}. \\
\addlinespace
Brownian-boundary terms in (C.18)--(C.26)
& Gaussian moments and the likelihood perturbation at the truncation
boundary
& The Gaussian constants use only the common bounds for
\(a_i(\gamma^\star)\); the perturbation is controlled by \(M_U\).
Neither quantity depends on a lower bound for
\(|\ell_{i,m}^{\pm}|\). \\
\addlinespace
Brownian--jump terms in (C.18)--(C.30) and Lemma~19
& Mixed truncated moments of the Brownian motion and the stable
reference process, together with the coefficients in the decomposition
of \(U_{i,t}^{(m)}\)
& The mixed moments use
\eqref{eq:supp_zero_coefficient_stable_inputs}; the coefficients occur
only through nonnegative powers bounded by \(L_\ell\), while the
finite-variation likelihood term is bounded by
\eqref{eq:supp_zero_coefficient_uniform_decomposition}. \\
\addlinespace
Multiple-jump terms in Propositions~4--5
& The large-jump intensity and the second factorial moment of its
Poisson count
& The stable part is bounded by \(v^{-\beta}\) in
\eqref{eq:supp_zero_coefficient_stable_inputs}; the part outside
\((-r_0,r_0)\) is bounded by
\eqref{eq:supp_zero_coefficient_finite_activity_input}. \\
\addlinespace
Local-density remainder in (C.31) and Lemma~7
& Integrability of
\(|\check q_{i,m}(z)-1-\ell_{i,m}(z)z||z|^{-1-\beta}\)
& Its common bound is exactly \(M_q\) in
\eqref{eq:supp_zero_coefficient_uniform_inputs}; on
\(|z|\le r_0\), the remainder is independent of \(m\). \\
\addlinespace
Fourth and annular estimates in (C.33), Lemmas~3--4, and Lemma~17
& Truncated stable moments of orders at least four and the
finite-activity moments
& Equations \eqref{eq:supp_zero_coefficient_stable_inputs} and
\eqref{eq:supp_zero_coefficient_finite_activity_input} give the common
\(v^{r-\beta}\) bounds, in particular \(v^{4-\beta}\). \\
\addlinespace
Mixed Brownian replacement in Lemmas~3--4 and~19
& The likelihood moments, the mixed stable moments, and Gaussian
moments
& These are controlled by \(M_U\),
\eqref{eq:supp_zero_coefficient_stable_inputs}, and the common bounds
for \(a_i(\gamma^\star)\), respectively. \\
\bottomrule
\end{tabular}
\end{adjustbox}
\caption{Uniformity audit for the nonzero-coefficient IV estimates.}
\label{tab:supp_zero_coefficient_uniformity_audit}
\end{table}
Thus every imported constant is a function only of
\(L_\ell,M_q,M_\phi,M_U\), the stable and finite-activity bounds in
\eqref{eq:supp_zero_coefficient_stable_inputs}--%
\eqref{eq:supp_zero_coefficient_finite_activity_input}, and the common
Gaussian coefficient bounds. All of these quantities are independent of
\(m\). In particular, no reciprocal coefficient or lower bound on
\(|\ell_{i,m}^{\pm}|\) enters. With \(Y=\beta\),
\(C_+=C_-=c_i\), \(\sigma=a_i(\gamma^\star)\), \(h=h_n\), and
\(\varepsilon=v\), the five remainder groups satisfy
\begin{equation}
\begin{aligned}
	|R_{i,m,n}^{\rm BB}(v)|
	&\lesssim R(x)h_n^3v^{-\beta-2},
	&|R_{i,m,n}^{\rm BJ}(v)|
	&\lesssim R(x)h_n^{3/2}v^{1-\beta/2},\\
	|R_{i,m,n}^{\ge2}(v)|
	&\lesssim R(x)h_n^2v^{2-2\beta},
	&|R_{i,m,n}^{q}(v)|
	&\lesssim R(x)h_nv^{2-\epsilon_0},\\
	|R_{i,m,n}^{\rm fa}(v)|
	&\lesssim R(x)h_n^2.
\end{aligned}
	\label{eq:supp_zero_coefficient_remainder_audit}
\end{equation}
Here the five superscripts denote the higher Brownian-boundary,
Brownian--small-jump, multiple-jump, local-density, and
finite-activity/drift terms. Their sum proves
\eqref{eq:supp_zero_coefficient_mean}; the cited fourth-moment and
mixed-moment estimates prove
\eqref{eq:supp_zero_coefficient_fourth}--%
\eqref{eq:supp_zero_coefficient_brownian_tail}, with constants uniform
in \(m\).

\proofstep{Step 3: passage to a zero coefficient}
Construct the compensated processes for \(\check q_{i,m}\) and
\(\check q_i\) from a common marked Poisson random measure. The
Burkholder--Davis--Gundy inequality gives, for \(t\le1\),
\[
	\E\!\left[\sup_{s\le t}
	|\check L_{i,s}^{(m)}-\check L_{i,s}|^2\right]
	\lesssim t\max_{\pm}|\ell_{i,m}^{\pm}-\ell_i^{\pm}|
	\longrightarrow0,
\]
and direct integration gives
\( |\check\kappa_{i,n}^{(m)}-\check\kappa_{i,n}|\to0\).
Thus \(\check\xi_{h_n}^{i,x,(m)}\to\check\xi_{h_n}^{i,x}\) in
\(L^2\). The nondegenerate Brownian component implies that the
threshold boundaries have probability zero. The globally admissible
moment bounds give uniform integrability of every integrand in the
lemma. Passing \(m\to\infty\) proves
\eqref{eq:supp_zero_coefficient_mean}--%
\eqref{eq:supp_zero_coefficient_brownian_tail}, including when either
coefficient is zero.
\end{proof}

There exist deterministic constants \(C_\star,c_\star>0\), depending
only on the constants in Assumptions~\ref{ass:smooth} and
\ref{ass:levy} and the
\greenrev{standing IV specification of the main paper}, such that all polynomial
moments used below are bounded by
\begin{equation}
	R_\star(x):=C_\star(1+|x|^{c_\star}).
	\label{eq:supp_local_iv_explicit_envelope}
\end{equation}

\begin{prop}[Localized IV moment expansion]
	\label{prop:supp_localized_iv_moments}
	Suppose the \greenrev{standing IV specification of the main paper} and
	Assumptions~\ref{ass:smooth},
	\ref{ass:levy}, and \ref{ass:iv_sampling} hold.
	Uniformly over \(i\in S\), \(x\in\mathbb R\), and
	\(v<w\) in \(\mathcal V_n\),
	\begin{align*}
		\left|
		\E_x^i[G_v(Y_{h_n}^{i,x})]
		-h_n\{a_i^2(\gamma^\star)+d_{i,1}v^{2-\beta}
		+d_{i,2}h_nv^{-\beta}\}
		\right|
		&\lesssim h_n\rho_nR_\star(x),\\
		\E_x^i[G_v(Y_{h_n}^{i,x})^2]
		&\lesssim R_\star(x)(h_n^2+h_nv^{4-\beta}),\\
		\E_x^i[
		\{G_w(Y_{h_n}^{i,x})-G_v(Y_{h_n}^{i,x})\}^2]
		&\lesssim h_nv^{4-\beta}R_\star(x).
	\end{align*}
	Moreover,
	\[
		\operatorname{Var}_x^i\{\mathcal B_{h_n}^{i,x}(v)\}
		\lesssim R_\star(x)
		(h_n^2v^{2-\beta}+h_nv^{4-\beta}+h_n^3).
	\]
\end{prop}

\begin{proof}[Proof of Proposition~\ref{prop:supp_localized_iv_moments}]
The proof has three transfers. We first construct a globally admissible
auxiliary L\'evy density that agrees with the model density near zero.
We then couple the two fixed-regime processes and quantify the
finite-activity difference. Finally, we transfer the resulting frozen
expansion through the state-dependent drift remainder.

\proofstep{Step 1: globally admissible auxiliary density}
Choose \(r_0\in(0,1/4)\) such that
\(1/2\le q_i(z)\le3/2\) for \(i\in S\) and
\(0<|z|\le r_0\). Define
\begin{equation}
	\widehat q_i(z):=
	\begin{cases}
		q_i(z),&0<|z|\le r_0,\\
		0,&|z|>r_0.
	\end{cases}
	\label{eq:supp_local_iv_auxiliary_density}
\end{equation}
This bounded Borel function agrees with \(q_i\) near zero and has
compact support. The finite regime set therefore makes
\(\{\widehat q_i:i\in S\}\) globally admissible with common constants.
Put
\[
	\widehat\nu_i(dz):=c_i\widehat q_i(z)|z|^{-1-\beta}\,dz,
	\qquad
	\widehat\kappa_{i,n}:=
	\int_{u_n<|z|\le1}z\,\widehat\nu_i(dz),
\]
and define \(\widehat L^i\), \(\widehat\xi_{h_n}^{i,x}\), and
\(\widehat R_{h_n}^{i,x}\) as in the setup of
Lemma~\ref{lem:supp_iv_zero_coefficient_extension}, with checks replaced
by hats.

Because \(v<r_0\) for all sufficiently large \(n\), the auxiliary and
model measures have the same local coefficients. In particular,
\begin{align}
	\int_{|z|\le v}z^2\nu_i^\star(dz)
	&=c_i\int_0^v z^{1-\beta}\{q_i(z)+q_i(-z)\}\,dz\nonumber\\
	&=d_{i,1}v^{2-\beta}
	+O(v^{3-\beta}+v^{2+\rho_q})
	=d_{i,1}v^{2-\beta}+O(v^{2-\epsilon_0}).
	\label{eq:supp_local_iv_levy_integral}
\end{align}
Applying Lemma~\ref{lem:supp_iv_zero_coefficient_extension} to
\(\widehat q_i\) gives
\begin{align}
	&\left|
	\E_x^iG_v(\widehat\xi_{h_n}^{i,x})
	-h_n\{a_i^2+d_{i,1}v^{2-\beta}
	+d_{i,2}h_nv^{-\beta}\}\right|\nonumber\\
	&\quad\lesssim R_\star(x)\{
	h_n^3v^{-\beta-2}+h_n^{3/2}v^{1-\beta/2}
	+h_n^2v^{2-2\beta}+h_nv^{2-\epsilon_0}+h_n^2\},
	\label{eq:supp_local_iv_auxiliary_expansion}\\
	\E_x^i[G_v(\widehat\xi_{h_n}^{i,x})^2]
	&\lesssim R_\star(x)(h_n^2+h_nv^{4-\beta}),\nonumber\\
	\E_x^i[\{G_w(\widehat\xi_{h_n}^{i,x})
	-G_v(\widehat\xi_{h_n}^{i,x})\}^2]
	&\lesssim h_nv^{4-\beta}R_\star(x).
	\label{eq:supp_local_iv_auxiliary_fourth}
\end{align}
The three Brownian-replacement bounds are those in
\eqref{eq:supp_zero_coefficient_mixed}--%
\eqref{eq:supp_zero_coefficient_brownian_tail}, with hats in place of
checks.

\proofstep{Step 2: finite-activity replacement}
Put
\begin{equation}
	\mu_i(dz):=\mathbf1_{\{|z|>r_0\}}\nu_i^\star(dz),
	\qquad
	\lambda_i^{\rm fa}:=\mu_i(\mathbb R).
	\label{eq:supp_local_iv_finite_activity_measures}
\end{equation}
Assumption~\ref{ass:levy} and finiteness of a L\'evy measure away
from zero give, for every fixed \(p>0\),
\begin{equation}
	\sup_{i\in S}\left\{
	\lambda_i^{\rm fa}+
	\int|z|^p\mu_i(dz)\right\}<\infty.
	\label{eq:supp_local_iv_finite_activity_moments}
\end{equation}
Let
\[
	V_{i,t}:=\int_0^t\int_{|z|>r_0}z\,N_i(ds,dz),
	\qquad
	m_i^{\rm fa}:=\int_{|z|>r_0}z\,\nu_i^\star(dz).
\]
The process \(V_i\) is compound Poisson, independent of
\((W,\widehat L^i)\). Since \(u_n<r_0\) for all sufficiently large
\(n\), \(\widehat\nu_i=\nu_i^\star\) on
\(\{u_n<|z|\le r_0\}\) and \(\widehat\nu_i=0\) on
\(\{|z|>r_0\}\). Consequently,
\begin{equation}
	\kappa_{i,n}-\widehat\kappa_{i,n}
	=\int_{|z|>u_n}z\,\nu_i^\star(dz)
	-\int_{u_n<|z|\le r_0}z\,\nu_i^\star(dz)
	=\int_{|z|>r_0}z\,\nu_i^\star(dz)
	=m_i^{\rm fa}.
	\label{eq:supp_local_iv_tail_compensator_cancellation}
\end{equation}
Under the corresponding disjoint-region coupling of the Poisson random
measures,
\begin{equation}
	L_t^i=\widehat L_t^i+V_{i,t}-tm_i^{\rm fa},
	\qquad
	\xi_{h_n}^{i,x}
	=\widehat\xi_{h_n}^{i,x}+V_{i,h_n}.
	\label{eq:supp_local_iv_finite_activity_decomposition}
\end{equation}
Indeed, the second identity follows from the first after substituting
\eqref{eq:supp_local_iv_tail_compensator_cancellation} into the
definitions of \(\xi_{h_n}^{i,x}\) and
\(\widehat\xi_{h_n}^{i,x}\). Thus the compensator of the added
finite-activity component cancels exactly with the difference between
the two tail-centerings; there is no further deterministic shift.

For \(F_{v,2}:=G_v\) and \(F_{v,4}:=G_v^2\), the algebraic
decomposition used in formula (C.35) of
\citet{BonieceFigueroaLopezHan2024} reads, for \(k\in\{2,4\}\),
\begin{align}
	|F_{v,k}(y+z)-F_{v,k}(y)|
	&\le |(y+z)^k-y^k|
	\mathbf1_{\{|y+z|\le v,\,|y|\le v\}}\nonumber\\
	&\quad+|y+z|^k
	\mathbf1_{\{|y+z|\le v,\,|y|>v\}}
	+|y|^k
	\mathbf1_{\{|y+z|>v,\,|y|\le v\}}\nonumber\\
	&\lesssim v^k\mathbf1_{\{z\ne0\}}.
	\label{eq:supp_local_iv_fa_pointwise}
\end{align}
Indeed, the first indicator implies \(|z|\le2v\), and each of the last
two terms is at most \(v^k\). If \(K_{i,h_n}^{\rm fa}\) denotes the
jump count of \(V_i\), then
\begin{equation}
	\Pb(K_{i,h_n}^{\rm fa}\ge1)\le h_n\lambda_i^{\rm fa},
	\qquad
	\Pb(K_{i,h_n}^{\rm fa}\ge2)\le
	\frac12h_n^2(\lambda_i^{\rm fa})^2.
	\label{eq:supp_local_iv_fa_count}
\end{equation}
Since \(V_i\) is independent of \(\widehat\xi^{i,x}\), equations
\eqref{eq:supp_local_iv_fa_pointwise}--%
\eqref{eq:supp_local_iv_fa_count} give
\begin{equation}
\begin{aligned}
	\E_x^i|F_{v,2}(\widehat\xi_{h_n}^{i,x}
	+V_{i,h_n})-F_{v,2}(\widehat\xi_{h_n}^{i,x})|
	&\lesssim h_nv^2,\\
	\E_x^i|F_{v,4}(\widehat\xi_{h_n}^{i,x}
	+V_{i,h_n})-F_{v,4}(\widehat\xi_{h_n}^{i,x})|
	&\lesssim h_nv^4.
\end{aligned}
	\label{eq:supp_local_iv_fa_addition}
\end{equation}
The same estimate at \(v\) and \(w\), with \(w/v\) uniformly bounded,
controls the annular functional by \(h_nv^4\). Moreover,
\(\omega>(4-\beta)^{-1}>(2+\beta)^{-1}\), because \(\beta>1\).
Hence, uniformly for \(v\asymp h_n^\omega\),
\begin{equation}
	\frac{h_nv^2}{h_n^{3/2}v^{1-\beta/2}}
	\asymp h_n^{\omega(1+\beta/2)-1/2}\longrightarrow0,
	\qquad
	h_nv^2=o\!\left(h_n^{3/2}v^{1-\beta/2}\right).
	\label{eq:supp_local_iv_fa_absorption}
\end{equation}
Combining \eqref{eq:supp_local_iv_fa_addition},
\eqref{eq:supp_local_iv_fa_absorption}, and
\eqref{eq:supp_local_iv_auxiliary_expansion} gives
\begin{align}
	&\left|
	\E_x^iG_v(\xi_{h_n}^{i,x})
	-h_n\{a_i^2+d_{i,1}v^{2-\beta}
	+d_{i,2}h_nv^{-\beta}\}\right|\nonumber\\
	&\quad\lesssim R_\star(x)\{
	h_n^3v^{-\beta-2}+h_n^{3/2}v^{1-\beta/2}
	+h_n^2v^{2-2\beta}+h_nv^{2-\epsilon_0}+h_n^2\},
	\label{eq:supp_local_iv_frozen_expansion}\\
	\E_x^i[G_v(\xi_{h_n}^{i,x})^2]
	&\lesssim R_\star(x)(h_n^2+h_nv^{4-\beta}),
	\label{eq:supp_local_iv_frozen_fourth}\\
	\E_x^i[\{G_w(\xi_{h_n}^{i,x})-G_v(\xi_{h_n}^{i,x})\}^2]
	&\lesssim h_nv^{4-\beta}R_\star(x).
	\label{eq:supp_local_iv_frozen_annulus}
\end{align}
For the Brownian-replacement bounds, put
\(\bar R_{h_n}^{i,x}:=\xi_{h_n}^{i,x}
-a_i(\gamma^\star)W_{h_n}\). On
\(\{K_{i,h_n}^{\rm fa}\ge1\}\),
\begin{equation}
\begin{aligned}
	W_{h_n}^2(\bar R_{h_n}^{i,x})^2
	\mathbf1_{\{|\xi_{h_n}^{i,x}|\le v\}}
	&\lesssim v^2W_{h_n}^2+W_{h_n}^4,\\
	(\bar R_{h_n}^{i,x})^4
	\mathbf1_{\{|\xi_{h_n}^{i,x}|\le v\}}
	&\lesssim v^4+W_{h_n}^4.
\end{aligned}
	\label{eq:supp_local_iv_fa_brownian_pointwise}
\end{equation}
The independence of \(K_{i,h_n}^{\rm fa}\) from \(W\), the Gaussian
moment formula, and \eqref{eq:supp_local_iv_fa_count} show that the
expectations of the two right-hand sides on this event are bounded by
\(h_n^2v^2+h_n^3\) and \(h_nv^4+h_n^3\), respectively. On the
zero-count event, \(V_{i,h_n}=0\), so
\eqref{eq:supp_local_iv_finite_activity_decomposition} gives the exact
identities
\[
	\xi_{h_n}^{i,x}=\widehat\xi_{h_n}^{i,x},
	\qquad
	\bar R_{h_n}^{i,x}=\widehat R_{h_n}^{i,x}.
\]
Thus the hatted bounds from Step~1 apply directly, without a boundary
perturbation. Combining the zero- and positive-count events proves
\begin{align}
	\E_x^i[W_{h_n}^2(\bar R_{h_n}^{i,x})^2
	\mathbf1_{\{|\xi_{h_n}^{i,x}|\le v\}}]
	&\lesssim R_\star(x)(h_n^2v^{2-\beta}+h_n^3),\nonumber\\
	\E_x^i[(\bar R_{h_n}^{i,x})^4
	\mathbf1_{\{|\xi_{h_n}^{i,x}|\le v\}}]
	&\lesssim R_\star(x)(h_nv^{4-\beta}
	+h_n^2v^{4-2\beta}+h_n^4),\nonumber\\
	\E_x^i[W_{h_n}^4\mathbf1_{\{|\xi_{h_n}^{i,x}|>v\}}]
	&\lesssim R_\star(x)\{h_n^3v^{-\beta}
	+h_n^2e^{-v^2/(Ch_n)}\}.
	\label{eq:supp_local_iv_frozen_brownian_moments}
\end{align}
For the last line, the deleted-Brownian contribution on
\(\{K_{i,h_n}^{\rm fa}\ge1\}\) is explicit:
\[
	\E[W_{h_n}^4;K_{i,h_n}^{\rm fa}\ge1]
	=3h_n^2\Pb(K_{i,h_n}^{\rm fa}\ge1)
	\lesssim h_n^3
	\lesssim h_n^3v^{-\beta}.
\]

\proofstep{Step 3: state-dependent drift and polynomial envelope}
Write \(a_i=a_i(\gamma^\star)\) and
\(R_{h_n}^{i,x}:=Y_{h_n}^{i,x}-a_iW_{h_n}\). By the Lipschitz
continuity of \(b\), Jensen's inequality, and the fixed-regime
increment moments, for every fixed \(p\ge2\),
\begin{equation}
	\E_x^i|D_{h_n}^{i,x}|^p
	\le h_n^{p-1}\int_0^{h_n}
	\E_x^i|b(X_s^{i,x},i,\alpha^\star)
	-b(x,i,\alpha^\star)|^p\,ds
	\lesssim h_n^{p+1}R_\star(x).
	\label{eq:supp_local_iv_drift_error}
\end{equation}
Proposition~\ref{prop:iv_off_diagonal_density} gives, for
\(0<r\le v/2\),
\begin{equation}
\begin{aligned}
	&\Pb_x^i\!\left(\left||\xi_{h_n}^{i,x}|-v\right|\le r\right)
	+\Pb_x^i\!\left(\left||Y_{h_n}^{i,x}|-v\right|\le r\right)\\
	&\qquad\lesssim R_\star(x)\left\{
	h_nrv^{-1-\beta}
	+rh_n^{-1/2}e^{-v^2/(Ch_n)}\right\}.
\end{aligned}
	\label{eq:supp_local_iv_boundary_transfer}
\end{equation}
Taking \(r=h_n\) and multiplying by \(v^2\) gives
\begin{equation}
\begin{aligned}
	&v^2\Pb_x^i\!\left(\left||\xi_{h_n}^{i,x}|-v\right|\le h_n\right)
	+v^2\Pb_x^i\!\left(\left||Y_{h_n}^{i,x}|-v\right|\le h_n\right)\\
	&\qquad\lesssim R_\star(x)\left\{
	h_n^2v^{1-\beta}
	+h_n^{1/2}v^2e^{-v^2/(Ch_n)}\right\}.
\end{aligned}
	\label{eq:supp_local_iv_boundary_hn}
\end{equation}
For \(|d|\le v/2\),
Lemma~\ref{lem:supp_truncation_perturbation} gives
\begin{equation}
	|G_v(y+d)-G_v(y)|
	\lesssim |y||d|+d^2
	+v^2\mathbf1_{\{||y|-v|\le|d|\}}.
	\label{eq:supp_local_iv_transfer_pointwise}
\end{equation}
Split according as \(|D_{h_n}^{i,x}|\le h_n\) or
\(|D_{h_n}^{i,x}|>h_n\). The Cauchy--Schwarz inequality,
Markov's inequality, and
\eqref{eq:supp_local_iv_drift_error}--%
\eqref{eq:supp_local_iv_boundary_hn} yield
\begin{align}
	&\E_x^i\left|G_v(Y_{h_n}^{i,x})
	-G_v(\xi_{h_n}^{i,x})\right|\nonumber\\
	&\quad\lesssim
	\E_x^i\!\left[
	\{|\xi_{h_n}^{i,x}||D_{h_n}^{i,x}|
	+|D_{h_n}^{i,x}|^2\}
	\mathbf1_{\{|D_{h_n}^{i,x}|\le h_n\}}\right]\nonumber\\
	&\qquad+v^2\Pb_x^i\!\left(
	\left||\xi_{h_n}^{i,x}|-v\right|\le h_n\right)
	+v^2\Pb_x^i(|D_{h_n}^{i,x}|>h_n)\nonumber\\
	&\quad\lesssim R_\star(x)\left\{
	h_n^2+h_n^2v^{1-\beta}
	+h_n^{1/2}v^2e^{-v^2/(Ch_n)}+h_nv^2\right\}\nonumber\\
	&\quad\lesssim R_\star(x)
	\{h_n^{3/2}v^{1-\beta/2}+h_n^2\}.
	\label{eq:supp_local_iv_mean_transfer_error}
\end{align}
The last absorption uses the rate relations in
Table~\ref{tab:supp_iv_rate_audit}. Since \(0\le G_v\le v^2\),
\begin{align*}
	&\left|G_v(Y_{h_n}^{i,x})^2
	-G_v(\xi_{h_n}^{i,x})^2\right|\\
	&\qquad\le2v^2
	\left|G_v(Y_{h_n}^{i,x})-G_v(\xi_{h_n}^{i,x})\right|.
\end{align*}
The same inequality at \(v\) and \(w\), together with the uniformly
bounded ratio \(w/v\), also controls the squared annular functional.
Thus \eqref{eq:supp_local_iv_mean_transfer_error} and
Table~\ref{tab:supp_iv_rate_audit} give, in particular,
\[
	\frac{h_n^{3/2}v^{3-\beta/2}}{h_nv^{4-\beta}}
	=h_n^{1/2}v^{-1+\beta/2}\to0,
	\qquad
	\frac{h_n^2v^2}{h_nv^{4-\beta}}
	=h_nv^{\beta-2}\to0,
\]
and hence
\begin{equation}
\begin{aligned}
	\left|\E_x^iG_v(Y_{h_n}^{i,x})^2
	-\E_x^iG_v(\xi_{h_n}^{i,x})^2\right|
	&\lesssim R_\star(x)
	\{h_n^{3/2}v^{3-\beta/2}+h_n^2v^2\}\\
	&\lesssim h_nv^{4-\beta}R_\star(x),\\
	\E_x^i\left|\{G_w(Y_{h_n}^{i,x})-G_v(Y_{h_n}^{i,x})\}^2
	-\{G_w(\xi_{h_n}^{i,x})-G_v(\xi_{h_n}^{i,x})\}^2\right|
	&\lesssim h_nv^{4-\beta}R_\star(x).
\end{aligned}
	\label{eq:supp_local_iv_fourth_transfer_error}
\end{equation}
Combining \eqref{eq:supp_local_iv_fourth_transfer_error} with
\eqref{eq:supp_local_iv_frozen_fourth}--%
\eqref{eq:supp_local_iv_frozen_annulus} yields
\begin{align}
	\E_x^iG_v(Y_{h_n}^{i,x})^2
	&\lesssim R_\star(x)(h_n^2+h_nv^{4-\beta}),\nonumber\\
	\E_x^i[\{G_w(Y_{h_n}^{i,x})-G_v(Y_{h_n}^{i,x})\}^2]
	&\lesssim h_nv^{4-\beta}R_\star(x).
	\label{eq:supp_local_iv_state_fourth}
\end{align}

Finally,
\begin{equation}
\begin{aligned}
	\left|\mathbf1_{\{|Y_{h_n}^{i,x}|\le v\}}
	-\mathbf1_{\{|\xi_{h_n}^{i,x}|\le v\}}\right|
	&\le\mathbf1_{\{||\xi_{h_n}^{i,x}|-v|
	\le|D_{h_n}^{i,x}|\}},\\
	R_{h_n}^{i,x}&=\bar R_{h_n}^{i,x}+D_{h_n}^{i,x}.
\end{aligned}
	\label{eq:supp_local_iv_brownian_transfer}
\end{equation}
The two pointwise consequences needed for the retained moments are
\begin{align*}
	(R_{h_n}^{i,x})^2\mathbf1_{\{|Y_{h_n}^{i,x}|\le v\}}
	&\lesssim (\bar R_{h_n}^{i,x})^2
	\mathbf1_{\{|\xi_{h_n}^{i,x}|\le v\}}
	+|D_{h_n}^{i,x}|^2\\
	&\quad+\{(\bar R_{h_n}^{i,x})^2+|D_{h_n}^{i,x}|^2\}
	\mathbf1_{\{||\xi_{h_n}^{i,x}|-v|
	\le|D_{h_n}^{i,x}|\}},\\
	(R_{h_n}^{i,x})^4\mathbf1_{\{|Y_{h_n}^{i,x}|\le v\}}
	&\lesssim (\bar R_{h_n}^{i,x})^4
	\mathbf1_{\{|\xi_{h_n}^{i,x}|\le v\}}
	+|D_{h_n}^{i,x}|^4\\
	&\quad+\{(\bar R_{h_n}^{i,x})^4+|D_{h_n}^{i,x}|^4\}
	\mathbf1_{\{||\xi_{h_n}^{i,x}|-v|
	\le|D_{h_n}^{i,x}|\}}.
\end{align*}
Split the two strips at \(|D_{h_n}^{i,x}|=h_n\). Applying H\"older's
inequality with a sufficiently large fixed moment order, then using
\eqref{eq:supp_local_iv_drift_error},
\eqref{eq:supp_local_iv_boundary_transfer}, and
\eqref{eq:supp_local_iv_frozen_brownian_moments}, gives
\begin{align*}
	\E_x^i[W_{h_n}^2(R_{h_n}^{i,x})^2
	\mathbf1_{\{|Y_{h_n}^{i,x}|\le v\}}]
	&\lesssim R_\star(x)(h_n^2v^{2-\beta}+h_n^3),\\
	\E_x^i[(R_{h_n}^{i,x})^4
	\mathbf1_{\{|Y_{h_n}^{i,x}|\le v\}}]
	&\lesssim R_\star(x)(h_nv^{4-\beta}
	+h_n^2v^{4-2\beta}+h_n^4),\\
	\E_x^i[W_{h_n}^4\mathbf1_{\{|Y_{h_n}^{i,x}|>v\}}]
	&\lesssim R_\star(x)\{h_n^3v^{-\beta}
	+h_n^2e^{-v^2/(Ch_n)}\}.
\end{align*}
Expanding
\(G_v(a_iW_{h_n}+R_{h_n}^{i,x})-a_i^2W_{h_n}^2\), and using
\(h_n/v^2\to0\), proves
\[
	\E_x^i|\mathcal B_{h_n}^{i,x}(v)|^2
	\lesssim R_\star(x)
	(h_n^2v^{2-\beta}+h_nv^{4-\beta}+h_n^3).
\]
Combining this estimate with
\eqref{eq:supp_local_iv_frozen_expansion},
\eqref{eq:supp_local_iv_mean_transfer_error}, and
\eqref{eq:supp_local_iv_state_fourth} proves the proposition. The
constants are uniform over the finite regime set and threshold grid,
and the displayed envelope \eqref{eq:supp_local_iv_explicit_envelope}
contains all dependence on the initial state.
\end{proof}

Put
\[
	\mathcal B_{i,j,n}(v)
	:=\Gamma_{i,j}\left\{G_v(Y_{i,j,n}^{\circ})
	-a_i^2(\gamma^\star)(\Delta_jW)^2\right\}.
\]

\begin{lem}[IV conditional expansion]
	\label{lem:iv_conditional_bridge}
	Suppose the \greenrev{standing IV specification of the main paper} and
	Assumptions~\ref{ass:smooth},
	\ref{ass:levy}, \ref{ass:iv_sampling}, and
	\ref{ass:ergodic} hold. Then
	\[
		\sqrt n\,\rho_n\to0,
		\qquad
		\max_{i\in S}\frac{\sqrt n}{T_{i,n}^{\circ}}
		\sum_{j=1}^nh_n\rho_nR_{j-1}=o_p(1).
	\]
	Uniformly over \(i\in S\), \(j\le n\), and
	\(v\in\mathcal V_n\),
	\begin{align*}
		\left|
		\E_{j-1}[\Gamma_{i,j}G_v(Y_{i,j,n}^{\circ})]
		-\E_{j-1}[\Gamma_{i,j}]h_n
		\{a_i^2(\gamma^\star)+d_{i,1}v^{2-\beta}
		+d_{i,2}h_nv^{-\beta}\}
		\right|
		&\lesssim h_n\rho_nR_{j-1},\\
		\E_{j-1}[\Gamma_{i,j}G_v(Y_{i,j,n}^{\circ})^2]
		&\lesssim R_{j-1}(h_n^2+h_nv^{4-\beta}).
	\end{align*}
	For \(v<w\) in \(\mathcal V_n\), uniformly over the same
	\(i\) and \(j\),
	\[
		\E_{j-1}\!\left[\Gamma_{i,j}
		\{G_w(Y_{i,j,n}^{\circ})-G_v(Y_{i,j,n}^{\circ})\}^2\right]
		\lesssim h_nv^{4-\beta}R_{j-1}.
	\]
	Finally,
	\[
		\E_{j-1}\!\left[
		\{\mathcal B_{i,j,n}(v)
		-\E_{j-1}\mathcal B_{i,j,n}(v)\}^2\right]
		\lesssim R_{j-1}
		(h_n^2v^{2-\beta}+h_nv^{4-\beta}+h_n^3).
	\]
\end{lem}

\begin{proof}[Proof of Lemma~\ref{lem:iv_conditional_bridge}]
This proof converts the one-regime estimates into the conditional
triangular-array bounds used by the estimators. The Markov property
identifies the law on a no-switch interval, and an explicit
Bernoulli-mixture variance decomposition accounts for the random
no-switch selector and its conditional mean.

Fix \(i,j\), condition on \(\mathcal F_{t_{j-1}}\), and write
\[
	x:=X_{t_{j-1}},\qquad
	p_{i,j,n}:=\E_{j-1}\Gamma_{i,j}
	=\mathbf1_{\{Z_{j-1}=i\}}e^{q_{ii}h_n}.
\]
By \eqref{eq:supp_local_iv_explicit_envelope} and the envelope
convention in Section~\ref{sec:supp_notation},
\(R_\star(X_{t_{j-1}})\le R_{j-1}\) after enlarging the deterministic
constants and exponent in \(R_{j-1}\).
By the Markov property and the independence of \(Z\) from \(W\) and
the Poisson random measures, for every integrable \(\Phi\),
\begin{equation}
	\E_{j-1}[\Gamma_{i,j}\Phi(Y_{i,j,n}^{\circ})]
	=p_{i,j,n}\E_x^i[\Phi(Y_{h_n}^{i,x})].
	\label{eq:supp_iv_bridge_identity}
\end{equation}
Put
\[
	m_{i,n}(v):=a_i^2(\gamma^\star)+d_{i,1}v^{2-\beta}
	+d_{i,2}h_nv^{-\beta}.
\]
Applying the first assertion of
Proposition~\ref{prop:supp_localized_iv_moments} in
\eqref{eq:supp_iv_bridge_identity} with \(\Phi=G_v\) gives
\begin{align*}
	&\left|
	\E_{j-1}[\Gamma_{i,j}G_v(Y_{i,j,n}^{\circ})]
	-p_{i,j,n}h_nm_{i,n}(v)\right|\\
	&\qquad=p_{i,j,n}
	\left|\E_x^iG_v(Y_{h_n}^{i,x})-h_nm_{i,n}(v)\right|
	\lesssim h_n\rho_nR_{j-1}.
\end{align*}
The second and third assertions of
Proposition~\ref{prop:supp_localized_iv_moments}, applied in
\eqref{eq:supp_iv_bridge_identity} with \(\Phi=G_v^2\) and
\(\Phi=(G_w-G_v)^2\), respectively, give
\begin{align*}
	\E_{j-1}[\Gamma_{i,j}G_v(Y_{i,j,n}^{\circ})^2]
	&\lesssim R_{j-1}(h_n^2+h_nv^{4-\beta}),\\
	\E_{j-1}\!\left[\Gamma_{i,j}
	\{G_w(Y_{i,j,n}^{\circ})-G_v(Y_{i,j,n}^{\circ})\}^2\right]
	&\lesssim h_nv^{4-\beta}R_{j-1}.
\end{align*}

It remains to prove the centered bound for \(\mathcal B_{i,j,n}(v)\).
Conditionally on \(\mathcal F_{t_{j-1}}\) and
\(\{\Gamma_{i,j}=1\}\), the pair consisting of
\(Y_{i,j,n}^{\circ}\) and \(\Delta_jW\) has the same law as
\((Y_{h_n}^{i,x},W_{h_n})\). Define
\[
	\overline m_{i,n}(v)
	:=\E_x^i\mathcal B_{h_n}^{i,x}(v)
	=\E_x^iG_v(Y_{h_n}^{i,x})-a_i^2(\gamma^\star)h_n.
\]
The conditional variance formula for the Bernoulli mixture
\(\mathcal B_{i,j,n}(v)=\Gamma_{i,j}
\mathcal B_{h_n}^{i,x}(v)\) is
\begin{equation}
\begin{aligned}
	&\operatorname{Var}_{j-1}\{\mathcal B_{i,j,n}(v)\}\\
	&\quad=p_{i,j,n}
	\operatorname{Var}_x^i\{\mathcal B_{h_n}^{i,x}(v)\}
	+p_{i,j,n}(1-p_{i,j,n})\overline m_{i,n}(v)^2.
\end{aligned}
\label{eq:supp_iv_bridge_variance_mixture}
\end{equation}
By the first assertion of
Proposition~\ref{prop:supp_localized_iv_moments},
\[
	|\overline m_{i,n}(v)|
	\lesssim h_nR(x)
	\{v^{2-\beta}+h_nv^{-\beta}+\rho_n\}.
\]
The relative-rate bounds in Table~\ref{tab:supp_iv_rate_audit} imply
\begin{equation}
	|\overline m_{i,n}(v)|^2
	\lesssim R(x)
	\{h_n^2v^{2-\beta}+h_nv^{4-\beta}+h_n^3\}.
	\label{eq:supp_iv_bridge_mean_square}
\end{equation}
Combining the final assertion of
Proposition~\ref{prop:supp_localized_iv_moments},
\eqref{eq:supp_iv_bridge_variance_mixture},
\eqref{eq:supp_iv_bridge_mean_square}, and \(p_{i,j,n}\le1\) proves
the required conditional variance estimate.

The IV exponent audit in Table~\ref{tab:supp_iv_rate_audit}
gives \(\sqrt n\rho_n\to0\). Finally,
part~(i) of Lemma~\ref{lem:supp_moment_bounds} and
\eqref{main-eq:Tin_convergence} imply
\begin{align*}
	\max_{i\in S}\frac{\sqrt n}{T_{i,n}^{\circ}}
	\sum_{j=1}^nh_n\rho_nR_{j-1}
	&=\sqrt n\rho_n
	\max_{i\in S}\frac{T_n}{T_{i,n}^{\circ}}
	\left(\frac1n\sum_{j=1}^nR_{j-1}\right)\\
	&=O_p(\sqrt n\rho_n)=o_p(1),
\end{align*}
which completes the proof.
\end{proof}

% The signed-tail decomposition and its global aggregation are contained
% in the proof of the tail-functional proposition in the main
% manuscript.

Use \(\mathcal V_n\), \(G_v\), \(\Gamma_{i,j}\), and
\(Y_{i,j,n}^{\circ}\) as defined in
Section~\ref{sec:supp_notation}.

For deterministic \(0<r\le v/2\), define
\[
	B_{i,n}(v,r):=\frac1{T_{i,n}^{\circ}}
	\sum_{j=1}^n\Gamma_{i,j}
	\mathbf1_{\{\left||Y_{i,j,n}^{\circ}|-v\right|\le r\}}.
\]

\begin{lem}[IV boundary counts]
	\label{lem:iv_boundary_probability}
	Suppose the \greenrev{standing IV specification of the main paper} and
	Assumptions~\ref{ass:smooth},
	\ref{ass:levy}, and \ref{ass:iv_sampling} hold. Uniformly over \(i,j\),
	\(v\in\mathcal V_n\), and \(0<r\le v/2\),
	\begin{equation}
		\label{eq:iv_boundary_probability}
		\E_{j-1}\!\left[
		\Gamma_{i,j}
		\mathbf1_{\{\left||Y_{i,j,n}^{\circ}|-v\right|\le r\}}
		\right]
		\le CR_{j-1}\left\{
			h_nrv^{-1-\beta}
		+r h_n^{-1/2}e^{-v^2/(Ch_n)}\right\}.
		\end{equation}
	If Assumption~\ref{ass:ergodic} also holds and
	\(0<r_n\le u_n/2\) is deterministic, then, for every fixed \(K>0\),
	\begin{equation}
		\max_{i\in S}\max_{v\in\mathcal V_n}B_{i,n}(v,r_n)
		=O_p\!\left(r_nu_n^{-1-\beta}
		+\sqrt{\frac{r_nu_n^{-1-\beta}}{T_n}}\right)+O_p(h_n^K).
		\label{eq:supp_fo_empirical_boundary}
	\end{equation}
\end{lem}

\begin{proof}[Proof of Lemma~\ref{lem:iv_boundary_probability}]
This is the empirical form of the boundary estimate in
Proposition~\ref{prop:iv_off_diagonal_density}. Conditioning on a
no-switch interval gives the one-step probability; summation and
conditional isometry then separate the predictable boundary count from
its martingale fluctuation.

If \(Z_{j-1}\ne i\), the left-hand side of
\eqref{eq:iv_boundary_probability} is zero. Otherwise condition on
\(\Gamma_{i,j}=1\). Independence of the chain from the Brownian motion
and Poisson random measures identifies the conditional law of
\(Y_{i,j,n}^{\circ}\) with that of \(Y_{h_n}^{i,x}\) at
\(x=X_{t_{j-1}}\). Hence
\begin{align*}
	&\E_{j-1}\left[\Gamma_{i,j}
	\mathbf1_{\{\bigl||Y_{i,j,n}^{\circ}|-v\bigr|\le r\}}\right]\\
	&\qquad=\Pb_{j-1}(\Gamma_{i,j}=1)
	\Pb_{X_{t_{j-1}}}^i\left(
	\bigl||Y_{h_n}^{i,X_{t_{j-1}}}|-v\bigr|\le r\right).
\end{align*}
The first factor is at most one. The boundary assertion of
Proposition~\ref{prop:iv_off_diagonal_density}, proved in
Section~\ref{sec:supp_iv_off_diagonal}, bounds the second factor by the
right-hand side of \eqref{eq:iv_boundary_probability}.

For the empirical assertion, sum
\eqref{eq:iv_boundary_probability}, divide by
\(T_{i,n}^{\circ}\), and use
part~(i) of Lemma~\ref{lem:supp_moment_bounds} and
\eqref{main-eq:Tin_convergence}. Uniformly over \(i\) and the finite
grid,
\begin{align}
	\frac1{T_{i,n}^{\circ}}\sum_{j=1}^n
	\E_{j-1}\!\left[
	\Gamma_{i,j}\mathbf1_{\{\left||Y_{i,j,n}^{\circ}|-v\right|\le r_n\}}
	\right]
	=O_p\!\left(r_nv^{-1-\beta}
	+r_n h_n^{-3/2}e^{-v^2/(Ch_n)}\right).
	\label{eq:supp_fo_boundary_predictable}
\end{align}
The centered summands are bounded martingale differences. Conditional
isometry and the predictable estimate
\eqref{eq:supp_fo_boundary_predictable} give
\begin{align}
	&\frac1{T_{i,n}^{\circ}}
	\sum_{j=1}^n\left[
	\Gamma_{i,j}\mathbf1_{\{\left||Y_{i,j,n}^{\circ}|-v\right|\le r_n\}}
	-\E_{j-1}\!\left\{
	\Gamma_{i,j}\mathbf1_{\{\left||Y_{i,j,n}^{\circ}|-v\right|\le r_n\}}
	\right\}\right]\nonumber\\
	&\qquad=O_p\!\left(
	\sqrt{\frac{r_nv^{-1-\beta}}{T_n}}
	+\sqrt{\frac{r_n h_n^{-3/2}e^{-v^2/(Ch_n)}}{T_n}}
	\right).
	\label{eq:supp_fo_boundary_martingale}
\end{align}
Because \(u_n/\sqrt{h_n}\to\infty\), both exponential terms are smaller
than every fixed power of \(h_n\), uniformly over the grid.  Since
\(v\asymp u_n\) there, the last two displays prove
\eqref{eq:supp_fo_empirical_boundary}.
\end{proof}

We next record the predictable random-width localization used whenever a
boundary width contains a
predictable polynomial envelope. Let \(r_{j-1,n}\ge0\) and
\(q_{j-1,n}\) be \(\mathcal F_{t_{j-1}}\)-measurable and satisfy
\[
	r_{j-1,n}+|q_{j-1,n}|\le C_0h_nR_{j-1}.
\]
For \(v\in\mathcal V_n\), define
\[
	\mathcal G_{j,n}(v)
	:=\{4C_0h_nR_{j-1}\le v\}.
\]
On \(\mathcal G_{j,n}(v)\), both the width and the predictable shift
are fixed under \(\E_{j-1}\), and their sum is at most \(v/4\).
Hence Lemma~\ref{lem:iv_boundary_probability}, after translation by
\(q_{j-1,n}\), gives for \(k\in\{1,2,4\}\)
\begin{align}
	&\E_{j-1}\!\left[
	\Gamma_{i,j}|Y_{i,j,n}^{\circ}-q_{j-1,n}|^k
	\mathbf1_{\{\left||Y_{i,j,n}^{\circ}-q_{j-1,n}|-v\right|
	\le r_{j-1,n}\}};
	\mathcal G_{j,n}(v)\right]
	\le Ch_n^2v^{k-1-\beta}R_{j-1}.
	\label{eq:supp_predictable_width_good}
\end{align}
Here products of polynomial envelopes are absorbed into \(R_{j-1}\),
and the Gaussian term is smaller than every power of \(h_n\).

On the complement, for every fixed integer \(p\ge1\),
\begin{equation}
	\mathbf1_{\mathcal G_{j,n}(v)^c}
	\le\left(\frac{4C_0h_nR_{j-1}}{v}\right)^p.
	\label{eq:supp_predictable_width_bad_indicator}
\end{equation}
The one-step moment bounds give
\[
	\E_{j-1}|Y_{i,j,n}^{\circ}-q_{j-1,n}|^k
	\le Ch_n^{a_k}R_{j-1},
	\qquad a_1=\tfrac12,\quad a_2=a_4=1.
\]
Since \(v\asymp h_n^\omega\) and \(1-\omega>0\), choose \(p\)
large enough, simultaneously for the finite set
\(k\in\{1,2,4\}\), that
\[
	h_n^{a_k}\left(\frac{h_n}{v}\right)^p
	\le Ch_n^2v^{k-1-\beta}.
\]
Multiplying \eqref{eq:supp_predictable_width_bad_indicator} by the
conditional moment bound therefore gives the same right-hand side as
\eqref{eq:supp_predictable_width_good}. Consequently,
\begin{equation}
\begin{aligned}
	&\E_{j-1}\!\left[
	\Gamma_{i,j}|Y_{i,j,n}^{\circ}-q_{j-1,n}|^k
	\mathbf1_{\{\left||Y_{i,j,n}^{\circ}-q_{j-1,n}|-v\right|
	\le r_{j-1,n}\}}\right]
	\le Ch_n^2v^{k-1-\beta}R_{j-1}.
	\label{eq:supp_predictable_width_localized}
\end{aligned}
\end{equation}
The identical argument applies after conditioning on a complete regime
path. Since a transition interval has conditional probability
\(O(h_n)\), its contribution is of no larger order.

For the drift step, define
\begin{equation}
	\label{eq:iv_drift_residual}
	d_{j-1,n}^{\rm dr}:=
	h_n\{b_{j-1}^\star-\kappa_{Z_{j-1},n}\},\qquad
	U_j:=\Delta_jX-d_{j-1,n}^{\rm dr},\qquad
	\bar I_{n,j}:=\mathbf1_{\{|U_j|\le u_n\}}.
\end{equation}
Let
\begin{equation}
	\label{eq:iv_drift_fourth_remainder}
	\chi_n^{\rm dr}:=
	h_n\left(h_n^2u_n^{2-\beta}
	+h_nu_n^{4-\beta}+h_n^3\right)^{1/2}
	+h_n^2u_n^{3-\beta}+h_n^{5/2}+h_n^3.
\end{equation}

\begin{lem}[IV drift-score moments]
	\label{lem:iv_drift_truncated_moments}
	Suppose the \greenrev{standing IV specification of the main paper} and
	Assumptions~\ref{ass:smooth},
	\ref{ass:levy}, and \ref{ass:iv_sampling} hold. Uniformly over \(j\le n\) and
	\(k\in\{1,2,4\}\),
	\begin{equation}
		\label{eq:iv_drift_indicator_shift}
		|I_{n,j}-\bar I_{n,j}|
		\le\mathbf1_{\{
		\left||U_j|-u_n\right|\le|d_{j-1,n}^{\rm dr}|\}},
		\qquad
		\E_{j-1}\!\left[
		|U_j|^k|I_{n,j}-\bar I_{n,j}|\right]
		\le Ch_n^2u_n^{k-1-\beta}R_{j-1}.
	\end{equation}
	\begin{equation}
		\label{eq:iv_drift_moments}
		\begin{aligned}
			|\E_{j-1}[I_{n,j}U_j]|
			&\le Ch_nb_n^{\rm det}R_{j-1},\\
			\left|\E_{j-1}[I_{n,j}U_j^2]
			-h_na_{Z_{j-1}}^2(\gamma^\star)\right|
			&\le Ch_n(u_n^{2-\beta}+h_nu_n^{-\beta})R_{j-1},\\
			\left|\E_{j-1}[I_{n,j}U_j^4]
			-3h_n^2a_{Z_{j-1}}^4(\gamma^\star)\right|
			&\le C\chi_n^{\rm dr}R_{j-1},
		\end{aligned}
	\end{equation}
	\[
		\chi_n^{\rm dr}/h_n^2\to0.
	\]
\end{lem}

\begin{proof}[Proof of Lemma~\ref{lem:iv_drift_truncated_moments}]
This final calculation proves Appendix
Lemma~\ref{lem:iv_drift_moments_main}. It first compares the
observable and drift-centered threshold indicators by the predictable
boundary localization above, and then derives the first, second, and
fourth retained moments in turn. Thus the small-jump bias, threshold
displacement, and Gaussian fourth-moment remainder remain separated.

\proofstep{Step 1: indicator displacement}
Since \(\Delta_jX=U_j+d_{j-1,n}^{\rm dr}\), the indicators
\(I_{n,j}\) and \(\bar I_{n,j}\) can differ only on the boundary strip
in \eqref{eq:iv_drift_indicator_shift}. Apply the predictable-width
localization \eqref{eq:supp_predictable_width_localized} with
\[
	r_{j-1,n}=|d_{j-1,n}^{\rm dr}|,
	\qquad q_{j-1,n}=h_nb_{j-1}^\star,
\]
and condition on the complete regime path on transition intervals. For
\(k=1,2,4\), this gives
\[
	\E_{j-1}\!\left[
	|U_j|^k|I_{n,j}-\bar I_{n,j}|\right]
	\le Ch_n^2u_n^{k-1-\beta}R_{j-1}.
\]
Here \(|d_{j-1,n}^{\rm dr}|+|h_nb_{j-1}^\star|\le
Ch_nR_{j-1}\). This proves
\eqref{eq:iv_drift_indicator_shift}.

\proofstep{Step 2: first retained moment}
The compensated SDE, Lipschitz continuity of the
drift, and Lemma~\ref{lem:supp_moment_bounds} imply
\[
	\left|\E_{j-1}\Delta_jX-h_nb_{j-1}^\star\right|
	\le Ch_n^{3/2}R_{j-1}.
\]
Repeat the zero-, one-, and multiple-count decomposition leading to
\eqref{eq:main_iv_count_bounds}--%
\eqref{eq:main_iv_detection_error} in the proof of
Proposition~\ref{prop:tail_functional}, with the endpoint selector
removed. This gives
\[
	\left|\E_{j-1}\!\left[
	\Delta_jX\mathbf1_{\{|\Delta_jX|>u_n\}}\right]
	-h_n\kappa_{Z_{j-1},n}\right|
	\le Ch_nb_n^{\rm det}R_{j-1},
\]
and
\[
	\Pb_{j-1}(|\Delta_jX|>u_n)
	\le CR_{j-1}\left\{h_nu_n^{-\beta}
	+e^{-u_n^2/(Ch_n)}\right\}.
\]
Removing the selector adds only paths with a chain transition; after
conditioning on the regime path, their contribution is \(O(h_n^2)R_{j-1}\).
Consequently,
\begin{align*}
	\E_{j-1}[I_{n,j}U_j]
	={}&\E_{j-1}\Delta_jX-h_nb_{j-1}^\star
	+h_n\kappa_{Z_{j-1},n}\\
	&-\E_{j-1}\!\left[
	\Delta_jX\mathbf1_{\{|\Delta_jX|>u_n\}}\right]
	+d_{j-1,n}^{\rm dr}\Pb_{j-1}(|\Delta_jX|>u_n).
\end{align*}
Since \(h_n^{1/2}=o(u_n^{2-\beta})\), this proves the first line of
\eqref{eq:iv_drift_moments}.

\proofstep{Step 3: second retained moment}
Work first on the no-switch event, put
\(i=Z_{j-1}\), and write
\[
	Y_j^\circ:=\Delta_jX+h_n\kappa_{i,n},\qquad
	q_{j-1}:=h_nb(X_{t_{j-1}},i,\alpha^\star).
\]
Then \(U_j=Y_j^\circ-q_{j-1}\). The deterministic inequality
\[
	|G_{u_n}(y-q)-G_{u_n}(y)|
	\le C(|q||y|+q^2)\mathbf1_{\{|y|\le3u_n/2\}}
	+Cu_n^2\mathbf1_{\{\lvert |y|-u_n\rvert\le|q|\}}
\]
and \eqref{eq:supp_predictable_width_localized}, applied with
\(r_{j-1,n}=|q_{j-1}|\) and zero translation, imply
\[
	\E_{j-1}\!\left[
	\Gamma_{i,j}|G_{u_n}(U_j)-G_{u_n}(Y_j^\circ)|\right]
	\le CR_{j-1}\{h_n^{3/2}
	+h_n^2u_n^{1-\beta}+h_n^2\}.
\]
Combining this with the mean expansion in
Lemma~\ref{lem:iv_conditional_bridge} at \(v=u_n\), using
\(\rho_n\lesssim u_n^{2-\beta}+h_nu_n^{-\beta}\), and restoring
the transition intervals yields
\[
	\left|\E_{j-1}[\bar I_{n,j}U_j^2]
	-h_na_{Z_{j-1}}^2(\gamma^\star)\right|
	\le Ch_n\{u_n^{2-\beta}+h_nu_n^{-\beta}\}R_{j-1}.
\]
The indicator-shift estimate with \(k=2\) transfers the formula from
\(\bar I_{n,j}\) to \(I_{n,j}\).

	\proofstep{Step 4: fourth retained moment}
	Set
	\(\mu_{r,i}(u)=\int_{|z|\le u}z^r\nu_i^\star(dz)\). If
	\(J_{i,h_n}^{(\le u_n)}\) denotes the frozen compensated small-jump
	martingale, then the frozen Brownian and small-jump variables are functions
	of independent driving measures. The Poisson cumulant identities give
\begin{align*}
	&\E\!\left[
	\{a_i(\gamma^\star)W_{h_n}
	+J_{i,h_n}^{(\le u_n)}\}^4\right]\\
	&\quad=3h_n^2a_i^4(\gamma^\star)
	+6h_n^2a_i^2(\gamma^\star)\mu_{2,i}(u_n)
	+h_n\mu_{4,i}(u_n)+3h_n^2\mu_{2,i}(u_n)^2.
\end{align*}
The \greenrev{standing IV specification of the main paper} gives
\(|\mu_{2,i}(u_n)|\lesssim u_n^{2-\beta}\) and
\(|\mu_{4,i}(u_n)|\lesssim u_n^{4-\beta}\). The Brownian-replacement
	bound in Proposition~\ref{prop:supp_localized_iv_moments}, the conditional
	Cauchy--Schwarz inequality, the predictable shift \(q_{j-1}\) localized as in
	\eqref{eq:supp_predictable_width_localized}, and the transition
intervals add at most
\[
	CR_{j-1}\left[
	h_n\{h_n^2u_n^{2-\beta}+h_nu_n^{4-\beta}+h_n^3\}^{1/2}
	+h_n^2u_n^{3-\beta}+h_n^{5/2}+h_n^3\right].
\]
This is \(C\chi_n^{\rm dr}R_{j-1}\); applying the indicator-shift
estimate with \(k=4\) proves the last line of
\eqref{eq:iv_drift_moments}.

\proofstep{Step 5: rate conclusion}
If
\(K_n^{\rm dr}=h_n^2u_n^{2-\beta}+h_nu_n^{4-\beta}+h_n^3\), then
\[
	\frac{K_n^{\rm dr}}{h_n^2}
	=u_n^{2-\beta}+\frac{u_n^{4-\beta}}{h_n}+h_n\to0.
\]
Therefore
\[
	\frac{\chi_n^{\rm dr}}{h_n^2}
	=O\!\left(\sqrt{K_n^{\rm dr}/h_n^2}
	+u_n^{3-\beta}+\sqrt{h_n}+h_n\right)\to0.
\]
\end{proof}

\section{Feasible two-step debiasing}

This section proves Appendix
Lemma~\ref{lem:iv_debiasing_main}. The two successive rational
corrections are adapted from the debiased truncated-variation method of
\citet{BonieceFigueroaLopezHan2024}. The proof separates four issues:
the abstract two-scale rational-map cancellation, one concrete feasible
expansion, verification of the abstract hypotheses, and the martingale
central limit theorem for the final debiased regime variances. The
feasible replacement and random regime exposures are specific to the
present switching model.

\begin{lem}[Rational-map perturbation]
	\label{lem:supp_rational_map_perturbation}
	Let \(\Psi(a,d):=a^2/d\).  Suppose \(x>0\),
	\(|a|\le C_0x\), \(|d|\ge c_0x\), and
	\(|\delta_d|\le |d|/2\).  Then
	\begin{align}
		\Psi(a+\delta_a,d+\delta_d)-\Psi(a,d)
		&=\frac{2a\delta_a}{d}
		-\frac{a^2\delta_d}{d^2}
		+\mathcal R(a,d;\delta_a,\delta_d),\nonumber\\
		|\mathcal R(a,d;\delta_a,\delta_d)|
		&\le
		C\frac{(|\delta_a|+|\delta_d|)^2}{x},
		\label{eq:supp_rational_taylor}
	\end{align}
	and consequently
	\begin{equation}
		|\Psi(a+\delta_a,d+\delta_d)-\Psi(a,d)|
		\le C\left\{e+\frac{e^2}{x}\right\},
		\qquad e:=|\delta_a|+|\delta_d|.
		\label{eq:supp_rational_lipschitz}
	\end{equation}
  In particular, let \(F,\delta:(0,\infty)\to\mathbb R\), fix
  \(\zeta>1\) and \(v>0\), and define
  \[
    \widetilde F(w):=F(w)+\delta(w),\qquad w>0.
  \]
  Thus \(\delta=\widetilde F-F\) is the pointwise perturbation function.
  Put
	\(a=\Delta_\zeta F(v)\), \(d=\Delta_\zeta^2F(v)\),
	\(\delta_a=\Delta_\zeta\delta(v)\), and
	\(\delta_d=\Delta_\zeta^2\delta(v)\).  On an event on which the
	stabilization indicators of both debiasing maps equal one and the
	preceding size conditions hold,
	\begin{align}
		&(\mathcal D_{\zeta,n}\widetilde F)(v)
		-(\mathcal D_{\zeta,n}F)(v)\nonumber\\
		&\quad=\delta(v)-\frac{2a\delta_a}{d}
		+\frac{a^2\delta_d}{d^2}
		+O\!\left(\frac{(|\delta_a|+|\delta_d|)^2}{x}\right).
		\label{eq:supp_debias_map_perturbation}
	\end{align}
\end{lem}

\begin{proof}[Proof of Lemma~\ref{lem:supp_rational_map_perturbation}]
	This deterministic Taylor bound is the algebraic device used to
	propagate feasible--oracle errors through both debiasing maps without
	repeating the same quotient expansion at each scale.

	After subtracting the two linear terms in
	\eqref{eq:supp_rational_taylor}, the remainder is
	\[
		\frac{\delta_a^2}{d+\delta_d}
		-\frac{2a\delta_a\delta_d}{d(d+\delta_d)}
		+\frac{a^2\delta_d^2}{d^2(d+\delta_d)}.
	\]
	The hypotheses imply \(|d+\delta_d|\ge|d|/2\ge c_0x/2\);
	bounding the three terms proves
	\eqref{eq:supp_rational_taylor}.  Its linear terms are bounded by
	\(Ce\), which proves \eqref{eq:supp_rational_lipschitz}.  Finally,
	on the stated event,
	\[
		(\mathcal D_{\zeta,n}\widetilde F)(v)
		-(\mathcal D_{\zeta,n}F)(v)
		=\delta(v)-\{\Psi(a+\delta_a,d+\delta_d)-\Psi(a,d)\},
	\]
	and \eqref{eq:supp_debias_map_perturbation} follows from
	\eqref{eq:supp_rational_taylor}.
\end{proof}

Fix \(p>0\), \(q<0\), and \(b,c\ne0\). Suppose
\begin{equation}
	\label{eq:debias_abstract_expansion}
	F_n(v)=A_n+bv^p+ch_nv^q+E_n(v),
	\qquad v\in\mathcal V_n.
\end{equation}
For deterministic \(s_n>0\), define
\begin{equation}
	\label{eq:debias_remainder_scales}
	s_{1,n}:=s_n+
	\frac{(h_nu_n^q+s_n)^2}{u_n^p},
	\qquad
	s_{2,n}:=s_{1,n}+\frac{s_{1,n}^2}{h_nu_n^q}.
\end{equation}
Set
\[
	F_n^{(1)}(v):=(\mathcal D_{\zeta_1,n}F_n)(v).
\]
The stabilization level in each occurrence of
\(\mathcal D_{\zeta,n}\) is the deterministic sequence \(h_n\)
specified in the main-paper definition; no additional stabilization
sequence is used.

\begin{lem}[Abstract two-step rational-map lemma]
	\label{lem:iv_debias_stability}
	Suppose
	\begin{align}
		&\max_{v\in\mathcal V_n}|E_n(v)|=o_p(n^{-1/2}),\qquad
		\max_{v,w\in\mathcal V_n}|E_n(v)-E_n(w)|=O_p(s_n),
		\label{eq:debias_remainder_conditions}\\
		&s_n=o(u_n^p),\quad
		h_nu_n^q=o(u_n^p),\quad
		h_n=o(u_n^p),\quad h_n=o(h_nu_n^q),
		\label{eq:debias_scale_separation}\\
		&\sqrt n\,s_{1,n}\to0,\qquad
		s_{1,n}/(h_nu_n^q)\to0.
		\label{eq:debias_rate_conditions}
	\end{align}
	Then, uniformly over \(v\in\mathcal W_n\),
	\begin{align}
		\Delta_{\zeta_1}^2F_n(v)
		&=b\eta_p(\zeta_1)^2v^p\{1+o_p(1)\},
		\label{eq:debias_first_denominator}\\
		F_n^{(1)}(v)
		&=A_n+\lambda_{p,q}(\zeta_1)ch_nv^q
		+E_n^{(1)}(v),
		\label{eq:debias_first_expansion}
	\end{align}
	\[
		\max_{v\in\mathcal W_n}|E_n^{(1)}(v)|
		=o_p(n^{-1/2})+O_p(s_{1,n}),\qquad
		\max_{v,w\in\mathcal W_n}
		|E_n^{(1)}(v)-E_n^{(1)}(w)|=O_p(s_{1,n}).
	\]
	Moreover,
	\begin{align}
		\Delta_{\zeta_2}^2F_n^{(1)}(u_n)
		&=\lambda_{p,q}(\zeta_1)c
		\eta_q(\zeta_2)^2h_nu_n^q\{1+o_p(1)\},
		\label{eq:debias_second_denominator}\\
		(\mathcal D_{\zeta_2,n}F_n^{(1)})(u_n)
		&=A_n+o_p(n^{-1/2})+O_p(s_{2,n})
		=A_n+o_p(n^{-1/2}).
		\label{eq:debias_second_expansion}
	\end{align}
	Finally,
	\[
		\Pb\!\left(
		\min\left\{
		\min_{v\in\mathcal W_n}|\Delta_{\zeta_1}^2F_n(v)|,
		|\Delta_{\zeta_2}^2F_n^{(1)}(u_n)|\right\}\ge h_n
		\right)\to1.
	\]
\end{lem}

\begin{proof}[Proof of Lemma~\ref{lem:iv_debias_stability}]
The first rational map removes the \(v^p\) term, and the second removes
the \(h_nv^q\) term. Lemma~\ref{lem:supp_rational_map_perturbation}
controls the nonlinear remainder at both denominator scales. Because
\(\mathcal V_n/u_n\) is a fixed finite set, all powers of a grid point
are comparable to the corresponding powers of \(u_n\).

\proofstep{Step 1: first denominator}
Equation~\eqref{eq:debias_abstract_expansion} gives
\begin{align*}
	\Delta_{\zeta_1}F_n(v)
	&=a_{0,n}(v)+a_{1,n}(v),\\
	\Delta_{\zeta_1}^2F_n(v)
	&=d_{0,n}(v)+d_{1,n}(v),
\end{align*}
where
\begin{align*}
	a_{0,n}(v)&:=b\eta_p(\zeta_1)v^p,
	&d_{0,n}(v)&:=b\eta_p(\zeta_1)^2v^p,\\
	a_{1,n}(v)&:=c\eta_q(\zeta_1)h_nv^q
	+\Delta_{\zeta_1}E_n(v),
	&d_{1,n}(v)&:=c\eta_q(\zeta_1)^2h_nv^q
	+\Delta_{\zeta_1}^2E_n(v).
\end{align*}
The remainder-difference condition and
\eqref{eq:debias_scale_separation} imply
\[
	\max_{v\in\mathcal W_n}
	\{|a_{1,n}(v)|+|d_{1,n}(v)|\}=o_p(u_n^p).
\]
Since \(b\eta_p(\zeta_1)^2\ne0\), this proves
\eqref{eq:debias_first_denominator}. The relation
\(h_n=o(u_n^p)\) shows that the first stabilization indicator equals
one uniformly over \(\mathcal W_n\), with probability tending to one.

\proofstep{Step 2: first rational map}
Apply Lemma~\ref{lem:supp_rational_map_perturbation} with
\((a_0,d_0,\delta_a,\delta_d)
=(a_{0,n},d_{0,n},a_{1,n},d_{1,n})\) and \(x=u_n^p\). The identity
\[
	1-\frac{2\eta_q(\zeta_1)}{\eta_p(\zeta_1)}
	+\frac{\eta_q(\zeta_1)^2}{\eta_p(\zeta_1)^2}
	=\lambda_{p,q}(\zeta_1)
\]
gives \eqref{eq:debias_first_expansion}. The part linear in the
remainder is
\[
	E_n(v)-\frac{2}{\eta_p(\zeta_1)}\Delta_{\zeta_1}E_n(v)
	+\frac{1}{\eta_p(\zeta_1)^2}\Delta_{\zeta_1}^2E_n(v),
\]
whereas \eqref{eq:supp_rational_taylor} bounds the nonlinear part by
\[
	O_p\!\left(\frac{(h_nu_n^q+s_n)^2}{u_n^p}\right)
\]
uniformly over the grid. Hence
\begin{align}
	\max_{v\in\mathcal W_n}|E_n^{(1)}(v)|
	&=o_p(n^{-1/2})+O_p(s_{1,n}),\nonumber\\
	\max_{v,w\in\mathcal W_n}|E_n^{(1)}(v)-E_n^{(1)}(w)|
	&=O_p(s_{1,n}).
	\label{eq:debias_first_remainder_proof}
\end{align}

\proofstep{Step 3: second denominator}
Put \(B:=\lambda_{p,q}(\zeta_1)c\ne0\). Then
\[
	F_n^{(1)}(v)=A_n+Bh_nv^q+E_n^{(1)}(v).
\]
Since \(s_{1,n}=o(h_nu_n^q)\),
\eqref{eq:debias_first_remainder_proof} gives
\[
	\Delta_{\zeta_2}^2F_n^{(1)}(u_n)
	=B\eta_q(\zeta_2)^2h_nu_n^q\{1+o_p(1)\},
\]
which proves \eqref{eq:debias_second_denominator}. The relation
\(h_n=o(h_nu_n^q)\) makes the second stabilization indicator equal one
with probability tending to one.

\proofstep{Step 4: second rational map}
Apply Lemma~\ref{lem:supp_rational_map_perturbation} with leading scale
\(x=h_nu_n^q\). Then
\begin{align*}
	(\mathcal D_{\zeta_2,n}F_n^{(1)})(u_n)
	&=A_n+E_n^{(2)}(u_n),\\
	|E_n^{(2)}(u_n)|
	&=o_p(n^{-1/2})+O_p\!\left(
	s_{1,n}+\frac{s_{1,n}^2}{h_nu_n^q}\right).
\end{align*}
The last order is \(o_p(n^{-1/2})+O_p(s_{2,n})\). Moreover,
\[
	\sqrt n\,s_{2,n}
	=\sqrt n\,s_{1,n}
	+\sqrt n\,s_{1,n}\frac{s_{1,n}}{h_nu_n^q}\to0.
\]
Thus \eqref{eq:debias_rate_conditions} proves
\eqref{eq:debias_second_expansion}. The two denominator expansions and
the two scale relations involving \(h_n\) prove the stabilization
probability. This completes the proof.
\end{proof}

For the concrete IV expansion, define the oracle curve
\[
 C_{i,n}^{\circ}(v):=\frac1{T_{i,n}^{\circ}}
 \sum_j\Gamma_{i,j}G_v(Y_{i,j,n}^{\circ}),
\]
with value zero when \(T_{i,n}^{\circ}=0\), and put
\[
 R_{i,n}^{(0)}(v):=\widehat C_{i,n}(v)-C_{i,n}^{\circ}(v).
\]
Define
\[
 \varepsilon_{0,n}^{\rm fo}
 :=\frac{h_n}{\sqrt{T_n}}+\frac1n
 +\frac{h_nu_n^{1-\beta}}{\sqrt{T_n}}
 +\frac{h_n^{1/2}u_n^{(3-\beta)/2}}{T_n^{3/4}},
\]
and put \(p:=2-\beta\), \(q:=-\beta\),
\[
	\mathcal Z_{i,n}^{\circ}:=
	\frac{\sqrt n}{T_{i,n}^{\circ}}
	\sum_{j=1}^n\Gamma_{i,j}a_i^2(\gamma^\star)
	\{(\Delta_jW)^2-h_n\},
\]
with \(\mathcal Z_{i,n}^{\circ}=0\) when \(T_{i,n}^{\circ}=0\). Finally,
define
\begin{align*}
	s_n^{\rm iv}
	&:=\rho_n+\frac{u_n^{(4-\beta)/2}}{\sqrt{T_n}}
	+\varepsilon_{0,n}^{\rm fo},\\
	E_{i,n}^{\rm fea}(v)
	&:=\widehat C_{i,n}(v)-a_i^2(\gamma^\star)
	-\frac{\mathcal Z_{i,n}^{\circ}}{\sqrt n}
	-d_{i,1}v^{2-\beta}-d_{i,2}h_nv^{-\beta}.
\end{align*}

\begin{prop}[Feasible IV expansion]
\label{prop:supp_iv_feasible_oracle}
Suppose the \greenrev{standing IV specification of the main paper} and
Assumptions~\ref{ass:smooth}, \ref{ass:levy},
\ref{ass:iv_sampling}, and \ref{ass:ergodic} hold. Then
\[
 \Pb\!\left(A_{i,j}=\Gamma_{i,j}\text{ for all }
 i\in S,\ j\le n\right)\to1,
 \qquad
 \Pb\!\left(T_{i,n}=T_{i,n}^{\circ}\text{ for all }
 i\in S\right)\to1.
\]
Uniformly over the displayed finite grids,
\begin{align*}
 \max_{i\in S}\max_{v\in\mathcal V_n}|R_{i,n}^{(0)}(v)|
 &=O_p(\varepsilon_{0,n}^{\rm fo}),\\
 \max_{\substack{i\in S,\ a\in\{1,2\}\\
 v,\zeta_av,\zeta_a^2v\in\mathcal V_n}}
 \left\{|\Delta_{\zeta_a}R_{i,n}^{(0)}(v)|
 +|\Delta_{\zeta_a}^2R_{i,n}^{(0)}(v)|\right\}
 &=O_p(\varepsilon_{0,n}^{\rm fo}).
\end{align*}
Moreover, uniformly over \(i\in S\) and \(v\in\mathcal V_n\),
\begin{equation}
	\widehat C_{i,n}(v)
	=a_i^2(\gamma^\star)+\frac{\mathcal Z_{i,n}^{\circ}}{\sqrt n}
	+d_{i,1}v^{2-\beta}+d_{i,2}h_nv^{-\beta}
	+E_{i,n}^{\rm fea}(v),
	\label{eq:supp_feasible_curve_expansion}
\end{equation}
where
\begin{equation}
\begin{aligned}
	\max_{i\in S}\max_{v\in\mathcal V_n}|E_{i,n}^{\rm fea}(v)|
	&=o_p(n^{-1/2}),\\
	\max_{i\in S}\max_{v,w\in\mathcal V_n}
	|E_{i,n}^{\rm fea}(v)-E_{i,n}^{\rm fea}(w)|
	&=O_p(s_n^{\rm iv}).
	\label{eq:supp_feasible_curve_remainders}
\end{aligned}
\end{equation}
The deterministic rates satisfy
\begin{equation}
\begin{gathered}
	\sqrt n\,\varepsilon_{0,n}^{\rm fo}\to0,
	\qquad
	\frac{\varepsilon_{0,n}^{\rm fo}}{u_n^p}\to0,
	\qquad
	\frac{\varepsilon_{0,n}^{\rm fo}}{h_nu_n^q}\to0,\\
	\sqrt n\,s_n^{\rm iv}\to0,
	\qquad
	\frac{s_n^{\rm iv}}{u_n^p}\to0,
	\qquad
	\frac{s_n^{\rm iv}}{h_nu_n^q}\to0.
	\label{eq:supp_feasible_curve_rates}
\end{gathered}
\end{equation}
\end{prop}

\begin{proof}[Proof of Proposition~\ref{prop:supp_iv_feasible_oracle}]
This proposition isolates the part of Appendix
Lemma~\ref{lem:iv_debiasing_main} that is absent in the
fixed-regime oracle calculation. We first prove the selector, exposure,
tail-mean, and boundary replacements and then combine them with the
oracle grid expansion. No rational map is applied in this proposition.

\proofstep{Step 1: endpoint equality and exposure}
Lemma~\ref{lem:supp_endpoint_exposure} gives
\begin{equation}
	\Pb\{(\Omega_n^{\rm ns})^c\}
	\le Cnh_n^2\to0.
	\label{eq:supp_fo_no_leave_return}
\end{equation}
On \(\Omega_n^{\rm ns}\),
\begin{equation}
	A_{i,j}=\Gamma_{i,j}\quad(i\in S,j\le n),
	\qquad T_{i,n}=T_{i,n}^{\circ}\quad(i\in S).
	\label{eq:supp_fo_exact_selector_exposure}
\end{equation}
Thus the selector and exposure replacements are exact simultaneously on
an event whose probability tends to one; in particular, they remain
exact after any of the normalizations in the proposition.

\proofstep{Step 2: uniform signed first-moment bound}
Put
\[
	H_v(y):=y\mathbf1_{\{|y|\le v\}},
	\qquad
	\kappa_i(v):=\int_{|z|>v}z\nu_i^\star(dz).
\]
The symmetric leading part of the L\'evy density has zero signed integral.
Moreover, the \greenrev{standing IV specification of the main paper} gives,
uniformly in \(i\) and
\(0<v\le1\),
\begin{equation}
\begin{aligned}
	\left|\int_{v<|z|\le1}z\nu_i^\star(dz)\right|
	&=c_i\left|\int_v^1z^{-\beta}
	\{q_i(z)-q_i(-z)\}\,dz\right|\\
	&\le C\int_v^1z^{1-\beta}\,dz+C\int_v^1z^{\rho_q}\,dz
	\le C.
	\label{eq:supp_fo_signed_tail_integral}
\end{aligned}
\end{equation}
The part outside \([-1,1]\) is finite by
Assumption~\ref{ass:levy}; hence
\begin{equation}
	\sup_{i\in S}\sup_{0<v\le1}|\kappa_i(v)|<\infty.
	\label{eq:supp_fo_kappa_v_bounded}
\end{equation}
The compensated SDE, the Lipschitz drift, and the one-step moment bound
give
\begin{equation}
\begin{aligned}
	\E_{j-1}[\Gamma_{i,j}Y_{i,j,n}^{\circ}]
	={}&\E_{j-1}[\Gamma_{i,j}]h_n
	\{b(X_{t_{j-1}},i,\alpha^\star)+\kappa_{i,n}\}
	+O(h_n^{3/2}R_{j-1}).
	\label{eq:supp_fo_full_first_moment}
\end{aligned}
\end{equation}
Next repeat the signed-tail count calculation
\eqref{eq:main_iv_count_bounds}--%
\eqref{eq:main_iv_detection_error} from the proof of
Proposition~\ref{prop:tail_functional}, with the cutoff \(v\) in
place of \(u_n\). Its constants and rates remain uniform because
\(v/u_n\) ranges over a fixed finite set. Replacing \(\Delta_jX\)
by \(Y_{i,j,n}^{\circ}=\Delta_jX+h_n\kappa_{i,n}\) changes the indicator
only on a boundary strip of width \(O(h_n)\);
Lemma~\ref{lem:iv_boundary_probability} makes the contribution of this strip
\(O(h_nR_{j-1})\). Consequently,
\begin{equation}
	\left|\E_{j-1}\!\left[
	\Gamma_{i,j}Y_{i,j,n}^{\circ}
	\mathbf1_{\{|Y_{i,j,n}^{\circ}|>v\}}\right]
	-\E_{j-1}[\Gamma_{i,j}]h_n\kappa_i(v)\right|
	\le Ch_nR_{j-1}.
	\label{eq:supp_fo_detected_first_moment}
\end{equation}
Subtracting \eqref{eq:supp_fo_detected_first_moment} from
\eqref{eq:supp_fo_full_first_moment}, and using
\eqref{eq:supp_fo_kappa_v_bounded} and
\(\sup_{i,n}|\kappa_{i,n}|<\infty\), yields
\begin{equation}
	\left|\E_{j-1}\{
	\Gamma_{i,j}H_v(Y_{i,j,n}^{\circ})\}\right|
	\le Ch_nR_{j-1},
	\qquad v\in\mathcal V_n.
	\label{eq:supp_fo_signed_one_step}
\end{equation}
Also, the second-moment assertion of
Lemma~\ref{lem:iv_conditional_bridge} gives
\begin{equation}
	\E_{j-1}\{
	\Gamma_{i,j}H_v(Y_{i,j,n}^{\circ})^2\}
	\le Ch_nR_{j-1}.
	\label{eq:supp_fo_signed_second_moment}
\end{equation}
Indeed, \(H_v^2=G_v\), and the conditional expectation of \(G_v\) is
bounded by \(Ch_nR_{j-1}\) uniformly on the grid.

Let
\[
	S_{i,n}(v):=\frac1{T_{i,n}^{\circ}}
	\sum_{j=1}^n\Gamma_{i,j}H_v(Y_{i,j,n}^{\circ}).
\]
The predictable part is \(O_p(1)\) by
\eqref{eq:supp_fo_signed_one_step},
part~(i) of Lemma~\ref{lem:supp_moment_bounds}, and
\eqref{main-eq:Tin_convergence}. By
\eqref{eq:supp_fo_signed_second_moment}, the conditional variance of the
centered part is \(O_p(T_n^{-1})\). Consequently,
\begin{equation}
	\max_{i\in S}\max_{v\in\mathcal V_n}|S_{i,n}(v)|=O_p(1).
	\label{eq:supp_fo_truncated_first_sum}
\end{equation}

\proofstep{Step 3: empirical boundary strip}
Lemma~\ref{lem:iv_boundary_probability} gives, for every deterministic
\(0<r\le u_n/2\), uniformly over \(i\) and the threshold grid,
\[
	\max_{i\in S}\max_{v\in\mathcal V_n}B_{i,n}(v,r)
	=O_p\!\left(ru_n^{-1-\beta}
	+\sqrt{\frac{ru_n^{-1-\beta}}{T_n}}\right)+O_p(h_n^K)
\]
for every fixed \(K>0\).

\proofstep{Step 4: replacement of the tail-mean correction}
Set
\[
	e_{i,n}:=\widehat\kappa_{i,n}-\kappa_{i,n},
	\qquad \delta_{i,n}:=h_ne_{i,n}.
\]
Proposition~\ref{prop:tail_functional} implies
\(\max_i|e_{i,n}|=O_p(T_n^{-1/2})\). Given \(M>0\), define
\[
	\mathcal E_{n,M}:=\{\max_i|e_{i,n}|\le MT_n^{-1/2}\},
	\qquad r_{n,M}:=Mh_nT_n^{-1/2}.
\]
Then
\(\lim_{M\to\infty}\liminf_{n\to\infty}
\Pb(\mathcal E_{n,M})=1\), and \(r_{n,M}=o(u_n)\).
For \(|d|\le v/2\), Lemma~\ref{lem:supp_truncation_perturbation}
gives
\begin{equation}
	|G_v(y+d)-G_v(y)-2dy\mathbf1_{\{|y|\le v\}}|
	\le Cd^2+Cv^2
	\mathbf1_{\{\left||y|-v\right|\le|d|\}}.
	\label{eq:supp_fo_G_shift}
\end{equation}

On \(\Omega_n^{\rm ns}\cap\mathcal E_{n,M}\), equations
\eqref{eq:supp_fo_exact_selector_exposure} and
\eqref{eq:supp_fo_G_shift} give, uniformly in \(i,v\),
\begin{align}
	|R_{i,n}^{(0)}(v)|
	\le{}&2|\delta_{i,n}|\,|S_{i,n}(v)|
	+\frac{n\delta_{i,n}^2}{T_{i,n}^{\circ}}
	+Cv^2B_{i,n}(v,r_{n,M}).
	\label{eq:supp_fo_raw_decomposition}
\end{align}
This is the sample-level localization for the random width
\(h_n|\widehat\kappa_{i,n}-\kappa_{i,n}|\): the deterministic-width
boundary estimate is used only on \(\mathcal E_{n,M}\), where the width
is bounded by \(r_{n,M}=o(u_n)\), and no boundary estimate is invoked on
its complement.
By \eqref{eq:supp_fo_truncated_first_sum}, the first term is
\(O_p(h_nT_n^{-1/2})\), and the second is \(O_p(n^{-1})\). Applying
\eqref{eq:supp_fo_empirical_boundary} with \(r=r_{n,M}\) to the last
term yields
\[
	O_p\!\left(
	\frac{h_nu_n^{1-\beta}}{\sqrt{T_n}}
	+\frac{h_n^{1/2}u_n^{(3-\beta)/2}}{T_n^{3/4}}
	\right).
\]
First let \(n\to\infty\) and then \(M\to\infty\). Together with
\eqref{eq:supp_fo_no_leave_return}, this proves
\begin{equation}
	\max_{i\in S}\max_{v\in\mathcal V_n}|R_{i,n}^{(0)}(v)|
	=O_p(\varepsilon_{0,n}^{\rm fo}).
	\label{eq:supp_fo_raw_rate}
\end{equation}
Every displayed first or second threshold difference is a fixed linear
combination of at most three values in \(\mathcal V_n\); hence
\eqref{eq:supp_fo_raw_rate} proves the asserted level, first-difference,
and second-difference bounds.

\proofstep{Step 5: oracle grid expansion}
Sum the conditional expansion in
Lemma~\ref{lem:iv_conditional_bridge} and separate the common
Brownian-square martingale. This gives
\begin{equation}
	C_{i,n}^{\circ}(v)
	=a_i^2(\gamma^\star)+\frac{\mathcal Z_{i,n}^{\circ}}{\sqrt n}
	+d_{i,1}v^{2-\beta}+d_{i,2}h_nv^{-\beta}
	+E_{i,n}^{\circ}(v).
	\label{eq:supp_fo_oracle_grid_expansion}
\end{equation}
The summable conditional-mean remainder is \(o_p(n^{-1/2})\).
To verify all stochastic remainders explicitly, put
\[
	p_{i,j}:=\E_{j-1}[\Gamma_{i,j}],
	\qquad \overline T_{i,n}^{\circ}:=h_n\sum_{j=1}^np_{i,j}.
\]
Since
\(T_{i,n}^{\circ}-\overline T_{i,n}^{\circ}
=h_n\sum_j(\Gamma_{i,j}-p_{i,j})\) and
\(\operatorname{Var}_{j-1}(\Gamma_{i,j})\le Ch_nR_{j-1}\),
conditional isometry gives
\begin{equation}
	\max_{i\in S}
	\left|\frac{\overline T_{i,n}^{\circ}}{T_{i,n}^{\circ}}-1\right|
	=O_p\!\left(\frac{h_n}{\sqrt{T_n}}\right).
	\label{eq:supp_fo_predictable_exposure}
\end{equation}
Multiplication by either
\(v^{2-\beta}\) or \(h_nv^{-\beta}\) therefore
produces an \(o_p(n^{-1/2})\) term uniformly over the grid.

For the centered Brownian-replacement remainder, let
\(\mathcal O_n:=\{\min_iT_{i,n}^{\circ}\ge
(\min_i\pi_i)T_n/2\}\). Then \(\Pb(\mathcal O_n)\to1\). The variance
bound in Lemma~\ref{lem:iv_conditional_bridge}, conditional isometry,
and localization of \(n^{-1}\sum_jR_{j-1}\) give
\begin{align}
	&\mathbf1_{\mathcal O_n}\left|
	\frac{\sqrt n}{T_{i,n}^{\circ}}
	\sum_{j=1}^n\left[
	\mathcal B_{i,j,n}(v)-\E_{j-1}\mathcal B_{i,j,n}(v)
	\right]\right|^2\\
	&\qquad=O_p\!\left(
	v^{2-\beta}+\frac{v^{4-\beta}}h_n+h_n\right)=o_p(1),
	\label{eq:supp_fo_oracle_level_variance}
\end{align}
because \(u_n^{4-\beta}/h_n\to0\). For \(v<w\) on the grid, the annular
estimate in Lemma~\ref{lem:iv_conditional_bridge} gives
\begin{equation}
	\frac1{T_{i,n}^{\circ}}\sum_{j=1}^n
	\left[\Gamma_{i,j}\{G_w(Y_{i,j,n}^{\circ})
	-G_v(Y_{i,j,n}^{\circ})\}
	-\E_{j-1}\!\left\{\Gamma_{i,j}(G_w-G_v)
	(Y_{i,j,n}^{\circ})\right\}\right]
	=O_p\!\left(\frac{u_n^{(4-\beta)/2}}{\sqrt{T_n}}\right)
	=o_p(u_n^{2-\beta}).
	\label{eq:supp_fo_oracle_annulus_rate}
\end{equation}
The bound in \eqref{eq:supp_fo_oracle_annulus_rate} applies to every
second difference by the triangle inequality. Combining
\eqref{eq:supp_fo_predictable_exposure}--%
\eqref{eq:supp_fo_oracle_annulus_rate} with the summability assertion of
Lemma~\ref{lem:iv_conditional_bridge} yields
\begin{equation}
	\max_{i,v}|E_{i,n}^{\circ}(v)|=o_p(n^{-1/2}),
	\qquad
	\max_{i,a,v}
	\{|\Delta_{\zeta_a}E_{i,n}^{\circ}(v)|
	+|\Delta_{\zeta_a}^2E_{i,n}^{\circ}(v)|\}
	=O_p(s_n^{\circ}),
	\label{eq:supp_fo_oracle_remainders}
\end{equation}
where the second maximum is over admissible grid triples and
\begin{equation}
	s_n^{\circ}:=\rho_n
	+\frac{u_n^{(4-\beta)/2}}{\sqrt{T_n}}
	+\frac{h_n}{\sqrt{T_n}}
	\{u_n^{2-\beta}+h_nu_n^{-\beta}\}.
	\label{eq:supp_fo_oracle_difference_scale}
\end{equation}
The last term in \eqref{eq:supp_fo_oracle_difference_scale} is
\(O(\varepsilon_{0,n}^{\rm fo})\), because
\(u_n^{2-\beta}+h_nu_n^{-\beta}=o(1)\). Hence
\begin{equation}
	s_n^{\circ}=O(s_n^{\rm iv}).
	\label{eq:supp_fo_oracle_scale_absorption}
\end{equation}

\proofstep{Step 6: feasible expansion and rate verification}
By the definitions of \(E_{i,n}^{\rm fea}\) and \(R_{i,n}^{(0)}\),
\[
	E_{i,n}^{\rm fea}(v)=E_{i,n}^{\circ}(v)+R_{i,n}^{(0)}(v).
\]
Equations \eqref{eq:supp_fo_raw_rate},
\eqref{eq:supp_fo_oracle_remainders}, and
\eqref{eq:supp_fo_oracle_scale_absorption} therefore give
\[
	\max_{i,v}|E_{i,n}^{\rm fea}(v)|
	=o_p(n^{-1/2}),
	\qquad
	\max_{i,v,w}|E_{i,n}^{\rm fea}(v)-E_{i,n}^{\rm fea}(w)|
	=O_p(s_n^{\rm iv}).
\]
For the second bound, any two points of the fixed \(3\times3\) grid
\(\mathcal V_n/u_n\) are joined by at most four first-difference steps,
so the first-difference bounds in
\eqref{eq:supp_fo_raw_rate} and
\eqref{eq:supp_fo_oracle_remainders} control every pairwise difference.
This proves \eqref{eq:supp_feasible_curve_expansion}--%
\eqref{eq:supp_feasible_curve_remainders}.

Using \(n=T_n/h_n\),
\[
	\sqrt n\,\varepsilon_{0,n}^{\rm fo}
	\lesssim\sqrt{h_n}+n^{-1/2}
	+\sqrt{h_n}\,u_n^{1-\beta}
	+T_n^{-1/4}u_n^{(3-\beta)/2}\to0.
\]
Furthermore,
\begin{align*}
	\frac{\varepsilon_{0,n}^{\rm fo}}{u_n^{2-\beta}}
	\lesssim{}&
	\frac{h_n}{\sqrt{T_n}u_n^{2-\beta}}
	+\frac{h_n}{T_nu_n^{2-\beta}}
	+\frac{h_n}{\sqrt{T_n}u_n}
	+\frac{h_n^{1/2}u_n^{(\beta-1)/2}}{T_n^{3/4}},\\
	\frac{\varepsilon_{0,n}^{\rm fo}}{h_nu_n^{-\beta}}
	\lesssim{}&
	\frac{u_n^\beta}{\sqrt{T_n}}
	+\frac{u_n^\beta}{T_n}
	+\frac{u_n}{\sqrt{T_n}}
	+\frac{u_n^{(3+\beta)/2}}{h_n^{1/2}T_n^{3/4}}.
\end{align*}
Every term converges to zero by
Table~\ref{tab:supp_iv_rate_audit}. For the last term, the lower bound
\(\omega>2/(4+\beta)\) implies \(\omega(3+\beta)>1\).

Finally,
\begin{align*}
	\sqrt n\,s_n^{\rm iv}
	&\lesssim\sqrt n\,\rho_n
	+\left(\frac{u_n^{4-\beta}}{h_n}\right)^{1/2}
	+\sqrt n\,\varepsilon_{0,n}^{\rm fo}\to0,\\
	\frac{s_n^{\rm iv}}{u_n^{2-\beta}}
	&\lesssim\frac{\rho_n}{u_n^{2-\beta}}
	+\frac{u_n^{\beta/2}}{\sqrt{T_n}}
	+\frac{\varepsilon_{0,n}^{\rm fo}}{u_n^{2-\beta}}\to0,\\
	\frac{s_n^{\rm iv}}{h_nu_n^{-\beta}}
	&\lesssim\frac{\rho_n}{h_nu_n^{-\beta}}
	+\frac{u_n^{(4+\beta)/2}}{h_n\sqrt{T_n}}
	+\frac{\varepsilon_{0,n}^{\rm fo}}{h_nu_n^{-\beta}}\to0,
\end{align*}
again by Table~\ref{tab:supp_iv_rate_audit}. This proves
\eqref{eq:supp_feasible_curve_rates} and the proposition.
\end{proof}

Fix \(i\in S\). The concrete quantities are matched to the abstract
notation of Lemma~\ref{lem:iv_debias_stability} by
\begin{equation}
\begin{gathered}
	F_n(v)=\widehat C_{i,n}(v),\qquad
	A_n=a_i^2(\gamma^\star)+\frac{\mathcal Z_{i,n}^{\circ}}{\sqrt n},
	\qquad (b,c)=(d_{i,1},d_{i,2}),\\
	p=2-\beta,\qquad q=-\beta,\qquad
	E_n(v)=E_{i,n}^{\rm fea}(v),\qquad s_n=s_n^{\rm iv}.
\end{gathered}
	\label{eq:debias_feasible_substitution}
\end{equation}
Proposition~\ref{prop:supp_iv_feasible_oracle}, specifically
\eqref{eq:supp_feasible_curve_expansion}--%
\eqref{eq:supp_feasible_curve_rates}, verifies the expansion, the two
remainder bounds, and the three basic rate conditions for this
substitution. Table~\ref{tab:debias_feasible_verification} records the
remaining deterministic checks.
Here \(s_{1,n}\) is defined by \eqref{eq:debias_remainder_scales} with
\(s_n=s_n^{\rm iv}\).
\begin{table}[H]
\centering
\scriptsize
\renewcommand{\arraystretch}{1.25}
\begin{adjustbox}{max width=\textwidth}
\begin{tabular}{
	>{\raggedright\arraybackslash}p{4.0cm}
	>{\raggedright\arraybackslash}p{9.7cm}
	>{\raggedright\arraybackslash}p{2.9cm}}
\toprule
Abstract hypothesis & IV verification & Verification source \\
\midrule
\(p>0\), \(q<0\)
& \(p=2-\beta>0\) and \(q=-\beta<0\).
& \greenrev{Standing IV specification in the main paper} \\
\(b,c\ne0\)
& \(b=d_{i,1}=2c_i/(2-\beta)>0\) and
  \(c=d_{i,2}=-c_i(\beta+1)(\beta+2)
  a_i^2(\gamma^\star)/\beta\ne0\).
& \eqref{eq:supp_iv_bias_coefficients} \\
\(\zeta_1,\zeta_2>1\), deterministic \(s_n>0\)
& The multipliers are fixed by Assumption~\ref{ass:iv_sampling}, and
  \(s_n^{\rm iv}\) in \eqref{eq:debias_feasible_substitution} is a positive
  deterministic sequence.
& Assumption~\ref{ass:iv_sampling};
  \eqref{eq:debias_feasible_substitution} \\
Expansion \eqref{eq:debias_abstract_expansion}
& \(\widehat C_{i,n}(v)=A_n+d_{i,1}v^{2-\beta}
  +d_{i,2}h_nv^{-\beta}+E_n(v)\).
& \eqref{eq:supp_feasible_curve_expansion} \\
\(\max_v|E_n(v)|=o_p(n^{-1/2})\)
& The feasible remainder is uniformly \(o_p(n^{-1/2})\) on the
  finite grid.
& \eqref{eq:supp_feasible_curve_remainders} \\
\(\max_{v,w}|E_n(v)-E_n(w)|=O_p(s_n)\)
& The feasible remainder differences are uniformly
  \(O_p(s_n^{\rm iv})\) on the finite grid.
& \eqref{eq:supp_feasible_curve_remainders} \\
\(s_n=o(u_n^p)\)
& \(s_n^{\rm iv}/u_n^{2-\beta}\to0\).
& \eqref{eq:supp_feasible_curve_rates} \\
\(h_nu_n^q=o(u_n^p)\)
& \(h_nu_n^q/u_n^p=h_n/u_n^2
  =O(h_n^{1-2\omega})\to0\).
& Table~\ref{tab:supp_iv_rate_audit} \\
\(h_n=o(u_n^p)\)
& \(h_n/u_n^p=O\{h_n^{1-\omega(2-\beta)}\}\to0\).
& Table~\ref{tab:supp_iv_rate_audit} \\
\(h_n=o(h_nu_n^q)\)
& \(h_n/(h_nu_n^q)=u_n^\beta\to0\).
& \(\beta>1\), \(u_n\to0\) \\
\(\sqrt n\,s_{1,n}\to0\)
& \(\sqrt n\,s_n^{\rm iv}\to0\),
  \(\sqrt n\,h_n^2u_n^{-\beta-2}\to0\), and the two relative rates
  imply this through
  \eqref{eq:debias_remainder_scales}.
& \eqref{eq:supp_feasible_curve_rates},
  \eqref{eq:supp_iv_rho_definition}, and
  Table~\ref{tab:supp_iv_rate_audit} \\
\(s_{1,n}/(h_nu_n^q)\to0\)
& \(s_n^{\rm iv}/(h_nu_n^{-\beta})\to0\),
  \(h_n/u_n^2\to0\), and
  \(s_n^{\rm iv}/u_n^{2-\beta}\to0\) imply this.
& \eqref{eq:supp_feasible_curve_rates} and
  Table~\ref{tab:supp_iv_rate_audit} \\
\bottomrule
\end{tabular}
\end{adjustbox}
\caption{Verification of the abstract two-step stability hypotheses for
the feasible IV variation curve.}
\label{tab:debias_feasible_verification}
\end{table}

For completeness, the two composite rate implications used in the last
two rows follow from
\begin{align*}
	\sqrt n\,s_{1,n}
	&\lesssim \sqrt n\,s_n^{\rm iv}
	+\sqrt n\,
	\frac{(h_nu_n^q)^2+(s_n^{\rm iv})^2}{u_n^p}=o(1),\\
	\frac{s_{1,n}}{h_nu_n^q}
	&\lesssim \frac{s_n^{\rm iv}}{h_nu_n^q}
	+\frac{h_nu_n^q}{u_n^p}
	+\frac{(s_n^{\rm iv})^2}{h_nu_n^{p+q}}=o(1).
\end{align*}
Indeed, because \(p=2-\beta\) and \(q=-\beta\),
\[
	\sqrt n\,\frac{(h_nu_n^q)^2}{u_n^p}
	=\sqrt n\,h_n^2u_n^{-\beta-2}\to0
\]
by \eqref{eq:supp_iv_rho_definition} and
Table~\ref{tab:supp_iv_rate_audit}. The other quadratic terms factor as
\[
	\sqrt n\,\frac{(s_n^{\rm iv})^2}{u_n^p}
	=(\sqrt n\,s_n^{\rm iv})
	\frac{s_n^{\rm iv}}{u_n^p},
	\qquad
	\frac{(s_n^{\rm iv})^2}{h_nu_n^{p+q}}
	=\frac{s_n^{\rm iv}}{u_n^p}
	\frac{s_n^{\rm iv}}{h_nu_n^q},
\]
and \((h_nu_n^q)/u_n^p=h_n/u_n^2\to0\). Hence every hypothesis of
Lemma~\ref{lem:iv_debias_stability} has now been verified for the
substitution \eqref{eq:debias_feasible_substitution}, uniformly over the
finite regime set.

Under this substitution,
\eqref{eq:debias_first_denominator},
\eqref{eq:debias_second_denominator}, and the stabilization conclusion of
Lemma~\ref{lem:iv_debias_stability} give, uniformly over \(i\in S\),
\begin{align*}
	\Delta_{\zeta_1}^2\widehat C_{i,n}(v)
	&=d_{i,1}(\zeta_1^{2-\beta}-1)^2
	v^{2-\beta}\{1+o_p(1)\},
	\qquad v\in\mathcal W_n,\\
	\Delta_{\zeta_2}^2\widehat C_{i,n}^{(1)}(u_n)
	&=\lambda_{p,q}(\zeta_1)d_{i,2}
	(\zeta_2^{-\beta}-1)^2h_nu_n^{-\beta}\{1+o_p(1)\}.
\end{align*}
Moreover,
\[
	\Pb\!\left(
	\min\left\{
	\min_{i\in S}\min_{v\in\mathcal W_n}
	|\Delta_{\zeta_1}^2\widehat C_{i,n}(v)|,
	\min_{i\in S}|\Delta_{\zeta_2}^2
	\widehat C_{i,n}^{(1)}(u_n)|\right\}\ge h_n
	\right)\to1.
\]
These are the denominator and stabilization assertions of Appendix
Lemma~\ref{lem:iv_debiasing_main}.

\begin{lem}[IV debiased variation]
	\label{lem:iv_debiased_variation}
	Under the \greenrev{standing IV specification of the main paper} and
	Assumptions~\ref{ass:smooth}, \ref{ass:levy},
	\ref{ass:iv_sampling}, and \ref{ass:ergodic}, there are
	\((\mathcal Z_{i,n}:i\in S)\) such that
	\[
		\widehat a_{i,n}^{\,2,\rm db}
		=a_i^2(\gamma^\star)+\frac{\mathcal Z_{i,n}}{\sqrt n}
		+o_p(n^{-1/2}),
	\]
	uniformly over \(i\in S\), where
	\[
		\mathcal Z_{i,n}
		=\frac1{\sqrt n\,\pi_i}\sum_{j=1}^n
		\Gamma_{i,j}a_i^2(\gamma^\star)
		\left\{\frac{(\Delta_jW)^2}{h_n}-1\right\}+o_p(1).
	\]
	Moreover,
	\[
		(\mathcal Z_{i,n}:i\in S)^\top
		\cil N\!\left(0,\operatorname{diag}\left(
		\frac{2a_i^4(\gamma^\star)}{\pi_i}:i\in S\right)\right).
	\]
\end{lem}

\begin{proof}[Proof of Lemma~\ref{lem:iv_debiased_variation}]
This is the final probabilistic step in Appendix
Lemma~\ref{lem:iv_debiasing_main}. The abstract stability lemma
reduces the estimator to the centered Brownian quadratic term, after
which the exposure law and a vector martingale central limit theorem
identify its joint regime-wise limit.

Apply Lemma~\ref{lem:iv_debias_stability} under the substitution
\eqref{eq:debias_feasible_substitution}. Its second-map expansion
\eqref{eq:debias_second_expansion} and
Table~\ref{tab:debias_feasible_verification} give, uniformly over the
finite regime set,
\[
	\widehat a_{i,n}^{\,2,\rm db}
	=a_i^2(\gamma^\star)
	+\frac{\mathcal Z_{i,n}^{\circ}}{\sqrt n}
	+o_p(n^{-1/2}).
\]

It remains to identify the joint law. Put
\[
	Q_{i,n}:=\frac1{\sqrt n}\sum_{j=1}^n
	\Gamma_{i,j}a_i^2(\gamma^\star)
	\left\{\frac{(\Delta_jW)^2}{h_n}-1\right\}.
\]
Then
\[
	\mathcal Z_{i,n}^{\circ}
	=\frac{T_n}{T_{i,n}^{\circ}}Q_{i,n}.
\]
The vector summands defining \(Q_n=(Q_{i,n}:i\in S)\) are martingale
differences. If
\(p_{i,j}:=\E_{j-1}\Gamma_{i,j}\), independence of \(W\) and \(Z\)
and the Gaussian fourth moment give
\[
	\sum_{j=1}^n\E_{j-1}
	\left[\frac{\Gamma_{i,j}a_i^2}{\sqrt n}
	\left\{\frac{(\Delta_jW)^2}{h_n}-1\right\}
	\frac{\Gamma_{k,j}a_k^2}{\sqrt n}
	\left\{\frac{(\Delta_jW)^2}{h_n}-1\right\}\right]
	=\frac{2a_i^4}{n}\mathbf1_{\{i=k\}}\sum_{j=1}^np_{i,j}.
\]
Here and below \(a_i=a_i(\gamma^\star)\). Since
\[
	p_{i,j}=\mathbf1_{\{Z_{j-1}=i\}}\{1+O(h_n)\},
	\qquad
	\frac1n\sum_{j=1}^np_{i,j}\cip\pi_i,
\]
	the predictable covariance of \(Q_n\) converges to
	\(\operatorname{diag}(2\pi_i a_i^4:i\in S)\). The corresponding sum of
	conditional fourth moments is \(O_p(n^{-1})\), so the vector Lindeberg
	condition holds. In particular, \(Q_{i,n}=O_p(1)\). Since
	\(T_{i,n}^{\circ}/T_n\cip\pi_i\),
	\[
		\mathcal Z_{i,n}^{\circ}
		=\frac{T_n}{T_{i,n}^{\circ}}Q_{i,n}
		=\frac1{\pi_i}Q_{i,n}+o_p(1)
		=\frac1{\sqrt n\,\pi_i}\sum_{j=1}^n
		\Gamma_{i,j}a_i^2(\gamma^\star)
		\left\{\frac{(\Delta_jW)^2}{h_n}-1\right\}+o_p(1).
	\]
	The martingale central limit theorem and Slutsky's theorem yield
	\[
		(\mathcal Z_{i,n}^{\circ}:i\in S)^\top
	\cil N\!\left(0,\operatorname{diag}
	\left(\frac{2a_i^4(\gamma^\star)}{\pi_i}:i\in S\right)\right).
\]
Thus the assertion holds with
\(\mathcal Z_{i,n}=\mathcal Z_{i,n}^{\circ}\).
\end{proof}

\end{document}